\documentclass{amsart}
\usepackage[dvips]{color}
\usepackage{graphicx}
\usepackage{epsfig}
\usepackage{tikz}
\usepackage{enumerate}
\usepackage{enumitem}
\usepackage{color}

\usepackage{latexsym}
\usepackage{amssymb}
\usepackage{amsmath}
\usepackage{amsthm}
\usepackage{amscd}
\usepackage{latexsym}
\usepackage{amssymb}
\usepackage{amsmath}
\usepackage{amsthm}
\usepackage{amscd}
\theoremstyle{plain}
\newtheorem{thm}{Theorem}[section]
\newtheorem{lem}[thm]{Lemma}
\newtheorem{defin}[thm]{Defintion}
\newtheorem{prop}[thm]{Proposition}
\theoremstyle{definition}
\newtheorem{cor}[thm]{Corollary}

\newtheorem{ques}[thm]{Question}
\newcommand{\spanD}{\text{span}_{\Delta}}

\newcommand{\future}[1]{{}}  

\usepackage{amsmath,amsfonts,amssymb,amsthm,epsfig,float, color,tikz}
\usepackage[margin=1.5in]{geometry}
\newcommand{\bb}{\mathbb}

\newcommand{\om}{\omega}

\newcommand{\Hdim}{\operatorname{Hdim}}
\newcommand{\nue}{\operatorname{NUE}}

\newcommand{\hh}{\mathcal H}

\newcommand{\constFL}{(2k+k_0)^6-(k+k_0)^4+(k+k_0)^{2.3}}
\newcommand{\radlower}{10^{-(k+k_0)^2}}

\title{The set of non-uniquely ergodic $d$-IETs has Hausdorff codimension 1/2}

\author{Jon Chaika}
\author{Howard Masur}

\email{chaika@math.utah.edu}
\email{masur@math.uchicago.edu}

\address{Department of Mathematics, University of Utah, 155 S 1400 E Room 233, Salt Lake City, UT 84112}
\address{Department of Mathematics, University of Chicago, 5734 S. University Avenue, Room 208C, Chicago, IL 60637, USA}

\begin{document}
\maketitle
\begin{center} \textit{Dedicated to the memories  of William Veech and Jean-Christophe Yoccoz}
\end{center}
\tableofcontents
\begin{abstract}
We show that the set of not uniquely ergodic $d$-IETs has Hausdorff dimension $d-\frac 3 2 $ (in the $(d-1)$-dimension space of $d$-IETs) for $d\geq 5$. For $d=4$ this was shown by Athreya-Chaika and for $d\in\{2,3\}$ the set is known to have dimension $d-2$. 
\end{abstract}

\section{Introduction}

\begin{defin}  Given $x=(x_1,x_2,\ldots,x_d)\in \bf{R}^d$ where $x_i > 0$, form the  $d$ sub-intervals of the
interval $[0,\sum_i x_i)$: $$I_1=[0,x_1) ,
I_2=[x_1,x_1+x_2),\ldots,I_d=[x_1+\ldots x_{d-1}, x_1+\ldots+x_{d-1}+x_d).$$ Given
 a permutation $\pi$ on  the set $\{1,2,\ldots,d\}$, we obtain a d-\emph{Interval Exchange Transformation} (IET)  $ T \colon [0,\underset{i=1}{\overset{d}{\sum}} x_i) \to
 [0,\underset{i=1}{\overset{d}{\sum}} x_i)$ which exchanges the intervals $I_i$ according to $\pi$. That is, if $x \in I_j$ then $$T(x)= x - \underset{k<j}{\sum} x_k +\underset{\pi(k')<\pi(j)}{\sum} x_{k'}.$$
 \end{defin}

Lebesgue measure is invariant under the action of $T$.
\begin{defin} $T$ is uniquely ergodic if up to scalar multiple Lebesgue is the only invariant measure
\end{defin}

The purpose of this paper is to prove the following Theorem. Let $\Delta\subset \bf{R}^d$  the standard simplex of dimension $d-1$.   Let  $\pi_s$ be the hyperelliptic permutation on 
 $d\geq 4$  letters (defined below) and let 
  $\mathcal{R}_d$ be the Rauzy class of $\pi_s$. Let $\pi\in\mathcal{R}_d$. 
Let $\nue(\pi)$ denote the set of $x\in\Delta$ such that $T(x)$ is not uniquely ergodic. 
\begin{thm}\label{thm:main}
For $\pi\in \mathcal{R}_d$
 the Hausdorff dimension of $\nue(\pi)$ is $d-\frac{3}{2}$.
\end{thm}
\noindent
 Note that the space is $(d-1)$-dimensional and so the Hausdorff codimension of this set is $\frac 1 2 $. 
 It is easy to show that the set of IET that are  not minimal (orbits are not dense) has Hausdorff codimension $1$. So the main theorem says that for $d\geq 4$ the minimal non-uniquely  ergodic IET have smaller codimension.
 In the case $d=2$ the classical Kronecker-Weyl Theorem says that 
 $T$ is not uniquely ergodic if and only if $x_1,x_2\in\mathbb{Q}$. (Here $x_1+x_2=1$). Since $3$-IETs are first return maps of 2-IETs to intervals, minimality implies unique ergodicity for 3-IETs as well.

 The upper bound $HDim(\nue)\leq d-\frac{3}{2}$ follows from Masur \cite{Mhdim}. (See Section 6 of \cite{ath-chai}). In the case of $d=4$ there is only one Rauzy class and Athreya-Chaika \cite{ath-chai} proved the Theorem in that case. This paper is devoted to the proof of the lower bound  for the given permutations $\pi$. 

 \begin{thm} For $g\geq 2$, let $\mathcal{H}_{hyp}(2g-2)$ and $\mathcal{H}_{hyp}(g-1,g-1)$ be the hyperelliptic components of the  strata $\mathcal{H}(2g-2)$ and $\mathcal{H}(g-1,g-1)$.  Then the set  of $(X,\omega)$ in these strata such that the vertical flow is not uniquely ergodic has Hausdorff codimension $1/2$.
 \end{thm}

 Given a translation surface $(X,\omega)$ denote by $\nue(X,\omega)$ the set of directions $\theta\in [0,2\pi)$ such that the vertical  flow  of $e^{i\theta} \omega$
 is not uniquely ergodic.

 \begin{thm}
 \label{ thm: most:surfaces}
 For almost every $(X,\omega)\in  \mathcal{H}_{hyp}(2g-2)$ or $\mathcal{H}_{hyp}(g-1,g-1)$ 
 we have $HDim(\nue(X,\omega))=\frac{1}{2}$
 \end{thm}
 
 This Theorem follows from Theorem~\ref{thm:main}  with the exactly same proof as in the paper 
\cite{ath-chai}  in the case of $\hh(2)$.

 It is worth noting that for $d>4$ there are Rauzy classes other than $\pi_s$ so an  open question is if the Theorem holds for all classes.

\begin{ques}What is the Hausdorff dimension of the set of non-weakly mixing IETs.
\end{ques}
 Avila and Leguil proved that it has positive Hausdorff codimension \cite{AM}. Our result shows its Hausdorff codimension is at most $\frac 1 2$. 
 Boshernitzan and Nogueira \cite{BN} showed that for interval exchanges with \emph{type W} permutations having that the abelian differentials define Teichm\"uller geodesics that are recurrent in the stratum implies that the interval exchange is weak mixing.
  (Note this is a result about the IET, not the vertical flow on the corresponding surface.) Al-Saqban, Apisa, Erchenko, Khalil, Mirzadeh and Uyanik \cite{AAEKMU} showed that the Hausdorff codimension of Teichm\"uller geodesics which are not recurrent in its stratum is at least one half, strengthening a result of \cite{Mhdim}. (They prove more, treating the larger set of trajectories that are \emph{divergent on average}.) Since when $d$ is odd the permutation $(d,d-1,....,2,1)$ is type W, combining our results with \cite{BN} and \cite{AAEKMU} one obtains that the set of non-weakly mixing IETs with permutation $(d,d-1,...,2,1)$ when $d$ is odd has Hausdorff codimension $\frac 12$. It is natural to wonder if this bound is uniform in all strata. 

In the construction in this paper  the non uniquely ergodic IET have exactly $2$ ergodic 
measures. One can therefore ask
\color{black}

\begin{ques}What is the Hausdorff dimension of $d$-IETs with $1<k\leq \frac{d}2$ ergodic measures for each $k$?
\end{ques}

\subsection{History}

Constructions of minimal non-uniquely ergodic IETs are due to Katok-Stepin \cite{KS}, Sataev \cite{Sataev}, Keane~\cite{Keane2} and Keynes-Newton~\cite{KeynesNewton}, based on examples of  Veech~\cite{Veech69}. 
Masur \cite{Masur1} and Veech \cite{Veech} independently proved the Keane conjecture that with respect to Lebesgue measure on $\Delta$, and any irreducible permutation $\pi$,  almost every IET is uniquely ergodic. 
There is a strong connection 
 between the theory of IETs and {\em translation surfaces}.
A genus $g$ \emph{translation surface} $(X, \omega)$ is a compact, genus $g$ Riemann surface together with a nonzero holomorphic one-form $\omega$. This gives the structure of a flat metric away from a finite number of singular points, as integrating the one-form $\omega$ gives charts (away from zeros of $\omega$) to $\bf C$ where the transition functions between charts are translations. The zeros of $\om$ are singular points of the metric, and have cone angles $2\pi(k+1)$ at a zero of order $k$.  Kerckhoff-Masur-Smillie \cite{KMS} showed that the Lebesgue measure of $\nue(X,\omega)$ is $0$. Masur~\cite{Mhdim} showed that $HDim\nue(X,\omega)\leq \frac{1}{2}$.


Moduli spaces of translation surfaces are stratified by their genus $g$ and the combinatorics of their singularities. We say a singularity has order $k$ if the angle is $2\pi(k+1)$. The sum of orders of singularities on a genus $g$ surface is $2g-2$. Given a partition $\alpha = (\alpha_1, \ldots, \alpha_m) \in \bb{N}^m$, $\sum \alpha_i = 2g-2$, define the stratum $\hh = \hh(\alpha)$ to be the moduli space of (unit-area) translation surfaces with singularity pattern $\alpha$. On each stratum $\hh$, there are coordinate charts to an appropriate Euclidean space, and pulling back Lebesgue measure yields a natural Lebesgue measure $\mu_{MV}$ on $\hh$.
For each translation surface, there is a countable set of directions where the flow is not minimal (that is, there are non-dense infinite trajectories).

Masur-Smillie \cite{MS} showed that for    every stratum of translation surfaces $\hh(\alpha)$ of surfaces of genus at least 2 there is a constant $c = c(\alpha)>0$ such that for $\mu_{MV}$-almost every flat surface $(X,\omega) \in \hh$, $$\Hdim(\nue(X,\omega)) = c.$$  In terms of the simplex $\Delta$ of IET, this implies the codimension of non uniquely ergodic IETs is less than 1.
The  result in \cite{Mhdim} referenced above said that $c(\alpha) \le 1/2$ for all $\alpha$.

Theorem~\ref{ thm: most:surfaces} says that for almost every surface the set of non-uniquely ergodic directions has Hausdorff dimension  exactly $\frac{1}{2}$. On the other hand for so-called lattice or Veech surfaces  the dimension is $0$. 
 It is likely that Boshernitzan's argument in the appendix of \cite{yit thesis} can be applied to strata and show that a residual set of surfaces have this property. 
 
This leads to a natural question:
\begin{ques} Is there a translation surface where the the set of non-uniquely ergodic directions has Hausdorff dimension $c \notin \{0,\frac 1 2 \}$?
\end{ques}

\textbf{Acknowledgements}
H. Masur was supported in part by NSF grant DMS-1607512.
J. Chaika thanks Jayadev Athreya, with whom this project began, for many helpful ideas and conversations. J. Chaika was supported in part by NSF grants DMS-135500 and DMS-1452762, the Sloan foundation, a Warnock chair and a Poincar\'e chair. 

William Veech and Jean-Christophe Yoccoz tragically died within a week of each other in the summer of 2016. They were visionaries who introduced and developed many of the ideas which allowed the field to grow. This paper is built on their work and perspective. We dedicate it to their memory.

\section{Plan of Paper and background material and notation \color{black}}\label{sec:background}

Our plan is to  build a subset of $\Delta$ consisting of  minimal and non-uniquely ergodic IETs that has Hausdorff codimension $\frac 1 2 $. 
The proof  (and indeed many results on ergodic properties of IETs) uses in a crucial fashion the Rauzy induction renormalization procedures for IETs, involving induced maps on certain subintervals, and closely related to the Teichm\"uller geodesic flow. Our treatment of Rauzy induction will be the same as in \cite{MMY}. For further details of the procedure (and much more on IETs) we refer the interested reader, to, e.g.~\cite{Yoccoz}, for an excellent survey. 
 
The idea in using  Rauzy induction to build non-uniquely ergodic IET is that one produces sequences $M_k=A_1A_2\ldots A_k$  of nonnegative matrices from Rauzy induction such   that under the (projective) action the nested sequence of simplices  $M_k\Delta$ converges to a positive dimensional simplex.  The limiting points are non-uniquely ergodic IET.   
The desired set of codimension $\frac{1}{2}$ is a limit of a procedure where we produce `good' descendants $M_k\Delta$ of ancestor simplices $M_{k-1}\Delta$.  We do this by showing that the intersection of the sequence of these sets with a positive measure set of a family of parallel planes has Hausdorff dimension $\frac 3 2 $. 

We describe our paths and matrix sizes in Section 3 and indicate some properties these matrices need to satisfy.  
Our paths break up into stages. 
Each stage essentially breaks into parts where we do Rauzy induction on $d-2$ intervals which we call the left hand side (LHS) and a Rauzy induction on $2$ intervals; the right hand side (RHS).
Each side itself breaks up into a freedom part and a restricted part. In freedom on the left Rauzy induction is essentially arbitrary in that any of the first $d-2$ intervals can win, and when one competes with one of the last two intervals it always wins.   During restriction the first $d-2$ intervals  never compete with the last two intervals,  and when the first interval competes with the other $d-3$ it always loses. This cuts down the measure so that infinitely many iterations, the result has measure $0$ as it must. However we keep enough simplices for the desired Hausdorff dimension.   We analogously  have freedom and restriction on the right.  

The very beginning will be restriction on the left and end with restriction on the right. The second stage begins then with freedom on the left and ends with restriction on the right.  In general, stage $k>1$ begins with freedom on LHS and ends with restriction on RHS.

The proof of Hausdorff dimension  has two parts: the abstract geometric framework, which is stated and proved to be sufficient in Section 4, and the much more involved part showing that the sets we build  satisfy the abstract geometric framework. We now outline some of the basic issues with this latter  argument. At its heart the argument is probabilistic. For this we use probabilistic results of Rauzy induction due to Kerckhoff and Veech and elementary probability theory to prove large deviations results (Section 5). However we will consistently need to use these probabilistic results to discuss the typical behavior of points lying in codimension 2, 3 and dimension 2 `faces' of a simplex (Sections 6, 10 and 11) and indicate how these results on subsimplices apply to the entire simplex.  Section 7 sets up the geometry of subsimplices intersecting planes, and is essentially a section on linear algebra.

Sections 8-11 are the most technically difficult part of the paper. The goal in Sections 8 and 9 is Theorem~\ref{thm:big shadow}  which says that during each stage in freedom on LHS, if we throw out a small set of points, the intersection of our planes with the remaining simplices intersect a fixed designated face.  
 During restriction,  we lose lots of measure, but we need to control this to prove a lower bound for the Hausdorff dimension. We do this by showing that even when we lose measure, we keep most of the measure in a small neighborhood of a (codimension 1) face of our simplex. This is done in Section 10 (for the left hand side).  On the right hand side we have a similar, if not easier picture since there are only $2$ intervals.  This is  done in section 11. In section 12 we put these estimates together to prove that the abstract setting of section 4 holds.

\subsection{Preliminaries    on   Rauzy induction}\label{sec:rauzy}

We follow the description of interval exchange transformations  introduced in \cite{MMY} and also explicated in \cite{AGY}.  We have the set  $\mathcal{A}$ which consists of the first $d$ positive integers. 
Break an interval $I=[0,x)$ into intervals 
$\{I_i; 1\leq i\leq d\}$
and rearrange in a new order by translations. 
 Thus the interval
exchange transformation is entirely defined by the following data:
\begin{enumerate}
\item The lengths of the intervals
\item Their orders before and after rearranging
\end{enumerate}

The first are called length data, and are given by a vector
$x\in \mathbb{R}^d$ .
The second are called combinatorial data, and are given by a pair  of bijections
$\pi=(\pi_t,\pi_b)$ from
$\mathcal{A}\to\mathcal{A}$. 

 The bijections 
can be viewed as a pair of rows, the top corresponding to
$\pi_t$
and the bottom corresponding to
$\pi_b$. 

 Given an interval exchange   $T$ defined by  $(x,\pi)$ let $i,j\in \mathcal{A}$ the last elements in the top and bottom.  The operation of  Rauzy induction is applied  when $x_i\neq x_j$ to give a new IET  $T'$ defined by $(x',\pi')$ where $x',\pi'$ are as follows.   If $x_i>x_j$  then $\pi'$  keeps the top row unchanged, and it changes the bottom row by moving $j$ to the position immediately to the right of the position occupied by $i$.     We say $i$ {\em  wins} and $j$ {\em loses}. For all $k\neq i$ define   $x'_k=x_k$ and define $$x'_i=x_i-x_j.$$   

If  $x_j>x_i$ then to define $\pi'$ we keep the bottom row the same and the top row is changed by moving $i$ to the position to the right of the position occupied by $j$.  Then $x'_k=x_k$ for all $k\neq j$ and $x'_j=x_j-x_i$.  We say $j$ wins and $i$ loses.  

In  either case one has a  new interval exchange $T'$ determined by  $(x',\pi')$
and defined on an interval $I'=[0,|x'|)$ where $$|x'|=\sum_i x'_i.$$

The map $T':I'\to I'$ is the first return map 
to a subinterval of
$I$
obtained by
cutting from
$I$
a subinterval with the same right endpoint and of length
$x_k$
where
$k$
is the loser of
the  process described above.

Let $\Delta$ be the standard simplex in $\mathbb{R}^d$ and let $\mathcal{P}$ be the set of permuations on $n$ letters.  We can normalize so that all IET are defined on the unit interval.  
Let  $${R}:\Delta\times \mathcal{P}\to \Delta\times \mathcal{P}$$ by $R(x,\pi)=(\frac{x'}{|x'|},\pi')$  denote (renormalized) Rauzy induction.

The set of permutations on $d$ letters form a collection of connected directed  graphs. We have a   directed edge joining $\pi$ to $\pi'$ if one of the two possibilities for Rauzy induction at $\pi$ yields $\pi'$. We say they are in the same \emph{Rauzy class} if they are in the same connected component.

 There is a corresponding visitation matrix  $M=M(T)$.  Let $\{e_i\}$ be the standard basis.  If $i$ is the winner and $j$ the loser, then $M(e_k)=e_k$ for $k\neq j$ and $M(e_j)=e_i+e_j$. We can view $M$ as simply arising  from the identity matrix by adding the $i$ column to the $j$ column. 
We can projectivize the matrix $M$ and consider it as $M:\Delta\to \Delta$.

When the interval exchange $T$ is understood, and we perform Rauzy induction $n$ times then define $M(1)=M(T)$ and inductively $$M(n)=M(n-1)M(R^{n-1}T).$$
That is, the matrix $M(n)$ comes from multiplying $M(n-1)$ on the right by the matrix of Rauzy induction applied to the IET after we have done Rauzy $n-1$ times. 
We will also use the following notation.
A vector  $x\in \Delta$ and permutation $\pi$ determines an IET $T$.  The corresponding matrix after performing Rauzy induction $n$ times and suppressing the permutations is denoted $M(x,n)$.
For $y\in M(x,n)\Delta$ denote by  $R^ny$ the $x'\in\Delta$ such that $M(x,n)x'=y$. 

\noindent Observe that if $x, \eta$ satisfy  $ \eta \in M(x,k){\Delta}$, then $$M(\eta, k) = M(x,k).$$ That is, the IETs determined by $x$ and $\eta $ have the same first $k$ steps of Rauzy induction.

  \noindent Given a matrix $M$, we write $$M{\Delta}=M\mathbb{R}_d^+ \cap\Delta=\left\{\frac{Mv}{|Mv|}:v\in \Delta\right\}.$$ 
\noindent

We also have the following notation.

\begin{itemize}

\item If $x^1,\ldots,x^j \in \mathbb{R}^d$, let $\spanD(x^1,\ldots ,x^j)=\{\sum a_i x^i:a_i\geq 0 \text{ and }\sum a_i x^i \in \Delta\}$. In a mild abuse of notation we sometimes  put subsets of $\mathbb{R}^d$ in the argument as well. 
\item For $M$ any matrix of Rauzy induction let $C_{\max}(M)$ be the column  $C_j(M)$ that maximizes $|C_j(M)|$ over $1\leq j\leq d$. If there are two or such columns choose the one with the smallest index. Similarly with $C_{\min}(M)$.

\item $\lambda_s$ refers to Lebesgue measure on a $s$ dimensional simplex. We will use this for $s\in\{1,2,d-4,d-3, d-2,d-1\}$.
\item For $0\leq c\leq 1$ let $\Delta_c=\{x\in \Delta:x_{d-1}+x_d=c\}$. \color{black}
\item If $v, w \in \mathbb{R}^k$ let $\Theta(v,w)$ denote the angle between $v$ and $w$. In a mild abuse of notation, if $V,W \subset\mathbb{R}^k$ then $\Theta(V,W)=\min\{\Theta(v,w):v\in V,\, w \in W\}.$
\item If $M$ is a matrix let $\|M\|$ denote the $L^1$ operator norm of $M$. If all the entries of $M$ are non-negative, this is $|C_{\max}(M)|$. 
\item $V(M)=\spanD(C_i(M_1), \ldots C_{d-2}(M))$  
\item $W(M)=\spanD(C_{d-1}(M),C_d(M))$.
\item $F_i(M)$ the $i^{th}$ face is the span of all but $C_i(M)$. 
\item If $\mathcal{M}$ is a set of matrices so that $\mathcal{M}\Delta$ is a simplex with labeled extreme points $p_1,...,p_d$, let $F_i(\mathcal{M})$ be the convex hull of $\{p_\ell\}_{\ell \neq i}$. In this setting we let $V(\mathcal{M})=\spanD(p_1,...,p_{d-2})$ and $W(\mathcal{M})=\spanD(p_{d-1},p_d)$. 
\color{black}
\end{itemize}

A note on constants.  On  numerous  occasions we will make use of a constant $C$ in upper bounds. It is a local constant in that it will not depend on matrices or $k$.  Similarly we will use $c>0$ for lower bounds.  We will also frequently have a constant $\rho$ 
that appears in  probablistic statements.

\section{Paths and matrices}\label{sec:paths}
Again let 
 $$\pi_s=\begin{pmatrix}1 & \dots &d\\
d&  \dots&1 
\end{pmatrix}$$
be the hyperelliptic permutation and $\mathcal{R}_d$ be its   Rauzy class.
Our set of non-uniquely ergodic IETs is obtained by producing large families of special paths in the graph. Our paths break into segments with five different types.
\begin{itemize}
\item Freedom on the left hand side.
\item Restriction on the left hand side.
\item Transition from the left side to the right side
\item Freedom on the right hand side.
\item Restriction on the right hand side.

\end{itemize}
\noindent
We now describe these types. 

$$\pi_L=\begin{pmatrix}1& d-1 & d& 2 & 3&  \dots &d-2\\
d & d-1 & d-2 &d-3 &\dots &  &1
\end{pmatrix}$$

$$\pi_R=\begin{pmatrix} 1 & 2  & \dots & & d-1 &d \\
d & d-2 & d-3 & \dots & 1 & d-1.
\end{pmatrix}.$$

Let us make some observations:
We reach $\pi_L$ from $\pi_s$ in two steps with $1$ beating $d$ and then $d-1$.  
  We are now on the left hand side (LHS) at $\pi_L$. 
 Starting at $\pi_L$ suppose we have a permutation in which   $1$ wins, and after an arbitrary sequence of permutations, (including possibly $1$ losing, but $d-1,d$ always losing whenever they are compared) eventually $1$ beats $d-3,\ldots, 2$ and  we reach $\pi_s$ before $d$ and $d-1$ are compared again. We call  such a path a path of  {\em freedom on the left hand side}.

On the other hand  starting at $\pi_L$,  for as long as $1$ loses when matched with $i;  1<i\leq d-2$, then  $d$ and $d-1$ will not be compared to anything. So they will neither be added to a column nor have a column added to them. We call this {\em restriction}  on LHS since $1$ is losing. The first row of the corresponding matrix is $(1,0,\ldots 0)$ and the first column reflects that other columns are added to the first. Also the graph formed from anything that can be reached from $\pi_L$ without $1$ winning is a copy of  $\mathcal{R}_{d-3}$ where the symbols are $2,...,d-2$. \color{black}  Note there is an extra vertex  at $\pi_L$ \color{black} where $1$ is compared to $d-2$ (and loses). In the set we are describing this vertex has exactly one incoming edge and one outgoing edge and so we ``collapse" this vertex and these two edges to obtain $\mathcal{R}_{d-3}$.
We return to $\pi_L$.

 Now we take any path from   $\pi_L$ to $\pi_s$ without returning to $\pi_L$ and where $\pi_s$ is reached only at the end and call this the  {\em transition from left side to right side}.

Suppose starting at $\pi_s$, $d$ successively beats $1,\ldots, d-2$ to reach $\pi_R$. We are now on the right hand side RHS performing Rauzy induction.  Then starting at $\pi_R$ there is a loop which consists of an arbitrary sequence  of  $d-1$ beating  $d$ followed by $d$ beating $d-1$, followed  by $d$ beating $1,\ldots, d-2$ returning to $\pi_R$. We can repeat this loop an arbitrary number of times.  This is {\em freedom}  on RHS.
On the other hand suppose  $d$ loses at $\pi_R$ a number of times before beating $d-1$.  Then  the permutation returns to $\pi_s$. We call this   {\em restriction} on RHS.   As long as $d$ keeps losing to $d-1$ the letters  $1,...,d-2$ are not compared.

\subsection{Choice of matrices}

 The letter $A$ denotes matrices for Rauzy induction on the LHS and $B$ denotes the matrices on the RHS. The matrices $T$ correspond to transition from left to right.
  We add $'$ for matrices $A,B$ during  restriction and no prime denotes freedom.  We will return infinitely often to each side and to freedom and restriction on each side. 
 
If we are at freedom on LHS via a path with corresponding matrix $M(x,r)$, 
  let $A(R^r(y),m)$ the matrix of Rauzy induction done $m$ times at $R^ry$.

Fix $k_0$ to be determined later. (We will have a (finite) collection of conditions on $k_0$ but they will all hold for all $k_0$ large enough.) \color{black}  For the  $k^{th}$ return and $k\geq 1$, we will build matrices  $A_k,A_k',B_k,B_k',T_k$.
Our matrices will  be products of these matrices and will be denoted with letters $M,\hat M$.

We will  start with $A_1'$ and end with any of $A_k,A_k',T_k,B_k,B_k'$. For example after ending 
after freedom on LHS ($k\geq 2$) we will have the matrix  $M=A_1'T_1B_1B_1'A_2A_2'\ldots A_{k-1}A_{k-1}'T_{k-1}B_{k-1}B_{k-1}'A_k$.

  We choose  $$10^{(3+k_0)^6} \leq \|A'_1\|\leq 2\cdot 10^{(3+k_0)^6}.$$
For $k\geq 2$  $$\|A_k\|\in [10^{(2k+k_0)^6-(k+k_0)^4},
10^{\constFL\color{black}}]$$ 
$$\|A_k'\|\in [10^{(2k+1+k_0)^6+(k+k_0)^4}, 10^{(2k+1+k_0)^6+(k+k_0)^4+(k+k_0)^2} ],$$ 
For $k\geq 1$, $$\|B_k\|\in [10^{(2k+1+k_0)^6-(k+k_0)^4 },10^{(2k+1+k_0)^6-(k+k_0)^4+(k+k_0)^2}],$$
 $$\|B_k'\|\in [10^{(2k+2+k_0)^6+(k+k_0)^4},2 \cdot 10^{(2k+2+k_0)^6+(k+k_0)^4}].$$
$$\|T_k\|\leq 10^{(k+k_0)^2}.$$
\color{black}
We will impose  the following conditions on these matrices, and in the course of the paper prove that there exists $\rho<1$ so that  they can be satisfied  at each stage $k$ except for \color{black} a set of  proportion $\rho^{(k+k_0)^2}$  in each simplex.  This will be sufficient for our purposes.

\noindent

\textbf{Conditions *} There exists $\zeta$ so that
\begin{enumerate}
\item $\frac{|C_{i}(A_k)|}{|C_{i'}(A_k)|}<\zeta$  for all $i,\, i'\leq d-2$ and $k$. 
\item $\frac{|C_i(A'_kT_k)|}{|C_{i'}(A'_kT_k)|}<10^{2(k+k_0)^{2\color{black}}}$ 
for all $i, \, i'\leq d-2$ and  $k$.
\item $\frac{|C_{j}(B_k)|}{|C_{j'}(B_k)|}\leq  \zeta $ for all $j,j'\geq d-1$ and $k$.
\item $\frac{|C_{j}(B'_k)|}{|C_{j'}(B'_k)|}\leq 2$ for all $j,j'\geq d-1$ and $k$.
\end{enumerate}
\noindent

The last condition is automatic.   The third condition follows from the bounds we  will  put on the matrices in $\mathcal{B}_k$ (see Section \ref{sec:rhs}).  The first condition will be  established in Corollary~\ref{cor:ready} condition 3 (up to renaming $\zeta'$ as $\zeta$). 
The second condition follows by combining Theorem \ref{thm:neighborhood} and  
Lemma~\ref{lem:second}.

\noindent

\textbf{Condition **} There exists $c>0$ so that:
\begin{enumerate} 
\item if $M$ is a matrix at the end of freedom on LHS at stage $k$ 
\begin{equation}
\Theta(C_i(M),C_{i'}(M)<10^{-c(2k+k_0)^6}\end{equation}  for $i,i'\leq d-2$ and 
\item if $M$ is a matrix at the end of restriction on RHS at stage $k$
\begin{equation}
\Theta(C_{d-1}(M),C_d(M))<10^{-(2k+1+k_0)^6}.
\end{equation}
\end{enumerate}
This is proven in Lemma \ref{lem:angle to sing}.

Now we wish to find bounds on the  size of columns of products  of  matrices.
Let  $$U_k=\max_{i\leq d-2} |C_i(A'_1.....A_k)|$$  $$u_k=\min_{i\leq d-2}|C_i(A'_1...A'_k)|$$ $$V_k=\max_{i>d-2}|C_i(A'_1...B_k)|$$  $$v_k=\min_{i> d-2}|C_i(A'_1...B'_k)|.$$

Note $U_k$ is the maximum size of the first $d-2$ columns after freedom on LHS,   $u_k$ the minimum size after restriction on LHS and so forth.

\begin{prop}
\label{prop:sizes} Under Conditions *, if $k_0$ is large enough we have
\begin{enumerate}
\item \begin{multline*} \frac{1}{\zeta^{2(k-2)}}10^{-(k+k_0)^4}10^{-\sum_{i=1}^{k-1}4(i+k_0)^{2}}
10^{\sum_{i=3}^{2k}(i+k_0)^6}\leq U_k\\ \leq 2^k 10^{-(k+k_0)^4}10^{\sum_{i=3}^{2k} (i+k_0)^6+  2\sum_{i=2}^k(i+k_0)^2+\sum_{i=2}^k(i+k_0)^{2.3}}
\end{multline*}
\item  
\begin{multline*}\zeta^{-2(k-1)}10^{-\sum_{i=1}^{k-1}4(i+k_0)^{2}}10^{\sum_{i=3}^{2k+1}(i+k_0)^6}\leq    u_k \\
\leq 2^k 10^{\sum_{i=3}^{2k+1} (i+k_0)^6+2\sum_{i=2}^k(i+k_0)^2+\sum_{i=1}^k(i+k_0)^{2.3}}
\end{multline*}
   \item $ \frac 1 {(2\zeta)^k} 10^{-(k+k_0)^4}10^{\sum_{i=3}^{2k+1}(i+k_0)^6}\leq V_k\leq 2^k 10^{-(k+k_0)^4}10^{\sum_{i=3}^{2k+1}(i+k_0)^6+2\sum_{i=1}^k(i+k_0)^2}$
\item   $\frac 1 {(2\zeta)^k}10^{\sum_{i=3}^{2k+2}(i+k_0)^6} 
 \leq v_k \leq 2^k10^{\sum_{i=3}^{2k+2}(i+k_0)^6+2\sum_{i=1}^k(i+k_0)^2}$   

\end{enumerate}

\end{prop}
 
 \color{black}
 \begin{proof}
We will find the upper and lower bounds for $U_k$. The proofs of the other inequalities are similar. 	 It is straightforward to check that our conclusions are satisfied for $u_1,U_1,v_1$ and $V_1$. We now prove the claim by induction, assuming the claim on $u_{k-1},U_{k-1},v_{k-1}$ and $V_{k-1}$ and then establish it for $U_k$. We claim first that \begin{equation}
 \label{eq:upper}
 U_k\leq \left(U_{k-1}10^{(2(k-1)+1+k_0)^6+(k-1+k_0)^4+2(k-1+k_0)^2}+V_{k-1}\right)10^{(2k+k_0)^6-(k+k_0)^4+(k+k_0)^{2.3\color{black}}}.
 \end{equation}  
 \begin{equation}
\label{eq:lower}
 U_k\geq \left({10^{-2(k+k_0)^{2}}}U_{k-1}10^{(2(k-1)+1+k_0)^6+(k-1+k_0)^4}+\frac{V_{k-1}}{\zeta}\right)10^{-2(k+k_0)^{2}}10^{(2k+k_0)^6-(k+k_0)^4}.
\end{equation}
 
 To justify these estimates note first  that going from the end of freedom at stage $k-1$ to the end of freedom at stage $k$ we first have restriction on the left followed by transition from left to right.
 During restriction  we add a column $C_i$ to a column $C_j$ where $i,j\leq d-2$.  
 The total will increase the   size of the first $d-2$ columns by  at most $\|A_{k-1}'\|$. The upper bound on $\|A_{k-1}'\|$  and  the upper bound on $\|T_{k-1}\|$ give \color{black}
the first term in the parentheses.  
The fact that the first $d-2$ columns of these \color{black} matrices are  $10^{2(k+k_0)^{2}}$ balanced (by Condition * (2))  means that each column is increased by a multiplicative factor  which is the lower bound of $\|A_{k-1}\|$ divided by  $10^{2(k+k_0)^{2}}$. This gives the first term in the lower bound, \eqref{eq:lower}. 

Then we enter freedom on RHS.   Now  the first $d-2$ columns are changed by adding the last two columns to the first $d-2$. Thus the effect of freedom on RHS is that we add  at most $V_{k-1}$ and at least $\frac{V_{k-1}}{\zeta}$.  The first $d-2$ columns are not changed during restriction on RHS.   
Then finally at level $k$ we have freedom on LHS and for an upper bound we multiply by an upper bound for $\|A_k\|$ and for a lower bound  we multiply by a lower bound for  $10^{-2(k+k_0)^{2}}\|A_k\|$ (because the previous matrix had that the first $d-2$ columns were $10^{(2k+k_0)^{2}}$ balanced). \color{black}   This proves (\ref{eq:upper}) and the corresponding lower bound (\ref{eq:lower}).

 Notice during restriction on LHS the last two columns do not change,
and so \color{black} in going from $U_{k-1}$ to $V_{k-1}$ we have a similar analysis to find
 $$U_{k-1}10^{(2(k-1)+1+k_0)^6}<V_{k-1}<U_{k-1}10^{(2(k-1)+1+k_0)^6+\color{black}(k-1+k_0)^2}$$ and so plugging this into (\ref{eq:upper})  and (\ref{eq:lower}) and using 
 \color{black}
 the induction hypothesis for the upper and lower bounds for $U_{k-1}$ to get the desired  bounds for $U_k$.


\end{proof}
 \color{black}
 \subsection{Non-unique ergodicity}
 \begin{thm}
 \label{thm:nue}
Under Conditions *, for all $\epsilon>0$, for $k_0$ large enough,   then for all $k$ we have  
$$\Theta\left(C_i(M_k),span (e_1,\ldots, e_{d-2}) \right)\leq  \epsilon $$  
for all $i\leq d-2$ and  $$\Theta\left(C_j(M_k),span (e_{d-1}, e_d)\right)<\epsilon\color{black}$$ for $j=d-1,d$.  
\end{thm}
\begin{proof}
Notice that during freedom and restriction on LHS the only columns added to the first $d-2$ columns are themselves. This implies that $\spanD(C_1(M),\ldots, C_{d-2}(M))$ is a subset of what it was at the start of these phases. So  $\underset{i\leq d-2}{\max} \Theta(C_i(M),span (e_1,\ldots  e_{d-2})$ can only decrease when we are on LHS.  Similarly, $\underset{j\geq d-1}{\max}\Theta(C_j(M),e_{d-1}\oplus e_d)$  can only decrease on RHS.  The theorem therefore will follow by estimating how much the angles of the first $d-2$ columns can change during RHS, and how much the angles of the last two columns can change on LHS. \color{black}

During freedom on the RHS we add a vector of norm at most $V_k$ to a vector of norm at least $u_k$. Proposition \ref{prop:sizes} says that  
$$V_k<\zeta^{2(k-1)}2^{2k}10^{-[(k_0+k)^4-2\sum_{i=2}^k (i+k_0)^2+4(i+k_0)^2]}u_k$$
  and so if $k_0$ is large enough we have $\frac{V_k}{u_k}<10^{-(k+k_0)^3}$. 
Similarly, during freedom on the LHS we add a vector of norm at most $U_k$ to a vector of norm at least $v_{k-1}$. We have $\frac{U_k}{v_{k-1}}<C10^{-(k+k_0)^3}$ for a constant $C$. These are the only times vectors in $C_1,...,C_{d-2}$ are added to $C_{d-1}$ and $C_d$ and vice-versa. 
Notice we started with $A_1'$ so that the first $d-2$ columns are not added to the $d-1$ and $d$ columns. Thus after $A_1'$ by taking $k_0$ sufficiently large the first $d-2$ colums are arbitrarily close to the span of $e_1,\ldots e_{d-2}$ and the last two columns are exactly $e_{d-1}$ and $e_d$.  Similarly considering $B_1$ and $B_1'$, for any $\epsilon>0$, by choosing  $k_0$ large enough we can ensure that $\sum_{i=1}^\infty \frac{V_i}{u_i}<\epsilon$ and $\sum_{i=2}^\infty \frac{U_i}{v_{i-1}}<\epsilon$.

\end{proof}

\section{Hausdorff dimension}\label{sec:hdim}

In our construction 
we will have  a parallel family $\mathcal P$ of $2$ planes,   parametrized by points in a $d-3$ dimensional orthogonal  subspace  intersected with $\Delta$\color{black}.  Using Lebesgue measure on the orthogonal complement gives a measure on the set of $2$-planes.

 The majority of the paper will be devoted to proving the following theorem.

\begin{thm}
\label{thm:planes} There exists 
\begin{itemize}
\item a positive measure set $\hat{\mathcal{P}}$  of parallel $2$-planes $P$, 
\item  for each $k\in \mathbb{N}$   a set  $\mathcal{C}_k$ of disjoint  simplices,  \color{black}
 $\Delta_k^j\subset \Delta$  and 
 \item for each  $P\in \hat{\mathcal{P}}$ and $k \in \mathbb{N}$, a set $\mathcal{C}_k(P) \subset \{\Delta_k^j\cap P: \Delta_k^j \in \mathcal{C}_k\}$ of polygons
 \end{itemize} 
   so that for each $P\in\hat{\mathcal{P}}$, and  $J\in\mathcal{C}_k(P) $, 
when  
  $r_k(J)$  is the diameter of this polygon,  we have 
 \begin{enumerate}[label=(\alph*)]
\item\label{A:nest} each polygon in $\mathcal{C}_{k+1}(P)$  is a subset of some element of $\mathcal{C}_{k}(P)$
 and is called a descendant.
\item\label{A:nue} each point  in an infinite nested sequence of polygons is not uniquely ergodic 
\item\label{A:close size} if we set $\bar r_k=\max_{P\in\hat{\mathcal{P}}} \max_{J\in \mathcal{C}_k(P)} r_k(J)$ and $\hat r_k=\min_{P\in\hat{\mathcal{P}}} \min_{J\in \mathcal{C}_k(P)} r_k(J)$, \color{black} then for each $\epsilon>0$ we have $\lim_{k\to\infty}\frac{\bar r_k^{1+\epsilon}}{\hat r_{k+1}}= 0$.
 \item \label{A:small loss} There exists $a_k$ so that letting $\mathcal{D}_{k+1}(J)$ be the set of all descendants $J'$ at level $k+1$ of a polygon $J$ at level $k$  then 
 except for a set of polygons  $\mathcal{B}_k(P)\subset\mathcal{C}_k(P)$ satisfying  
$$\sum_{J\in\mathcal{B}_k(P)}\lambda_2(J)< \frac {1}{9^{k}}\sum_{J\in\mathcal{C}_k(P)} \lambda_2(J)
$$ 
 we have 
 $$\sum_{J'\in\mathcal{D}_{k+1}(J)}\lambda_2(J')>a_k\lambda_2(J)$$ 
 \item \label{countchildren}The $a_k$ satisfy that for all $\epsilon>0$, 
 $\underset{k \to \infty}{\lim} \hat{r}_k^{\frac 1 2 +\epsilon}(\prod_{j=1}^k a_j)^{-1}=0.$  
  \item \label{A:sep} For each $P\in \hat{\mathcal{P}}$,  $J\in \mathcal{C}_k(P)$ and $J'\in\mathcal{C}_{k+2}(P)$, where $J'\subset J$,  then $\mathcal{N}(J',\bar r_{k+2}) \cap P\subset J.$ \color{black}
 \end{enumerate}
 \end{thm}

Assuming Theorem~\ref{thm:planes}  we show how  Theorem \ref{thm:main} follows. 
We first recall Frostman's Lemma and prove a useful Corollary.

\begin{lem}(Frostman) Let $A\subset \mathbb{R}^k$ be Borel. The following are equivalent: 
\begin{itemize}
\item $\mathcal{H}^s(A)>0$ where $\mathcal{H}^s$ denotes $s-$dimensional Hausdorff measure.
\item  There exists a Borel measure $\mu$  satisfying $\mu(A)>0$ and $\mu(B(x,r))\leq r^s$ for all $x \in \mathbb{R}$ and $r>0$.
\end{itemize}
\end{lem}

We will use the following.
\begin{cor} 
\label{cor:Frostman}
Suppose $A \subset \Delta_{d-1} \subset \mathbb{R}^d$  is Borel  and there exists a Borel measure $\mu$ so that $\mu(A)>0$ and for all $\epsilon>0$ there exists $r_0$ so that for all $0<r<r_0$ and $x\in \mathbb{R}^d$ one has $\mu(B(x,r))<r^{s-\epsilon}$.  Then $Hdim(A)\geq s$. 
\end{cor}
\begin{proof}First observe that if $\mathcal{H}^t(A)>0$ for all $0 \leq t<s$ then $Hdim(A)\geq s$. Now observe that if there exists a measure $\nu$ with $\nu(A)>0$ and there exists $r_0$ so that $\nu(B(x,r))<r^{s-\epsilon}$ for all $0<r<r_0$ and $x\in \mathbb{R}^d$ then there exists a measure $\mu$ so that $\mu(A)>0$ and $\mu(B(x,r))<r^{s-\epsilon}$ for all $r$. Indeed, choose $r<r_0$ and $x$ so that $\nu(B(x,r)\cap A)>0$ and choose $\mu$ to be $\nu$ restricted to $B(x,r)$. So it is clear that $\mathcal{H}^{s-\epsilon}(A)>0$ for all $\epsilon>0$ and by our previous observation we have the corollary.
\end{proof}

\begin{prop} \label{prop:good planes} Each $2$-plane $P$ occurring in  Theorem~\ref{thm:planes} satisfies 
$H_{dim}(\cap_{k=1}^{\infty}\mathcal{C}_k(P))\geq \frac 3 2 $. 
\end{prop}

\begin{proof}
We will  put a non-zero Borel measure $\mu$ on $\cap_{k=1}^{\infty}\mathcal{C}_k(P)$ with the property that for all $\epsilon>0$ there exists $r_0>0$ so that 
$\mu(B(x,r))< r^{\frac 3 2 -\epsilon}$  for all $r<r_0$. By Corollary~\ref{cor:Frostman} this will prove the proposition. 

The  measure $\mu$ will be  the weak-$*$ limit of measures defined inductively. Let $\mu_1$ be Lebesgue measure restricted to $\mathcal{C}_1(P)$. Given $\mu_k$ defined on $\mathcal{C}_k(P)$, we inductively define $\mu_{k+1} $ in the following way.
Define $\mu_{k+1}$ so that 
\begin{itemize}
\item $\mu_{k+1}$ is supported on the set of $J\in \mathcal{C}_k( P)$ that satisfy \ref{A:small loss}. 
\item $\frac{\mu_{k+1}(J'_1)}{\mu_{k+1}(J'_2)}=\frac{\lambda_2(J'_1)}{\lambda_2(J_2')}$ for all $J_1',J_2'\in \mathcal{C}_{k+1}(P)$, with $J'_1,J'_2\subset J$. 
\item  if $\mu_{k+1}(J)>0$ and $J\in \mathcal{C}_k(P)$ then $\mu_{k+1}(J)=\mu_{k}(J)$,
\item $\mu_{k+1}$ is a multiple of Lebesgue on each $J'\in \mathcal{C}_{k+1}(P)$ (the constant can depend on its immediate ancestor) and 
\item $\mu_{k+1}(J \setminus \mathcal{C}_{k+1})=0.$  
\end{itemize}
(Informally, we rescale $\lambda_2$ restricted to $\mathcal{C}_{k+1}(P)\cap J$ so that $\mu_{k+1}(J)=\mu_k(J)$.)
\color{black}
Let $\mu_{k+1}$ be the zero measure everywhere else. Let $\mu$ be a weak-* limit of the $\mu_k$.
Now the second bullet implies $\frac{\mu_{k+1}(J')}{\lambda_2(J')}$ is independent of the descendents $J'$ of $J$.   The third bullet says if we sum up $\mu_{k+1}(J')$ over these descendents we get $\mu_{k+1}(J)=\mu_k(J)$. Combined with Conclusion \ref{A:small loss} of Theorem~\ref{thm:planes} 
we get \begin{equation}
\label{eq:constant}
\frac{\mu(J')}{\lambda_2(J')}\leq \frac{\mu_k(J)}{a_k\lambda_2(J)}
\end{equation}
 for all  $J'$ descendants  of 
$J\in \mathcal{C}_{k+1}(P)$. 
We claim  
\begin{equation}\label{eq:more than zero}\mu(\cap_{i=1}^\infty \mathcal{C}_i(P))>0.
\end{equation}

By Conclusion \ref{A:small loss} of Theorem~\ref{thm:planes} we have 
$$\mu_{k+1}(\cap_{i=1}^{k+1} \mathcal{C}_i(P))\geq (1-\frac{1}{9^k})\mu_k(\cap_{i=1}^{k+1} \mathcal{C}_i(P)).$$ It follows that 
$$\mu(\cap_{i=1}^\infty \mathcal{C}_i(P))\geq \underset{k \to \infty}{\liminf} \mu_k(\cap_{i=1}^k \mathcal{C}_i(P))\geq (1-\sum_{j=1}^\infty \frac{1}{9 ^j})\mu_1(P)>0,$$
 establishing Inequality \eqref{eq:more than zero}.

Next we check that for each $k$, if  $r<\frac {\bar r_{k+2}} 2 $, 
then for all $x$,  
\begin{equation}\label{eq:ball meas} \mu(B(x,r))\leq \lambda_2(B(x,r_k))(\prod_{j=1}^ka_j)^{-1}.
\end{equation}
By the construction of the measures we claim that if $J\in \mathcal{C}_k(P)$, then $\mu(J)\leq \lambda_2(J) (\prod_{j=1}^ka_j)^{-1}$. Indeed, by (\ref{eq:constant}) and induction,
$$\mu_k(J)\leq \lambda_2(J)(\prod_{j=1}^ka_j)^{-1}\leq \lambda_2(B(x,\bar{r}_k))(\prod_{j=1}^ka_j)^{-1}$$ and also by construction $\mu_\ell(J)\leq \mu_k(J)$ for any $\ell>k$. 

If $r<\frac {\bar r_{k+2}}2$ and $\mu(B(x,r))>0$ then there exists $J' =\Delta_{k+2}^i \cap P$ so that $B(x,r) \cap J' \neq \emptyset$. Let $J$ be the (unique) element of $\mathcal{C}_k(P)$ so that $J'\subset J$. By Conclusion  \ref{A:sep} of Theorem \ref{thm:planes} we have that $B(x,r) \subset \mathcal{N}(J',r) \subset J$ and so $\mu(B(x,r))\leq \mu(J)$, establishing Inequality \eqref{eq:ball meas}.

We now finish the proof of Proposition~\ref{prop:good planes} by showing that 
for all $\epsilon>0$ sufficiently small there exists $r'$ so that $\mu(B(x,r))<r^{\frac 3 2 -\epsilon}$ for all $x$ and $0<r<r'$. 

By Conclusion \ref{countchildren}  of Theorem ~\ref{thm:planes}, for all $0<\epsilon<\frac 1 2 $ there exists $k'$ so that for all $k>k'$  $$(\prod_{i=1}^ka_i)>\hat r_k^{\frac 1 2 +\frac \epsilon 4}.$$  Given $r>0$ let $$k=\max\{i:\frac{\hat r_{i+2}}2>r\}.$$  By Conclusion \ref{A:close size}   of Theorem~\ref{thm:planes},  we have that for all $\epsilon>0$ sufficiently small there exists $r'$ so that for $r<r'$ the $k$ defined above satisfies  $$\hat r_k<r^{1-\frac \epsilon 4}.$$  
We can also assume $r'$ small enough so 
  $k>k'$.  Putting this together, we now have our claim: 
$$\mu(B(x,r))<\lambda_2(B(x,r_k))(\prod_{i=1}^ka_i)^{-1}<\pi r_k^2\hat r_k^{-(\frac 1 2 +\frac \epsilon 4)}<r^{(1-\frac{\epsilon}4)(\frac 3 2 -\frac \epsilon 4)}<r^{\frac 3 2 -\epsilon}.$$ We now apply Corollary~\ref{cor:Frostman}. This completes the proof of the Proposition.
\end{proof}
\color{black}
The Main Theorem will now  follow from the following standard result:
\begin{prop}
\label{prop:Matilla}
(\cite[Proposition 6.6]{mattila}) Let $A \subset \mathbb{R}^d$. If $\mathcal{H}^s(A)<\infty$ then for all $m$ planes $V\subset \mathbb{R}^d$,  for  a.e. $a \in V$  $\mathcal{H}^{s-m}(A\cap (V^\perp+a))<\infty$ 
\end{prop}

We now prove the Main Theorem. 
\begin{proof}[ Proof of Theorem \ref{thm:main} assuming Theorem \ref{thm:planes}] We prove the contrapositive. Let $V$ be the orthogonal subspace to $P$.  Assume $$ Hdim(NUE)<d-1-\frac 1 2$$ so there exists $s<d-1-\frac 1 2 $ so that $\mathcal{H}^s((NUE))<\infty$. So by 
Proposition~\ref{prop:Matilla} for almost $\tau\in V$, $$\mathcal{H}^{s-(d-3)}(P+\tau\cap NUE))<\infty.$$ Since $s<d-1-\frac 1 2$   this contradicts Proposition~\ref{prop:good planes}.
\end{proof}

\section{Distortion and Probabilistic results} \label{sec:prob}

In this section we prove that certain bad sets have measure that decay exponentially. 
We start by recalling some known results due to Kerckhoff  \cite{ker} and Veech \cite{Veech78}
\begin{defin}
We say a matrix is $\zeta$-balanced if the ratio of the sizes of any two columns is bounded by $\zeta$.
\end{defin}
    
  The following Lemma says that given any matrix of Rauzy induction, a definite proportion of points with that matrix of Rauzy induction will have a balanced matrix of Rauzy induction before the norm increases by more than a fixed factor. 
  
  \begin{lem}(\cite[Corollary 1.7]{ker}) \label{lem:bal often} There exists $\zeta, K',\rho'>0$ so that if $M=M(x,r)$ is a matrix of Rauzy induction then 
 \begin{multline*}\lambda_{d-1}\Big(\big\{y \in M\Delta: \exists n \text{ so that } MA(R^ry,n) \text{ is 
$\zeta$-balanced and } \\
|C_{\max}(MA(R^ry,n))|\in [|C_{\max}(M)|,K'|C_{\max}(M)|]\big\}\Big)>\rho' \lambda_{d-1}(M\Delta).
\end{multline*}
 \end{lem}
 \begin{lem}(\cite[Corollary 1.2]{ker})\label{lem:bal distort} Let $M(x,r)$ be a $\zeta$-balanced matrix of Rauzy induction. Let $U \subset \Delta$ be measurable. 
 Then $\zeta^d \lambda_{d-1}(U)>\frac{\lambda_{d-1}(MU)}{\lambda_{d-1}(M\Delta)}>\zeta^{-d}
\lambda_{d-1}(U)$.
 \end{lem}
 \begin{lem}(\cite[Proposition 5.2]{Veech78}) \label{lem:jacobian} Given a matrix of Rauzy induction $M(x,n)$ and $W\subset \Delta$ then
 $$\lambda_{d-1}(\{y \in M(x,n)\Delta:R^ny\in W\})=\int_{W}\frac 1 {(\sum_{i=1}^d|C_i(M(x,n))|z_i)^d}dz.$$
 \end{lem}
 
 We also have the following version on faces:
  Let $M$ be a matrix of Rauzy induction in $\mathcal{R}_d$ and
 $J_{i_1,\ldots i_k}$ be the Jacobian of the projective action of $M$ restricted to $\spanD\{e_{i_1},...,e_{i_k})$

\begin{lem}\label{lem:face jacobian}  For  $u\in \spanD(e_{i_1},...,e_{i_k})$  
then $$J_{i_1,\ldots i_k}(M)(u)=\frac{c_M}{(\sum_{j=1}^k u_{i_j}|C_{i_j}(M)|)^k},$$
where $c_M$ is a constant depending on $M$ and $i_1,..,i_k$. 
\end{lem}
We include a proof for completeness. 

\begin{proof} We treat the case of $k=d-1$.  The   general situation follows by repeating the proof below $d-k$ times.
 For simplicity  of notation  assume that $i_\ell=\ell$ so the Jacobian of interest is $J_{1,...,d-1}$ which we denote $J_{d-1}$. 
 
 We have $$J_{d-1}(u)=\frac 1 {(\sum_{i=1}^{d}u_i|C_i(M)|)^d}=\frac 1 {(\sum_{i=1}^{d-1}u_i|C_i(M)|)^d},$$ the Jacobian of $M$ acting on the entire simplex $\Delta$ evaluated at $u$.

 Let $W$ be a small neighborhood of $u \in \spanD(e_1,..,e_{d-1})$ restricted to $\spanD(e_1,...,e_{d-1})$. Let $W_\epsilon=\{v \in \Delta:v=(1-s)w+sv' \text{ with }w\in W, \, v' \in \Delta \text{ and }s\leq \epsilon \}.$

For small $\epsilon$ and small $W$, we  now approximate $\lambda_{d-1}(MW_\epsilon)$ in two different ways.
First, since $W$ is a small neighborhood of $u$ and $\epsilon$ small,   
$$\lambda_{d-1}(MW_\epsilon)\sim J(u)\lambda_{d-1}(W_\epsilon)\sim \epsilon \lambda_{d-2}(W)\frac{1}{(\sum_{i=1}^{d-1}u_i|C_i(M)|)^d}.$$
The notation here $\sim$ is that the ratio goes to  $1$  as $\epsilon$ goes to $0$ and the neighborhood shrinks to $u$.

On the other hand let $\gamma(u)$ be the line segment in $MW_\epsilon$ orthogonal to $MW$.  Then  $$\lambda_{d-1}(MW_\epsilon)\sim  \lambda_{d-2}(MW)\cdot |\gamma(u)|\sim $$
$$ J_{d-1}(u)\lambda_{d-2}(W)\frac{\epsilon|C_d(M)|}{\sum_{i=1}^{d-1}u_i|C_i(M)|}d\Big(\frac{C_{d}(M)}{|C_d(M)|},\spanD(C_1(M),...,C_{d-1}(M)\Big).$$

Taking the ratio of these two expressions for $\lambda_{d-1}(MW_\epsilon)$ and letting $\epsilon\to 0$ and the neighborhood $W$ converging down to $u$,
we get 

$$J_{d-1}(u)\frac{|C_d(M)|}{\sum_{i=1}^{d-1}u_i|C_i(M)|}d\Big(\frac{C_{d}(M)}{|C_d(M)|},\spanD(C_1(M),...,C_{d-1}(M)\Big)=\frac{1}{(\sum_{i=1}^{d-1}u_i|C_i(M)|)^d},$$
which solving for $J_{d-1}$ gives the desired expression  for $J_{d-1}$.
 
\end{proof}
\color{black}

\begin{lem}\label{lem:volume}(\cite[Equation 5.5]{Veech78}) There exists a constant $c_d$ depending only on dimension, so that  $\lambda_{d-1}(M\Delta)=c_d\prod |C_i(M)|^{-1}.$
\end{lem}

We have a version for faces,   with the same proof as the previous result in \cite{Veech78}\color{black}.
\begin{lem}\label{lem:volume face}
Let $M$ be a matrix of $\mathcal{R}_d$ and $A_1,A_2$ be a matrices of freedom on LHS then 
$\frac{\lambda_{d-3}(V(MA_1))}{\lambda_{d-3}(V(MA_2))}=\frac{\prod_{i=1}^{d-2}|C_i(MA)|^{-1}}{\prod_{i=1}^{d-2}|C_i(MA_2)|^{-1}}$. 
\end{lem}

We will prove 
 at the end of the section: 
 \begin{prop}\label{prop:balanced} There exists  $K>1$ and $\sigma<1$ so that for all large enough $\zeta$,  if $M=M(x,r)$ is a matrix of Rauzy induction then for all $m$,
 \begin{multline}\lambda_{d-1}\Big(\big\{y \in M\Delta: \exists n \text{ so that } 
MA(R^ry,n) \text{ is $\zeta$-balanced and } \\
 |C_{\max}(MA(R^ry,n))|\in [|C_{\max}(M)|,K^m|C_{\max}(M)|]\big\}\Big)>(1-\sigma^m) \lambda_{d-1}(M'\Delta)
 \end{multline}
 \end{prop}

In fact we will prove a stronger result:

   \begin{prop}\label{prop:determined}  If $M''=M(w,s)$ is a fixed matrix of Rauzy induction, then there exists $K''>1,\delta''<1$  so that if  $M'=M'(x,r)$   then for all $m$,
 \begin{multline}\lambda_{d-1}\Big(\big\{y \in M'\Delta: \exists n\    \text{so that }\ R^{r}y\in M''\Delta\  \text{with} \\
 |C_{\max}( M'A(R^ry,n))|\in [|C_{\max}(M')|,K''^m|C_{\max}(M')|]\big\}\Big)>(1-\delta''^{m}) \lambda_{d-1}(M'\Delta)
 \end{multline}
 \end{prop}
 This implies the following useful, weaker result: \color{black} 
   \begin{prop}\label{prop:positive} There exists $K',\sigma'$ with $0<\sigma'<1$  so that if $M=M(x,r)$ is a matrix of Rauzy induction, then for all $m$,
 \begin{multline}\lambda_{d-1}\Big(\big\{y\in M(\Delta): \exists n \text{ so that } 
A(R^ry,n)\text{ is positive and } \\
 |C_{\max}(MA(R^ry,n))|\in [|C_{\max}(M)|,K'^{m}|C_{\max}(M)|]\big\}\Big)>(1-\sigma'^m) \lambda_{d-1}(M\Delta)
 \end{multline}
 \end{prop}
 \begin{proof}
  In Proposition \ref{prop:determined} choose $M''$ to be a positive matrix. 
 \end{proof}
We also have:

 \begin{prop}\label{prop:contraction} There exist constants $\tau<1,\alpha<1$, and $K$ 
 so that for any matrix of Rauzy induction $M=M(x,r)$ and $j\in \mathbb{N}$  we have
 \begin{multline*}
 \lambda_{d-1}\{y\in M\Delta: \exists m \text{ so that }|C_{\max}(MA(R^ry,m))|<K^j|C_{\max}(M)| \text{ and }\\
  diam(V(MA(R^ry,m))<\tau^j diam(V(M)\}>(1-\alpha^j)\lambda_{d-1}M\Delta
 \end{multline*}
 \end{prop}

 In order to prove these propositions 
 we first prove

\begin{prop}\label{prop:prob decay} Let  $(\Omega,\mu)$ be a measure space and 
$F_i:(\Omega,\mu) \to \{0,1\}$ be a sequence of random variables such  that there exists $0<\rho<\frac{1}{2}$ so that for any $j$, the conditional probability that $F_j$ is $1$ given $F_1,...,F_{j-1}$ is at least $\rho$. Then  
\begin{itemize}
\item For all $j>i$, $\mu(\{\omega:F_\ell(\omega)=0 \text{ for all }i\leq \ell\leq j\})\leq (1-\rho)^{j-i}$. 

\item 
 For all $\epsilon>0$ there exists $C$ and $\tau<1$ depending on $\epsilon$ and $\rho$ so that for all $N$,
$$\mu(\{\omega: \sum_{i=1}^NF_i(\omega)<N(1-\epsilon)\rho\})<C\tau^N\mu(\Omega).$$
\end{itemize}
\end{prop}

To prove Proposition~\ref{prop:prob decay} we make a comparison to a case of independent random variables:
\begin{lem}Let  $(\Omega,\mu)$ be a measure space and 
$F_i:(\Omega,\mu) \to \{0,1\}$ be a sequence of random variables such  that there exists $0<\rho<1$ so that for any $j$, the conditional probability that $F_j$ is $1$ given $F_1,...,F_{j-1}$ is at least $\rho$. Let $G_i:(\Omega,\mu) \to \{0,1\}$ be independent and distributed according to $\mu(G_i^{-1}(1))=\rho$. Then  for all $\ell$ and $r$,
$$\mu(\{\omega:\sum_{i=1}^\ell F_i(\omega)\leq r\})\leq \mu(\{\omega:\sum_{i=1}^\ell G_i(\omega)\leq r\}).$$

\end{lem}
\begin{proof} Let $X=\{0,1\}^\ell \times[0,1]^\ell$. Let $\nu$ be a measure defined on $X$ by 
$$\nu((a_1,...,a_\ell),A_1 \times...\times A_\ell))=\mu(F_1^{-1}(a_1)\cap ...\cap F_\ell^{-1}(a_\ell))\lambda_\ell(A_1\times...\times A_\ell).\footnote{$\lambda_\ell$ is Lebesgue measure on $[0,1]^\ell$.}$$ 
Notice that $$\mu(\{\omega: \sum_{i=1}^\ell F_i(\omega)=r\})=\nu(\{(\vec{v},\vec{t}):\sum_{i=1}^\ell v_i=r\})$$ for all $r$. 
Let $\Phi:X\to \{0,1,*\}^\ell \times [0,1]^\ell$ by  $\Phi((\vec{v},\vec{t}))=(\vec{w},\vec{t})$ where
$$w_i=\begin{cases}0 &  \text{ if }v_i=0\\
1& \text{ if }v_i=1 \text{ and }\mu(F_1^{-1}(v_1)\cap...\cap F_{i-1}^{-1}(v_{i-1}))\rho\geq t_i\mu(F_1^{-1}(v_1)\cap...\cap F_i^{-1}(v_i))\\
* & \text{ else}
\end{cases}.$$ 
By construction we have that the $(\Phi_*\nu)$ conditional probability that $w_i=1$ given $w_1,...,w_{i-1}$ is exactly $\rho$. To see this, first let $\Phi_{\vec{s}}(\vec{v})=\vec{w}$ where $\Phi(\vec{v},\vec{s})=(\vec{w},\vec{s})$. Notice that for any $a_1,...,a_{i-1}\in \{0,1\}^{i-1}.$
\begin{multline*}
\nu(\{(\vec{v},\vec{s}):(v_1,...,v_{i-1})=(a_1,...,a_{i-1}) \text{ and } \Phi_{\vec{s}}(\vec{v})_i=1\})=\\
\rho\nu(\{(\vec{v},\vec{s}):(v_1,...,v_{i-1})=(a_1,...,a_{i-1})\}).
\end{multline*}
Also, because $\Phi_{\vec{s}}(\vec{v})$ does not depend on $s_1,...,s_{i-1}$,  we have that the tuples  $(a_1,...,a_{i-1}),(b_1,..,b_{i-1})\in \{0,1,*\}^{i-1}$ have the property \color{black} that $a_j=0$ iff $b_j=0$ implies 
\begin{multline*}
\frac{\Phi_*\nu(\{(\vec{w},\vec{t}): w_1=a_1,..,w_{i-1}=a_{i-1} \text{ and }w_i=1\})}{\Phi_*\nu(\{(\vec{w},\vec{t}): w_1=a_1,..,w_{i-1}=a_{i-1}\})}=\\
\frac{\Phi_*\nu(\{(\vec{w},\vec{t}): w_1=b_1,..,w_{i-1}=b_{i-1} \text{ and }w_i=*\})}{\Phi_*\nu(\{(\vec{w},\vec{t}): w_1=b_1,..,w_{i-1}=b_{i-1}\})}
=\rho.
\end{multline*}

\color{black}

That is to say, changing $*$ to $1$ or vice-versa does not affect conditional probabilities. 
Therefore we have that 
$$(\Phi_*\nu)(\{(\vec{w},\vec{t}):|\{i\leq \ell:w_i=1\}|=r\})=\mu(\{\omega:\sum_{i=1}^\ell G_i(\omega)=r\})$$ for all $r$. However we also clearly have that 
$$(\Phi_*\nu)(\{(\vec{w},\vec{t}):|\{i\leq \ell:w_i=1\}|\leq r\})\leq \nu(\{(\vec{v},\vec{t}):\sum_{i=1}^\ell v_i\leq r\})=\mu(\{\omega: \sum_{i=1}^\ell F_i(\omega\leq r\}).$$
This establishes the lemma.
\end{proof}

Proposition \ref{prop:prob decay} now follows from  the following standard large deviations estimate whose proof is omitted. 
\begin{lem} Let $F_i:(\Omega,\mu) \to \{0,1\}$ be independent and distributed according to $\mu(F_i^{-1}(1))=\rho$. Then for any $\epsilon>0$ there exists $c_1$ and $c_2<1$ so that for all $\ell$,$$\mu(\{\omega:\sum_{i=1}^\ell F_i(\omega)\leq (\rho-\epsilon)\ell\})\leq c_1c_2^{\ell}\mu(\Omega).$$
\end{lem}

\begin{proof}[Proof of Proposition \ref{prop:balanced}] Whenever $\zeta_1\geq \zeta_2$, we have that $\zeta_2$-balanced implies $\zeta_1$-balanced and so it suffices to prove the proposition for $\zeta$ as in Lemma \ref{lem:bal often}.  Let $K'$ be the constant given by Lemma \ref{lem:bal often}.   Define  $F_j:\Delta \to \{0,1\}$ by 
$$F_j(y)=\begin{cases}1 &\text{ if there exists }n \text{ so that } |C_{max}(M(R^ry,n)|\color{black} \in [2^jK'^j,2^{j+1}K'^{j+1}] 
\text{ and }M(R^ry,n) \text{ is }\zeta\text{-balanced}\\0& \text{ else}\end{cases}.$$
We claim that by Lemma \ref{lem:bal often}, the  $F_i$ satisfy the assumption of Proposition \ref{prop:prob decay}.  This is because if $M$ is a matrix of Rauzy induction so that $C_{\max}(M)\in [2^jK'^j,2^{j+1}K'^j]$ we have 
\begin{equation}\label{eq:each mat indep}
\mu(\{x\in M\Delta:F_j(x)=1\})>\rho' \mu(M\Delta).
\end{equation} Indeed, the chance $M$ becomes balanced before its norm increases by $K'$ is at least $\rho'$. This will give a matrix of norm at least $2^jK'^j$ and most $2^{j+1}K'^jK'$. This matrix will cause $F_j$ to be $1$,  establishing \eqref{eq:each mat indep}. Since this is true for every matrix, and the outcome of previous $F_i$ is about matrices with norm less than $2^jK'^j$ we have the assumption of Proposition \ref{prop:prob decay}. Indeed we apply \eqref{eq:each mat indep} to the  matrices $M(x,n)$ so that $|C_{\max}(M(x,n))|\in[2^jK'^j,2^{j+1}K'^j]$ and 
$|C_{\max}(M(x,n-1))|<2^{j}K'^j$. 
 The proposition follows with $K=2K'$ and $\sigma=\rho'$. 
\end{proof}

\begin{proof}[Proof of Proposition \ref{prop:determined}] This is similar to the previous proof. Let $K=2^{s+1}K'$ (where $s$ is as in the definition of $M''$).
Define  $F_j:\Delta \to \{0,1\}$ by $F_j(y)=1$ if there exists $n$ 
such that 
$R^ry \in M''\Delta$
 and $|C_{\max}(M'(R^ry,n))| \in [2^jK^j,2^{j}K^{j+1}]$ and $0$ otherwise. 
We claim that the conditional probability $F_j=1$ given $F_1,...,F_{j-1}=0$ is at least $\rho\zeta^{-d}\lambda(M'(w,r)\Delta)$.
 Indeed the probability that there exists $n$ so that $ M'(R^ry,n)$ is $\zeta$-balanced and $|C_{\max}( M'(R^ry,n)| \in [2^jK^j,2^{j}K^{j+1}]$ is at least $\rho$.
 By Lemma \ref{lem:bal distort} once this occurs, the conditional probability  that 
 $z\in  M'(y,r+n)\Delta$ 
 satisfies $z\in M''\Delta$ is at least $\zeta^{-d}\lambda_{d-1}(M''\Delta).$ 
We may now apply Proposition \ref{prop:prob decay}. 
\end{proof}

\begin{proof}[Proof of Propostion \ref{prop:contraction}] Choose $M''$ to be a positive matrix and note that each occurrence of a fixed positive matrix contracts the simplex by a definite amount. Repeat the proof of Proposition \ref{prop:determined} for this $M''$ and apply the second conclusion of Proposition \ref{prop:prob decay} to obtain that off of a set of exponentially small measure in $n$ we have at least $\frac {\rho'} 2n$ occurrences of $M''$ by the time the matrix norm increases by a factor of $(2K)^n$. By the first sentence of the proof this establishes the proposition.  
\end{proof}

\section{Remaining on Left Hand Side, Remaining on Right Hand Side} 

The object of this section is to establish  the two Theorems below and Lemma \ref{lem:angle to sing} at the end of this section. The first Theorem says if we are on the left hand side, a property that holds for most points on
the face $V(M)$  of a simplex $M\Delta$ leads to saying the same about most points in the simplex $M\Delta$ itself.  The second  makes the same statement on right hand side using the face $W(M)$.  We prove the first Theorem; the proof of the second is identical. 
We will need these statements because for example we wish to make statements about matrices  for Rauzy induction just on LHS and have the estimates hold for the entire simplex.   

We say that  a matrix of Rauzy induction, $M(x,n)$ is {\em maximal} for given $N$, if $\|M(x,n)\|\geq \frac{N}2$, and $\|M(x,n-1)\|<\frac N 2$. 
The point of this definition is the set of simplices  $A\Delta$ for $A$ maximal for given $N$   will cover $\Delta$ (minus the codimension $1$ set  where some power of Rauzy induction is not defined)  and there is no redundancy.

\begin{thm}
\label{thm:most lhs} Given $\zeta$ there exists $C$ so that if $M:=M(x,n)$ is a matrix of Rauzy induction,  $N \in \mathbb{R}_+,$ and  $0<\epsilon<1,  0<\delta<1$  are  constants  such that
\begin{enumerate}
\item $\pi (R^nx)=\pi_L$
\item \label{cond:last 2 bigger}$\underset{j \in \{d-1,d\}}{\min}|C_j(M)|>\frac{N}{\epsilon^2}\underset{i \in \{1,...,d-2\}}{\max}|C_i(M)|$
\item\label{cond:bal in pieces} $\frac{\underset{j \in\{ d-1,d\}}{\max}\, |C_j(M)|}{ \underset{j \in\{ d-1,d\}}{\min}\, |C_j(M)|}<\zeta$ and $\frac{\underset{i \leq d-2}{\max}\, |C_i(M)|}{ \underset{i \leq d-2}{\min}\, |C_i(M)|}<\zeta.$
\end{enumerate}
\begin{itemize}
 \item if    $A_1,...,A_r$ are a set of matrices on LHS such that
$$\lambda_{d-3}\bigl(V(M)\setminus \cup_{i=1}^r V(MA_i)\bigr)>\delta \lambda_{d-3}(V(M)) \text{ and } \|A_i\|<N$$  then  $$\lambda_{d-1}(M\Delta \setminus \cup_{i=1}^r MA_i\Delta)>\frac{\delta}{C}\lambda_{d-1}(M\Delta).$$
\item If  $A_1,...,A_r$  satisfy $$\lambda_{d-3}\bigl(V(M)\setminus \cup_{i=1}^r V(MA_i)\bigr)<\delta \lambda_{d-3}(V(M)) \text{ and } \|A_i\|<N$$  then  
$$\lambda_{d-1}(M\Delta \setminus \cup_{i=1}^r MA_i\Delta)<C(\delta+\epsilon)\lambda_{d-1}(M\Delta).$$

\end{itemize}

\end{thm}

We also have a similar result about the right hand side. 

\begin{thm}
\label{thm:most rhs} Given $\zeta$ there exists $C$ so that if $M:=M(x,n)$ is a matrix of Rauzy induction,  $N \in \mathbb{R}_+,$ and  $0<\epsilon<1,  0<\delta<1$  are  constants  such that
\begin{enumerate}
\item $\pi (R^nx)=\pi_R$
\item \label{cond:last 2 bigger}$\underset{i \in \{1,...,d-2\}}{\min}|C_i(M)|>\frac{N}{\epsilon^2}\underset{j \in \{d-1,d\}}{\max}|C_j(M)|$
\item\label{cond:bal in pieces} $\frac{\underset{j \in\{ d-1,d\}}{\max}\, |C_j(M)|}{ \underset{j \in\{ d-1,d\}}{\min}\, |C_j(M)|}<\zeta$ and $\frac{\underset{i\leq d-2}{\max}\, |C_i(M)|}{ \underset{i \leq d-2}{\min}\, |C_i(M)|}<\zeta.$
\end{enumerate}
\begin{itemize}
 \item if    $B_1,...,B_r$ are a set of matrices on right hand side such that
$$\lambda_{1}(W(M)\setminus \cup_{i=1}^r W(MB_i))>\delta \lambda_1(W(M)) \text{ and } \|B_i\|<N$$  then  $$\lambda_{d-1}(M\Delta \setminus \cup_{i=1}^r MB_i\Delta)>\frac{\delta}{C}\lambda_{d-1}(M\Delta).$$
\item If  $B_1,...,B_r$  satisfy $$\lambda_1(W(M)\setminus \cup_{i=1}^r W(MB_i))<\delta \lambda_{1}(W(M)) \text{ and } \|B_i\|<N$$  then  
$$\lambda_{d-1}(M\Delta \setminus \cup_{i=1}^r MB_i\Delta)<C(\delta+\epsilon)\lambda_{d-1}(M\Delta).$$

\end{itemize}

\end{thm}

\color{black}
We prove Theorem~\ref{thm:most lhs}. The proof of Theorem~\ref{thm:most rhs} is essentially identical.

\begin{proof}
[Proof of Theorem \ref{thm:most lhs}  modulo \eqref{eq:most} below]
Given $M$
let $\mathcal{M}_N$ be  the set of all maximal matrices $A$ of size $N$  where $d-1$ and $d$ have not won. 
The corresponding  simplices  $MA\Delta$ cover   $V(M)$. 
 The first  conclusion says  it 
{\em cannot} happen that a subset of $V(M)$ whose complement has definite proportion of the measure of $V(M)$ can produce matrices in $\mathcal{M}_N$ that almost cover $M\Delta$.
To show this it is  enough to show for all $A,A'\in {\mathcal{M}}_N$ that 
\begin{equation}\label{eq:comparison}
\frac{\lambda_{d-3}(V(MA))}{\lambda_{d-3}(V(MA'))}<\zeta^2\color{black} \frac{\lambda_{d-1}(MA\Delta)}{\lambda_{d-1}(MA'\Delta)}
\end{equation}

We prove (\ref{eq:comparison}). First, notice that applying Lemma  \ref{lem:volume face}  we have 

$$\frac{\lambda_{d-3}(V(MA))}{\lambda_{d-3}(V(MA'))}=\frac{\prod_{i=1}^{d-2}|C_i(MA)|^{-1}}{\prod_{i=1}^{d-2}|C_i(MA')|^{-1}}.$$
\color{black} Now, 
for $j \in \{d-1,d\}$ and $A\in \mathcal{M}_N$  
$$|C_j(MA)|<|C_j(M)|+N\max\{|C_i(M)|:i\leq d-2\}<2|C_j(M)|$$  by Assumption (\ref{cond:last 2 bigger}) of the Theorem.  Therefore, for any $A,A'\in \mathcal{M}_N$ we have 
$$\frac{\prod_{i=1}^{d-2}|C_i(MA)|}{\prod_{i=1}^{d-2}|C_i(MA')|} =\frac{|C_{d-1}(MA')|\cdot |C_d(MA')|}{|C_{d-1}(MA)|\cdot |C_d(MA)|}\frac{\prod_{i=1}^d|C_i(MA)|}{\prod_{i=1}^d|C_i(MA')|}\leq 4\frac{\prod_{i=1}^d|C_i(MA)|}{\prod_{i=1}^d|C_i(MA')|} .$$
So by Lemma \ref{lem:volume} we have Inequality (\ref{eq:comparison}).   


The heart of the proof of the second conclusion is  the following inequality which we will prove {\em after} we use it.  There exists $C$ such that for all  $\epsilon<1$,  if we cover {\em all} of $V(M)$ with simplices corresponding to matrices in $\mathcal{M}_N$ then 
\begin{equation}
\label{eq:most}\lambda_{d-1}(\cup_{A\in \mathcal{M}_N}MA\Delta)>(1-C\epsilon)\lambda_{d-1}( M\Delta).
\end{equation}
The above inequality (which we have not proven) and \eqref{eq:comparison}, which we have, establishes the the Theorem. Indeed, we treat $M\Delta$ as $\cup_A MA\Delta \cup (M\setminus \cup_AMA\Delta)$ and apply \eqref{eq:comparison} on the first part and invoke \eqref{eq:most} to show the second part is small. 
\end{proof}

Now (\ref{eq:most}) follows from the next two Propositions. 
\begin{prop}\label{lem:prob 2 short} There exists $C'$ such that for $\epsilon<1$ if $M:=M(x,n)$ is a matrix of Rauzy induction satisfying hypotheses (2) and (3) of Theorem~\ref{thm:most lhs}  and the permutation of $R^nx$ is $\pi_L$, 
then 
$$\lambda_{d-1}(\{y \in M\Delta:\underset{j \in \{d-1,d\}}{\max}(R^ny)_j>\frac \epsilon N\})<C'\epsilon \lambda_{d-1}(M\Delta).$$
\end{prop}

\begin{prop}\label{lem:lhs future} There exists $C$ such that  for all small enough $\epsilon$ and all $N$ and permutation $\pi_L$ 
\begin{multline*}\lambda_{d-1}(\{x\in \Delta: x_{d-1}+x_d<\frac {\epsilon}{N} \text{and d-1  or d win and the corresponding matrix }\|A(x,m)\|<N\})\\ <C\epsilon\lambda_{d-1}(\{x:x_{d-1}+x_d<\frac{\epsilon}{N}\}).
\end{multline*}
\end{prop}

\noindent
\textit{Proof of (\ref{eq:most}) assuming Propositions \ref{lem:prob 2 short} and \ref{lem:lhs future}.} \color{black} 
  We choose $C$ large enough so that (\ref{eq:most}) is vacuously true if $\epsilon$ is so large that the set defined by Proposition \ref{lem:lhs future} is empty.  Let $M$ in \eqref{eq:most} be of form $M(z,n)$. \color{black}
By Proposition~\ref{lem:prob 2 short} (and taking complements) there is $C'$ such  that $$\lambda_{d-1}(\{y\in M\Delta: R^n(y)_{d-1}+R^n(y)_d<\frac{\epsilon}{N}\})>
(1-C'\epsilon)\lambda_{d-1}(M\Delta).$$ 
To prove (\ref{eq:most}) it then suffices to show 
that for  $C$ large enough,  
\begin{multline}
\label{eq:enough}\lambda_{d-1}(\{y \in\cup_{A\in \mathcal{M}_N} MA\Delta:R^n(y)_{d-1}+R^n(y)_d <\frac{\epsilon}N \}) \geq\\
(1-C\epsilon)\lambda_{d-1}(\{y \in M\Delta:R^n(y)_{d-1}+R^n(y)_d <\frac{\epsilon}N\}).
\end{multline} For then we would combine these last two inequalities.
We now prove (\ref{eq:enough}).
Denote by $T_0(\frac{\epsilon}{N})$ the set on the left.   Set $x=R^ny$ so $x_{d-1}+x_d<\frac{\epsilon}{N}$.  We now show that the set of $x$ such that $x_{d-1}+x_d <\frac{\epsilon}{N}$ and  {\em  $d-1$ or $d$ wins} (so we leave the left side) with a matrix of size at most $N$ is $O(\epsilon)$.  Then we will apply Jacobian 
estimates to conclude the same thing about the set of $y$.    
Then we take complements to establish \eqref{eq:enough}\color{black}.

To that end, define  $$S_0(\frac{\epsilon}{N})=R^n(T_0(\frac{\epsilon}{N}))=\{x:x_{d-1}+x_d<\frac{\epsilon}{N}\}$$ where as always  $R$ denotes (normalized)  Rauzy induction. For $j\in {\mathbb N}$
set 
$$S_j(\frac{\epsilon}{N})=\{x\in\Delta:\frac {\epsilon} {N 2^{j+1}}\leq x_{d-1}+x_d<
\frac{\epsilon}{N 2^{j}}\}$$
so $S_0(\frac{\epsilon}{N})$  is a disjoint union of  $S_j(\frac{\epsilon}{N})$ There are corresponding sets $$T_j(\frac{\epsilon}{N})=R^{-n}(S_j(\frac{\epsilon}{N})) \cap M\Delta \color{black}$$ whose union over $j$ is $T_0(\frac{\epsilon}{N})$.  

Now for each $j$  apply Proposition~\ref{lem:lhs future}  to $\cup_{i\geq j} S_i$ to find    
$$       \lambda_{d-1}(\{x \in \cup_{i\geq j}S_i(\frac{\epsilon}{N}):d-1 \text{ or }d  \text{ wins within } n \text{ steps and }\|A(x,n)\|<2^jN\})<C\epsilon\lambda_{d-1}(\cup_{i\geq j}S_i(\frac{\epsilon\color{black}}{N})).$$ 
Now 
$\{x \in S_j(\frac{\epsilon}{N}):d-1 \text{ or }d  \text{ wins within } n \text{ steps and }\|A(x,n)\|<2^jN\}$ is contained in the set on the left and the measure of the set on the right is proportional to  the measure of $S_j$. 
We conclude for a new constant $\tilde{C}$, 
$$  \lambda_{d-1}(\{x \in S_j(\frac{\epsilon}{N}):d-1 \text{ or }d  \text{ wins within } n \text{ steps and }\|A(x,n)\|<2^jN\})\leq \tilde{C}\epsilon\lambda_{d-1}(S_j(\frac{\epsilon\color{black}}{N})).$$

The measure where $\|A\|\leq N$ is even smaller. 
\color{black}
By Lemma~\ref{lem:jacobian}, and Assumption (\ref{cond:bal in pieces}) of Theorem \ref{thm:most lhs} there exists $\hat{C}$ depending on $\zeta$ such that for any $U,V \subset S_j(\frac{\epsilon}{N})$ we have 
$$\frac{\lambda_{d-1}(R^{-n}U \cap M\Delta)}{\lambda_{d-1}(R^{-n}V\cap M\Delta)}<\hat{C}\frac{\lambda_{d-1}(U)}{\lambda_{d-1}(V)}.$$
This says that inside $T_j(\frac{\epsilon}{N})$ the proportion of $y$ such that $d-1$ or $d$ wins is at most $\hat{C}\color{black}\epsilon\lambda_{d-1}(T_j(\frac{1}{N}))$.      Summing over  $j$ and then taking complements 
we  have proven
 (\ref{eq:most}).

\begin{proof}[Proof of Proposition~\ref{lem:prob 2 short}]

 As before,  set $$S_0(\frac{\epsilon^2}{N})=\{x\in\Delta:x_{d-1}+x_d<\frac{\epsilon^2}{N}\}$$
and now for $j\in\mathbb{N}$, we define  $$\hat{S}_j(\frac{\epsilon}{N})=\{x\in\Delta:  \frac{2^j\epsilon}{N}\leq x_{d-1}+x_d<\frac{2^{j+1}\epsilon}{N}\}.$$

For some $c'>0$ depending on $\zeta$, the Assumption (3)  of Theorem~\ref{thm:most lhs}  implies
$$\inf_{z \in \hat{S}_j(\frac{\epsilon}{N})}|C_{d-1}(M)|z_{d-1}+|C_d(M)|z_d>c'\frac{2^{j}}{\epsilon}\sup_{y \in S_0(\frac{\epsilon^2}{N})}|C_{d-1}(M)|y_{d-1}+|C_d(M)|y_d,$$


This inequality  together with Assumption  (\ref{cond:last 2 bigger}) \color{black} on the size of the columns which says   
  
 $$\frac 1 {\underset{i \in \{d-1,d\}}{\min} |C_j(M)|}\sum_{i=1}^{d-2} |C_i(M)|z_m  \leq \frac 1 {\underset{i \in \{d-1,d\}}{\min} |C_i(M)|}\sum_{i=1}^{d-2} |C_i(M) <d\frac{\epsilon^2}{N}$$  

  implies there is a constant $C$ such that 
 \begin{equation} \label{eq:bound}
\sup_{z \in \hat{S}_j(\frac{\epsilon}{N})}\frac{1}{(|C_1(M)|z_1+...+|C_d(M)|z_d)^d}<C2^{-dj} \epsilon^d\inf_{y \in S_0(\frac{\epsilon^2}{N})} \frac{1}{(|C_1(M)|y_1+..+|C_{d}(M)|y_d|)^d}.
\end{equation}

Moreover  for some constant $C'$,  
\begin{equation}\label{eq:next bound}\lambda_{d-1}(S_j(\frac{\epsilon}{N}))\leq C'\frac{2^{2j}}{\epsilon^2}\lambda_{d-1}(S_0(\frac{\epsilon^2}{N})),
\end{equation}
Recalling Veech's Jacobian formula, Lemma \ref{lem:jacobian}, we see from (\ref{eq:bound}) and \eqref{eq:next bound} that there exists $c''>0$ so  that for each $j$, 
\begin{multline*}\lambda_{d-1}(\{y\in M\Delta: R^ny \in S_0(\frac{\epsilon^2}{N})\})
\geq c''2^{dj-2j}\epsilon^{2-d}\lambda_{d-1}(\{y\in M\Delta: R^ny\in \hat{S}_j(\frac{\epsilon}{N})\})\geq \\
c''2^j\epsilon^{-2}\lambda_{d-1}(\{y\in M\Delta: R^ny\in \hat{S}_j(\frac{\epsilon}N)\}).
\end{multline*} This uses that $d\geq 4$. \color{black}
Summing over  $j$ from $1$ to $[\log_2 \frac{N}{\epsilon}]$ we see for some $C'$ 
 that  $$\lambda_{d-1}(\{y\in M\Delta:\max_{i\in \{d-1,d\}} (R^ny)_i)
 \geq  \frac{\epsilon}{N}\}\leq  C'\epsilon^2\lambda_{d-1}(\{y\in M\Delta: R^ny\in S_0(\frac{\epsilon^2}{N})\})\leq C' \epsilon^2  \lambda_{d-1}(M\Delta).$$
\end{proof}

\subsection{Proof of Proposition \ref{lem:lhs future}}  Let $\Lambda_{d-2}=\spanD(e_1,...,e_{d-2})$. Let $\tilde{\mathcal{A}}$ be a set of matrices where $d$ and $ d-1$ have not won and 
\begin{itemize}
\item $\|A\|\in [N,2N]$ for all $A \in \tilde{\mathcal{A}}$
\item $A\Delta \cap A'\Delta=\emptyset$ for $A,A'\in \tilde{\mathcal{A}}$ with $A\neq A'$. 
\end{itemize}
For example we could choose $A$ maximal for $2N$ with the additional property that $d$ and $d-1$ have not won. 

The assumption that $d-1$ and $d$ have not won implies $\Lambda_{d-2}\subset \cup_{A \in \tilde{\mathcal{A}}}A\Delta$.
We need the following lemma in the proof. In this lemma let $\Lambda_{d-2}$ denote $\spanD(e_1,...,e_{d-2})$.

\begin{lem} For all $s,t\in [0,1]$ and $A \in \tilde{A}$ we have
$$\lambda_{d-3}(A\Delta \cap \{x\in \Delta:x_{d-1}=t,x_d=s\})\geq \lambda_{d-3}(V(A))(1-2(t+s)N)^{d-3}.$$
\end{lem}
\begin{proof} Consider the simplex $A \Delta$ as being made of codimension 2 slices parallel to $V(A)$ (which is $A\Delta \cap \Lambda_{d-2}$). 
Let $p_{d-1}(A)=\frac{C_{d-1}(A)}{|C_{d-1}(A)|}, \, p_d(A)=\frac{C_d(A)}{|C_d(A)|}$  denote the two extreme points, of $A\Delta$ that are disjoint from $V(A)$. Every  point of a  slice parallel to $V(A)$ has the form $w+ap_{d-1}+bp_d$, where $a+b\leq 1$ are fixed and determine the slice and $\frac{w}{|w|}\in V(A)$.  Every side has length $1-a-b$ times what it had in $V(A)$. It follows that the volume of this slice is $\lambda_{d-3}(V(A))(1-a-b)^{d-3}$.

Now the $d-1$ and $d$ entries of $p_{d-1}$ and $p_d$  are respectively at least {$\|A\|^{-1}=(2N)^{-1}$}. Thus  $A\Delta \cap \{x\in \Delta:x_{d-1}=t,x_d=s\})$ is the set of points in $A\Delta$ that have the form $(1-a-b)w+ap_{d-1}+bp_d$, where $a\leq 2tN$, $b\leq 2sN$ and $w \in V(A)$.  The lemma follows. 

\end{proof}

\begin{proof}[Proof of Proposition \ref{lem:lhs future}]
The set in the proposition $$\{x\in \Delta: x_{d-1}+x_d<\frac {\epsilon}{N} \text{ and }d-1 \text{ or }d \text{ win and the corresponding matrix }\|A(x,m)\|<N\}$$  that we would like to show has small measure is contained in the complement of $\tilde {\mathcal{A}}\Delta  \cap (\cup_{c\leq \frac\epsilon N}\Delta_c)$. We view $\tilde {\mathcal{A}}\Delta  \cap (\cup_{c\leq \frac\epsilon N}\Delta_c)$  as being cut by codimension 2 planes parallel to $\Lambda_{d-2}$. By the previous lemma $\tilde{\mathcal{A}} \Delta$ intersected with any such slice has measure at least $(1-2N\frac\epsilon N)^{d-3}\lambda_{d-3}(\Lambda_{d-2})$. 
Since the volume of a slice in $\Delta$ parallel to $\Lambda_{d-2}$ is at most the volume of $\Lambda_{d-2}$, we have that $\tilde{\mathcal{A}}\Delta$ occupies at least a $(1-2\epsilon)^{d-3}$ proportion of this set. Because for small $s$ we have  $(1-s)^r=1-rs+O(s^2)$ the proposition follows with $C=2(d-3)$ (since we may choose $\epsilon$ small enough).
\end{proof}

 Later in the paper,  we will need one additional result.

\begin{lem}\label{lem:angle to sing} There exist $c>0$ so that for $k_0$ large enough, given a matrix $M=M(x,r)$ at beginning  of freedom on LHS, for all $y$ except for a subset of $M\Delta$ of measure at most $10^{-c_1(k+k_0)^4}\lambda_1(M\Delta)$, there exists   $A(R^ry,m)$, a matrix of freedom on LHS such that 

\begin{equation}\label{eq:ang LHS small 1}\Theta(C_i(MA(R^ry),C_1(MA(R^ry,m)))<10^{-c(2k+k_0)^6}\end{equation}  for $i\leq d-2$ and 
\begin{equation}\label{eq:ang LHS small 2}\Theta(C_{d-1}(M),C_d(M))<10^{-(2k+1+k_0)^6}.
\end{equation}
\end{lem}
\begin{proof}  By Proposition \ref{prop:contraction} applied to $\mathcal{R}_{d-2}$ we have that there exists $\tau<1$ and $c>0$ so that for any positive matrix $M$,
 for all but a set of measure $10^{-c(k+k_0)^6}\lambda_{d-3}(V(M))$ set of points $y \in V(M)$ we have that there exists $A(R^ry,m)$ with 
 ${\|A(R^ry,m)\|<10^{(k+k_0)^6-(k+k_0)^4+(k+k_0)^2}}$ such that 
 $$diam\big(V(MA(R^ry,m))\big)<\tau^{c[((k+k_0)^6-((k+k_0)^4]}diam(V(M))<\tau^{c'(k+k_0)^6}diam\big((V(M)\big),$$ 
 with the last inequality holding for some $c'>0$. 
 This implies that for all but a $10^{-c(k+k_0)^6}$ proportion of $V(M)$ we have that the matrix given by freedom on the left hand side $A$ has 
$$\max\{\Theta(C_i(MA),C_{i'}(MA)):i,i'\leq d-2\}<\tau^{c'(k+k_0)^6}\max\{\Theta(C_i(M),C_{i'}(M)):i,i'\leq d-2\}.$$
We now apply Theorem \ref{thm:most lhs} with $N=10^{(k+k_0)^6-(k+k_0)^4+(k+k_0)^2}$ and $\frac{1}{\epsilon^2}=10^{(k+k_0)^4+(k+k_0)^2}$.  We obtain inequality \eqref{eq:ang LHS small 1}   for matrices of freedom that cover all but a proportion $10^{c'(k+k_0)^4}$ of $M\Delta$. 
 During restriction on the LHS, the angle between these columns can only get smaller and so we obtain inequality \eqref{eq:ang LHS small 1} for the remainder of LHS.\color{black} 

We now prove the bound on $\Theta(C_{d-1}(M),C_d(M))$. Let $M'$ be the ancestor of $M$ at the end of freedom on the right hand side.  We have $C_{d-1}(M)=C_d(M')+(b+1)C_{d-1}(M')$ and $C_d(M)=C_{d}(M')+bC_{d-1}(M')$ where $b \geq10^{(2k+1+k_0)^6+(k+k_0)^4}$. So by Lemma \ref{lem:add vectors} we have  \eqref{eq:ang LHS small 2}.
\end{proof}

\section{Input and Output singular direction}
We will need to control  the size of singular values and directions  for the matrices arising from Rauzy induction. This will be necessary to control the geometry of simplices. This section is devoted to this endeavor.  Bounds on large singular values give bounds on small singular values, because our matrix preserves a (possibly degenerate) symplectic form. Before we begin our estimates we briefly describe this.

\subsection{Symplectic} \label{sec:symplectic}

Let $\Omega_\pi$ denote the (possibly degenerate) symplectic form, preserved by matrices of Rauzy induction from $\pi$ to $\pi$.  The preserved subspace of the symplectic form is its image and is the orthogonal complement of its kernel.  (See for example \cite[Section 1.9]{Viana}.) Let $Im(\Omega_\pi)$ be $\Omega_\pi \mathbb{R}^k$, the preserved subspace of the symplectic form.
When we speak of $\Omega_\pi^{-1}$ it is defined on the image of $\Omega_\pi$.

We use the following formula. For any path of Rauzy induction joining $\pi$ to $\pi'$ with matrix $M$ 
\begin{equation}
\label{eq:Viana}
M^T\Omega_\pi M=\Omega_{\pi'},
\end{equation}
where $M^T$ denotes $M$ transpose.

We now discuss the polar decomposition   $M=UP$ where $U$ is unitary and $P$ is positive definite. Since $P=U_1DU_1^{-1}$ for some unitary $U_1$ and diagonal matrix $D$ we have that $M=U'DV'$ with $U'$ and $V'$ unitary. Note that if $(v,a)$ is a pair consisting of a singular input direction and value of $M$, then $(v,\frac 1 a)$ is a pair consisting of singular output direction and value of $M^{-1}$, and vice-versa.  Also the singular values of $M$ and $M^T$ are the same.
By the symplectic property we can relate singular input or output directions and values of $M^{-1}$ to corresponding directions and values of $M^T$. We obtain:
\begin{lem}\label{lem:singular relations}
\begin{enumerate}[label=(\roman*)]
\item\label{sing rel 1} If $(v,a)$ is a pair of singular input direction and  singular value of $M^{-1}$   and $v\in Im(\Omega_\pi)$ 
then  $(\Omega_\pi  v,a)$ is such a pair for $M^T$.
\item\label{sing rel 2} If $(\Omega_\pi v,a)$ is a pair of singular input direction and value of $M^T$, then $(v,\frac 1 a)$ is a pair of singular output direction and value of $M$. 
\item\label{sing rel 3} The singular values of $M$ in the invariant subspace preserved by the symplectic form come in pairs $a$ and  $\frac 1 a$
\end{enumerate}
\end{lem}

\color{black}
\subsection{Largest singular input and output vectors  of $M^T$}
 By the previous section this will also tell us about the singular directions of $M$. \color{black} In this section we consider the singular decomposition of $M^T=UP$. 
By largest (second largest) singular input vectors $w,w'$ we mean those vectors expanded most (second most) by $P$ and by output vector we mean their images under $U$.   
Let $W$ be the ortho-complement of $w$, so $w'\in W$ and $Proj_W$ the  orthogonal projection onto $W$.
\color{black}
\begin{prop} 
\label{prop:closeangle}
At the end of freedom on LHS, the largest and second largest  singular input  vectors $w,w'$ of $M^T$ satisfy
\begin{enumerate}
\item For all $\epsilon>0$,  for $k_0$ large enough, then  at stage $k$, 
 $\Theta(C_d(M),w)<10^{-(\frac 3 2 -\epsilon)(k+k_0)^4}$
  
 \item For all $\epsilon>0$ for $k_0$ sufficiently large, $\Theta(Proj_WC_1(M),w')<10^{-(\frac 1 2 -\epsilon)(k+k_0)^4}$.


\end{enumerate}
\end{prop}

\begin{proof}
We first note that the estimates on $U_k$ and $v_k$ in Proposition \ref{prop:sizes}
imply that for all $\epsilon>0$, for  $k_0$ large enough, then for any matrix $M$ during freedom on LHS at stage $k$ we have that for all $i\leq d-2$
\begin{equation}
\label{eq:columnbounds}10^{-(1+\epsilon)(k+k_0)^4}<\frac{|C_i(M)|}{|C_d(M)|}<10^{-(1-\epsilon)(k+k_0)^4}.
\end{equation}

\color{black}
We prove the first conclusion.
It suffices to show that if $v$ is a unit vector so that $\Theta(v,C_d(M))=10^{-(\frac 3 2-3\epsilon) (k+k_0)^4}$ then

\begin{equation}
\label{eq:suffice}|M^Tv|_2<|M^T\frac{C_d(M)}{|C_d(M)|_2}|_2
\end{equation}

Indeed, this establishes that there is 
a local 
 maximum of the function $f:S^{d-1}\subset \mathbb{R}^d \to \mathbb{R}_+$ defined by $f(v)=|M^Tv|$ within angle  $10^{-(\frac 3 2-3\epsilon) (k+k_0)^4}$  of $ \frac{C_d(M)}{|C_d(M)|_2}$.  However a local maximum in the positive quadrant is a global maximum. The reason is that at  a local maximum the level set of $f$ defined on all of $\mathbb{R}^d$ which  is an ellipsoid must be tangent to the sphere. There can only be one such point in the  positive quadrant.

Showing (\ref{eq:suffice}) 
 is equivalent to showing that  
\begin{equation}\label{eq:angle suff est}\sum_{i=1}^{d}(C_i(M)\cdot v)^2<\sum_{i=1}^d(C_i(M) \cdot \frac{C_d(M)}{|C_d(M)|_2})^2.
\end{equation} 
Now  from  the assumed equation for $\Theta(v,C_d(M))$ and  the bound 
${\Theta(C_{d-1}(M),C_d(M))\leq 10^{-(k+k_0)^6}}$ (Lemma \ref{lem:angle to sing}), we get,  writing inner products in terms of $\cos$ and using Taylor's expansion for $\cos\theta$ that   for $j=d-1,d$ and $k_0$ large enough  
$$(C_j(M) \cdot v)-(C_j(M) \cdot \frac{C_d(M)}{|C_d(M)|_2})<-|C_{j}(M)|\frac{1}{3} (10^{- (\frac 
3 2- 3\epsilon)(k+k_0)^4})^2.$$

Then for  $j\geq d-1$  and $i\leq d-2$
 \begin{multline}\label{eq:bigger}(C_j(M) \cdot v)^2-(C_j(M) \cdot \frac{C_d(M)}{|C_d(M)|_2})^2< 
 -|C_j(M)|\cdot |C_{j}(M)|\frac{ 1}{9}(10^{- (\frac 3 2 -3\epsilon)(k+k_0)^4})^2<
 \\ -\frac{1}{9}(10^{2(1-\epsilon)(k+k_0)^4}|C_i(M)|)^2 (10^{- ( 3  -6\epsilon)(k+k_0)^4})=-\frac{1}{9} \color{black}10^{(-1+4\epsilon)(k+k_0)^4)}|C_i(M)|^2
 \end{multline}
 The first inequality uses that $a^2-b^2=(a-b)(a+b)$ and  
 ${C_j(M) \cdot \frac{C_d(M)}{|C_d(M)|_2}}>\frac{ 1}{3} |C_d(M)|$\color{black}.
\color{black}  The second inequality uses (\ref{eq:columnbounds}) to relate $|C_j|$ and $|C_i|$. 
  
We also have that for $i\leq d-2$ 
$$C_i(M) \cdot v-C_i(M) \cdot \frac{C_d(M)}{|C_d(M)|_2}<|C_{i}(M)|10^{-  (\frac 3 2 -3\epsilon)(k+k_0)^4} .$$ 
Thus for all $i \leq d-2$
\begin{equation}\label{eq:smaller}(C_i(M) \cdot v)^2-(C_i(M) \cdot \frac{C_d(M)}{|C_d(M)|_2})^2<2|C_{i}(M)|^210^{-  (\frac 3 2 -3\epsilon)(k+k_0)^4}.\color{black}
\end{equation}
Seeing that the sum of the upper bound  of (\ref{eq:bigger}) and  the upper bound  of (\ref{eq:smaller}) is negative for $k_0$ large enough, 
we have that our sufficient condition, Inequality (\ref{eq:angle suff est}), is satisfied establishing (1).

Now we outline (2). 
From the angle bound just proven, the length bound (\ref{eq:columnbounds}), 
and the fact the angle between $C_i(M)$ and $C_j(M)$ is bounded away from $0$ for $i\leq d-2$ and $j\geq d-1$ (Theorem \ref{thm:nue}), 
we have that for  all $\epsilon>0$ if $k_0$ is big enough, then for all $i\leq d-2, j\geq d-1$ 
\begin{equation}
\label{eq:smallbig}
|Proj_W(C_j(M))|<10^{(\frac 1 2 -\epsilon) (k+k_0)^4}|Proj_W(C_i(M))|.
\end{equation}

From this the second conclusion follows by an analogous argument.
\end{proof}

\begin{lem}\label{cor:second place}  Again let $w'$ the  second largest input vector of $M^T$,  a matrix at the end of freedom on LHS\color{black}. Then  for $k_0$ sufficiently large, $M^T(w')$ makes angle less than $10^{-(k+k_0)}$ with 
$(|C_1(M)|,...,|C_{d-2}(M)|,0,0)$.
\end{lem}
\begin{proof}  Let $W_d$ be the the orthocomplement of $C_d(M)$.  Let  $w_i$ be  the vector which is the orthogonal projection of $C_i(M)$ to $W_d$.  In particular $w_d$ is the zero vector.  Let $\hat w$ be the vector 
$$\hat w=(w_1\cdot w_1, \ldots, w_1\cdot w_{d-2},0,0).$$
 By Inequality \eqref{eq:ang LHS small 2} of Lemma \ref{lem:angle to sing} 
\begin{equation}\label{eq:small proj}
|w_{d-1}|\leq  10^{-(2k+k_0)^6}|C_{d-1}|.
\end{equation}
 Because $w'\in W^\perp$, \color{black} by the first  conclusion of Proposition~\ref{prop:closeangle}, for $\epsilon>0$, for $k_0$  large enough, 
 \begin{equation}\label{eq:close to proj}
 \Theta(w', W_d)\leq 10^{-(\frac 3 2-\epsilon)(k+k_0)^4}.
 \end{equation}


By the second conclusion of Proposition~\ref{prop:closeangle}, for all $\epsilon>0$, for $k_0$ large enough,
\begin{equation}\label{eq:close to w1}
\Theta(w',w_1)\leq 10^{-(\frac 1 2-{\epsilon})(k+k_0)^4}.
\end{equation}

These two  angle bounds, the bounds  $\frac{|C_j(M)|}{|C_i(M)|}\leq 10^{(1+\epsilon)\color{black}(k+k_0)^4}$ for $j\geq d-1$ and $i\leq d-2$  (inequality \eqref{eq:columnbounds}) \color{black} and the bound on $|w_{d-1}|$ imply that  
$$|w_i\cdot w_1|\geq 10^{2(k+k_0)}|w_j \cdot w_1|$$ for all $i\leq d-2$, $j\geq d-1$. Now $$M^Tw'=(C_1(M)\cdot w',\ldots, C_d(M)\cdot w').$$ Putting this together we see that $|M^Tw'-\hat{w}|\leq 10^{-2(1-\epsilon)(k+k_0)}|\hat{w}|.$ 
 This implies \begin{equation}
\label{eq:M^T}\Theta(M^Tw',\hat w)\leq 10^{-(1-\epsilon)(k+k_0)}.
\end{equation}
\color{black}
  Using  Lemma \ref{lem:angle to sing} which says the $C_i(M)$ are exponentially close in angle to each other, we see that  for all $i,i'\leq d-2$,  
  $$\left |\frac{|w_i\cdot w_1|}{|w_{i'}\cdot w_1|}-\frac{|C_i(M)|}{|C_{i'}(M)|	}\right|\leq 10^{-\tilde{c}(k+k_0)^6}.$$ 
This estimate together with (\ref{eq:M^T})  implies the statement of  the lemma.
\end{proof}

In the next Lemma, by proportional  we mean the ratio of two terms is bounded above and below by uniform constants.

\begin{lem}\label{lem:second freedom}  At the end of freedom on LHS,  the second largest singular value of $M^T$ is proportional to $|C_{\min}(M)|$.
\end{lem}
\begin{proof}By Proposition \ref{prop:closeangle} the second largest \color{black} input singular direction $w'$ is exponentially  close to $C_1(M)$. The proof of Lemma \ref{cor:second place} says that $M^Tw'$ is exponentially  close to $(w_1\cdot w_1,\cdots, w_{d-2}\cdot w_1,0,0)$. By Theorem \ref{thm:nue}, $|M^Tw'|$ is proportional to $\sum_{i=1}^{d-2} |C_i(M)|$.
 By Condition (1)* this quantity is proportional to $(d-2)|C_i(M)|$ for any $i\leq d-2$. 
\end{proof}

The next result is used in Section \ref{sec:rhs}.
\begin{prop}\label{prop:RHS sing} At the start of freedom on RHS
\begin{enumerate}
\item the largest singular  value of $M^T$ is proportional to $|C_{\max}(M)|$. 
\item The second largest is smaller than $10^{-(k+k_0)^5}\|M\|$.\end{enumerate} 
At the end of freedom on the RHS,  
\begin{enumerate}
\item[(3)] the largest is  proportional to \color{black}  $|C_{\max}(M)|$ and
\item[(4)] the second largest  is at least a constant multiple of $|C_{\min}(M)|$.
\end{enumerate} 
\end{prop}
\begin{proof}The claims (1) and (3)  are  trivial. We now prove (2). Inequality \eqref{eq:ang LHS small 1} of Lemma \ref{lem:angle to sing},  implies that for a constant $c>0$,  $\Theta(C_i(M),C_{i'}(M))<10^{-c(2k+k_0)^6}$ for all $i,i'\leq d-2$. Also by the bounds on $V_k,\, u_k$ in  Proposition \ref{prop:sizes},  we see  we have that $$\max\{\frac{|C_j(M)|}{|C_i(M)|}:i\leq d-2, \, j\geq d-1\}\leq \frac{V_{k-1}}{u_k}\leq10^{-\frac{1}2(k+k_0)^6},$$ for $k_0$ big enough. So analogously to  the first conclusion of Proposition~\ref{prop:closeangle}  for some $c'$, we have that the angle the  largest singular input vector makes with $C_1(M)$ is less than $10^{-c'(k+k_0)^6}$. It follows that $M^T$ restricted to the ortho-complement of the direction of largest singular vector is less than $10^{-(k+k_0)^5}\|M\|$ (if $k_0$ is large enough). The claim of the second largest singular value follows. 

The last conclusion (4) follows analogously to above and we sketch it. Analogously to Proposition~\ref{prop:closeangle} we have that the angle the top singular input vector  makes with $|C_{1}(M)|$ is smaller than $10^{-(\frac 32 -\epsilon)(k+k_0)^4}$ (if $k_0$ is large enough). From this it follows that the operator norm of $M^T$ on the orthocomplement of the largest input singular vector is at most $\max\{10^{-(\frac 3 2-\epsilon)(k+k_0)^4}|C_i(M)|,|C_j(M)|:i\leq d-2,\, j\geq d-1\}\geq |C_{\min}(M)|$.  
\end{proof}
\color{black}

\subsection{Small singular input and output directions of $M$ and choice of planes}

For this section let $\Delta'=(\{x\in \Delta:x_{d-1}=x_d=0\},\pi_L)$ and $\tilde{R}:\Delta' \to \Delta'$ be the first return of Rauzy induction to this set. Let $\tilde{A}(x,1)$ be the corresponding incidence matrix. Let $\tilde{A}(x,n+1)=\tilde{A}(x,n)\tilde{A}(\tilde{R}^nx,1).$
This is a transpose cocycle.

This section is devoted to the proof of the following Proposition. It says that we can arrange things so that a certain face always has volume at least comparable  to the volumes of other faces.  Its proof will be at the end of the section after several preliminaries.
\begin{prop} \label{prop:big face} There exists $c>0$, $N$ so that for any $M$ there exists $x\in \Delta'$, and   $n$  with
$$\lambda_{d-2}\big(F_1(MA(x,n))\big)>c\max_{i}\lambda_{d-2}\big(F_i(MA(x,n))\big)$$ for all $i\leq d-2$ where $\|A(x,n)\|<N$. 
\end{prop}

\begin{lem}\label{lem:smallest move} If $M$ is a matrix of Rauzy induction during freedom on LHS,  $w$ is the smallest singular input direction of $M$, $\omega$ the corresponding singular value and $A$ is a matrix of LHS with $\|A\|<10^{(k+k_0)^3}$ then, if $k_0$ is large enough, 
the smallest singular direction of $MA$ makes angle at most $10^{-\frac 1 {9}(k+k_0)^4}$ with $A^{-1}w$.
\end{lem}
\begin{proof} For any vector $v$ express it as   $v=v'+cw$ where $v'\perp w$. Because $w$ is an input singular direction, $(Mw)^{\perp}=M(w^{\perp})$ and so $|Mv|_2^2=|Mu'|^2+c^2|Mw|_2^2$. Now, if the angle between $v$ and $w$ is at least $10^{-\frac 1 8(k+k_0)^4}$ then \begin{equation}
\label{eq:anglebound}|Mv|>10^{\frac1 3(k+k_0)^4}\omega |v|.\end{equation} This follows from  Lemma \ref{lem:second freedom} and the fact that the smallest singular value of $M$ has size proportional to $\frac 1 {|C_{\max}(M)|}$ and so we have that for any singular value $\sigma\neq\omega$ of $M$, we  have $$\sigma\geq 10^{\frac 1 2 (k+k_0)^4}\omega.$$

Let $u$ be a smallest unit input singular vector of $MA$, and so  $$|MAu|\leq |MA\frac{A^{-1}w}{|A^{-1}w|}|.$$
We claim then  
\begin{equation}\label{eq:smallest move key}
\Theta(w,Au)<10^{-\frac 1 8(k+k_0)^4}.
\end{equation} To prove \eqref{eq:smallest move key} notice first that  our matrix is symplectic so $|Au|\geq \|A\|^{-1}$. Now if \eqref{eq:smallest move key} is false, then by  by the bound on $\|A\|$ 
and (\ref{eq:anglebound})  $$|MAu|>\|A\|^{-1}10^{\frac 1 3 (k+k_0)^4}\omega>\|A\|\omega\geq \left|MA\frac{A^{-1}w}{|A^{-1}w|}\right|,$$ a contradiction proving \eqref{eq:smallest move key}. 
We finish the proof of the Lemma. If $\Theta(u,A^{-1}w)>10^{-\frac 1 9(k+k_0)^4}$, then
 $\Theta(Au,w)>\|A\|^{-2}10^{-\frac 1 9(k+k_0)^4}>10^{-\frac 1 8(k+k_0)^4}$  (for all $k_0$ large enough), a contradiction to  \eqref{eq:smallest move key}  and so we have the lemma.
\end{proof}

\begin{defin} Given a matrix valued cocycle or transpose cocycle $A(x,n)$ we say a subspace $W$ is \emph{left invariant} if $WA=W$ for all $A$ that can occur as matrices of the cocycle. Similarly it is \emph{right invariant} if $AW=W$ for all such $A$. 
\end{defin}

A key tool in the proof of Proposition \ref{prop:big face} is the following theorem which shows that the Rauzy cocycle has few invariant subspaces. 

\begin{thm}(Avila-Viana \cite[Corollary 5.2]{AV}) \label{thm:AV} For any permutation $\pi \in \mathcal{R}_{d-2}$, $v\mathbb{R}$ a one dimensional subspace of $Im(\Omega_\pi)$ 
 and $W$ a codimension 1 subspace of $Im(\Omega_\pi)\color{black}$ we have that there exists a matrix $M$ corresponding to a path  from $\pi$ to $\pi$ so that $Mv\notin W$. 
\end{thm}
This directly follows from Avila-Viana's Theorem that the Rauzy monoid twists subspaces of the preserved subspace of the symplectic form (which is part of the statement that its action on this subspace is \emph{simple}).

We wish to apply this result to matrices of freedom on the LHS in $\mathcal{R}_d$ and subspaces of $e_1\oplus\dots \oplus e_{d-2}$. To do this, note that these matrices have corresponding matrices of $\mathcal{R}_{d-2}$ and the action on the symbols 
1,\ldots, $d-2$ are the same for both of them.
 Let $\pi_0$ be the symmetric permutation for $\mathcal{R}_{d-2}$.
 Let $K_{d-2}$ be the subspace of $e_1\oplus...\oplus e_{d-2} \subset \mathbb{R}^d$ that corresponds to the kernel of 
 $\Omega_{\pi_0}$ in  $\mathcal{R}_{d-2}$.  (This is $\{0\}$ if $d-2$ is even and $(1,-1,...,-1,1,0,0)$ if $d-2$ is odd.)   Similarly let $I_d$ be the image of the symplectic form preserved by Rauzy induction from $\pi_L$ to $\pi_L$  and 
 $I_{d-2}$ be the subspace of $e_1\oplus...\oplus e_{d-2} \subset \mathbb{R}^d$ that corresponds to  the image of $\Omega_{\pi_0}$ in $\mathcal{R}_{d-2}$.

\begin{lem} $W$ is right invariant for the  cocycle $A(x,n)^{-1}$ if and only if  $W$ is right invariant for the transpose cocycle $A(x,n)$. 
\end{lem}
\begin{proof}W is right invariant for $A(\cdot,\cdot)^{-1}$ if and only if  $A(x,n)^{-1}W=W$ for all $x,n$. This is if and only if  $A(x,n)W=W$ for all $x,n$, so $ W$ is right invariant for $A(\cdot,\cdot)$. 
\end{proof}

\begin{lem}\label{lem:few invar}The only nontrivial right $\tilde{A}(\cdot,\cdot)$-invariant subspace contained in $e_2\oplus....\oplus e_d$ is $(e_{d-1}-e_d)\mathbb{R}$. 
\end{lem}
\begin{proof} First, we have that  $I_{d-2}$, $K_{d-2}$ and $(e_{d-1}-e_d)\mathbb{R}$ are invariant subspaces. Let $\mathcal{B}_1$ be a basis for $I_{d-2}$, $\mathcal{B}_2$ be a basis for $K_{d-2}$ and write $A(x,n)$ in the basis $\mathcal{B}_1\cup \mathcal{B}_2\cup \{e_{d-1}-e_d\}\cup \{e_{d-1}+e_d\}$. It suffices to show that all of these matrices have the following form:
\begin{itemize}
\item The columns corresponding to $\mathcal{B}_1$ have  that the entries not corresponding to elements of $\mathcal{B}_1$ are zero.. Moreover, there are no proper right invariant subspaces for the cocycle in this block.
\item The cocycle is diagonal on $(e_{d-1}-e_d)$ and  $K_{d-2}$. 
\item $A(x,n)(e_{d-1}+e_d)=v_n+e_{d-1}+e_d$ where $v_n\in e_1\oplus....\oplus e_{d-2}$ and can have arbitrarily large norm. 
\end{itemize}
To see why the bullets suffice,  the first and second bullets  imply that any invariant subspace non-trivially intersecting $e_1\oplus...\oplus e_{d-2}$ 
 is either $K_{d-2}$, $I_{d-2}$ or $e_1\oplus...\oplus e_{d-2}$. The second and third bullets  say  that $(e_{d-1}-e_d)\mathbb{R}$ is the only non-trivial subspace that only trivially intersects $e_1\oplus...\oplus e_{d-2}$. 
The intersection of an invariant subspace and $e_1\oplus...\oplus e_{d-2}$ is an invariant subspace contained in $e_1\oplus....\oplus e_{d-2}$. If it is non-trivial, it falls in the previous list of such subspaces, none which are contained in $e_2\oplus...\oplus e_{d}$.

The first bullet follows by Avila-Viana's Theorem \ref{thm:AV}. The second is by how our cocycle act on $K_{d-2}$ and the fact that the last two colunms of $A(x,n)$ in the standard basis are always the same. 
The final bullet is because (in the standard basis), $C_j(\tilde{A}(x,1))=C_1(\tilde{A}(x,1))+e_j$ for $j \geq d-1$.
 
\end{proof}

\begin{lem}If there exists a subspace $W$ and a vector $v$ so that $\tilde{A}(x,n)^{-1}v\in W$ for all $x,n$ then there is an $\tilde{A}(\cdot,\cdot)^{-1}$ invariant subspace $V \subset W$ such that $\tilde{A}(x,n)^{-1}v\in V$ for all $x,n$ 
\end{lem}
\begin{proof} First note that by how our cocycle acts for any $x,y,n,m$ there exists $z$ so that $\tilde{A}(z,n+m)^{-1}=\tilde{A}(x,n)^{-1}\tilde{A}(y,m)^{-1}$. This implies that $\tilde{A}(x,n)^{-1}\tilde{A}(y,m)^{-1}v \in W$ for all $x,y,n,m$. This implies that 
$\tilde{A}(x,n)^{-1}v \in \text{span}\{\tilde{A}(y_1,m_1)^{-1}v,...,\tilde{A}(y_k,m_k)^{-1}v\}\subset W$ for all $x,n,y_1,...,y_k,m_1,...,m_k.$ The result follows with $V$ being the span of the images of $v$ under the cocycle.
\end{proof}

\begin{cor}\label{cor:off noninvar} There exists a finite set $\tilde{A}(x_1,n_1),...,\tilde{A}(x_k,n_k)$ so that for any vector $v$ and subspace $W$ if  $v$ is not contained in any invariant subspace of $W$ there exists $\tilde{A}(x_i,n_i)$  such that  
$\tilde{A}(x_i,n_i)^{-1}v\notin W$. Moreover by compactness of projective space, for every $\epsilon>0$ there exists $c'>0$ so that if $\Theta(v,V)>\epsilon$ for all invariant subspaces $V$ of $W$, then $\Theta(\tilde{A}(x_i,n_i)^{-1}v,W)>c'$.
\end{cor}
This follows from the previous lemma by compactness of the space of subspaces minus an $\epsilon$-neighborhood of the invariant subspaces.

\begin{proof}[Proof of Proposition \ref{prop:big face}] Again let $w(M)$ be the smallest singular input direction of $M$ and let $w'(M)$ be the second smallest. 
Fix $\epsilon_0>0$ small. Let $c'$ be the constant given by Corollary~\ref{cor:off noninvar}. 
The proof is split into 2 cases:

\noindent \textit{Case 1:} $\Theta(w(M),e_{d-1}-e_d)>\epsilon_0$. We assume that $k_0$ is large enough so that $c'>9\cdot 10^{-k_0^4}$ and $N<9\cdot 10^{k_0^3}$. 

By Corollary \ref{cor:off noninvar} there exists $x,n$ with $\|A(x,n)\|<N<<10^{(k+k_0)}$ (if $k_0$ is large enough) so that 
$$
\Theta(A^{-1}(x,n)w(M), e_2\oplus \dots \oplus e_{d})>c'
$$
for all $i\leq d-2$ where $\|A(x,n)\|<N$. Indeed, if $w(M)$ is within $\epsilon_0$ of an invariant subspace $U$, we apply the corollary to $W=U\cap e_2\oplus...\oplus e_d$. Note that by Lemma \ref{lem:smallest move} this implies that 
\begin{equation}\label{eq:ang away}\Theta(w(MA(x,n)),e_2\oplus...\oplus e_d)>c'-10^{(k+k_0)^4}>\frac 1 2 c'.
\end{equation}
\color{black}

 Now let $Conv_i$ be the convex hull of $\{\frac{C_1(MA(x,n))}{|C_1(MA(x,n))|},...,\frac{C_{i-1}(MA(x,n))}{|C_{i-1}(MA(x,n))|},\frac{C_{i+1}(MA(x,n))}{|C_{i+1}(MA(x,n)|},...,\frac{C_d(MA(x,n))}{|C_d(MA(x,n))|}\}$, 
 which is $\spanD(C_1(MA(x,n)),...,C_{i-1}(MA(x,n)),C_{i+1}(M(A(x,n)),...,C_d(MA(x,n)))$. 
 So if $\sigma_1\geq....\geq\sigma_d$ are the singular values of $MA(x,n)$ 
 we have  by \eqref{eq:ang away}  that
  $\lambda_{d-2}(Conv_1)$ is at least proportional (in terms of $c'$) to 
 $$\sigma_1\cdot\sigma_2\cdot... \cdot \sigma_{d-1} \zeta^{d-1}  |C_1(MA(x,n))|^{-d+3}|C_d(MA(x,n))|^{-2}$$ 
 which is the largest the volume of a face can be. Indeed, $Conv_i$ is the image of the convex hull of 
 $\{\frac{1}{|C_1(MA(x,n))|},...,\frac{1}{|C_{i-1}(MA(x,n))|},\frac 1 {|C_{i+1}(MA(x,n))|},...,\frac 1 {|C_d(MA(x,n))|}\}$ under the linear action of $MA(x,n)$.

\noindent \textit{Case 2:} $\Theta(w(M),e_{d-1}- e_d)<\epsilon_{0}$. 

Since the largest singular input direction of $M$ is very close to $e_{d-1}+e_d$, it follows that 
\begin{equation}\label{eq:movable} \Theta(w'(M),e_{d-1}\oplus e_d)=\frac {\pi}2-C\epsilon_{0}
\end{equation} 
 for some $C$. 
Indeed, $\omega'(M)$ is in the orthocomplement of $\omega(M)\oplus v$ where $v$ is the largest input singular direction. Since the last two columns are much larger than the first two, $v$ makes small angle with $e_{d-1}\oplus e_d$, establishing 
\eqref{eq:movable}.

Let $A(x,n)$ be given by Corollary \ref{cor:off noninvar} for $w'(M)$ and $W=U \cap e_2\oplus....\oplus e_d$, where $U$ is the smallest invariant subspace $w'(M)$ is contained in. We may apply the corollary because $w'(M)$ makes a definite angle with $e_{d-1}-e_d$ (since it is perpendicular to $w(M)$ which makes a small angle with $e_{d-1}-e_d$) and so by Lemma \ref{lem:few invar}, $U \cap e_2\oplus....\oplus e_d$ is not invariant. We now control the smallest singular direction:
Because $A(x,n)$ acts as the identity on $e_{d-1}-e_d$ and $\|A\|\leq N$, we have 
$$\Theta(w(MA(x,n)),e_{d-1}-e_d)\in [\frac{\Theta(w(M),e_{d-1}-e_d)}{CN},CN\Theta(w(M),e_{d-1}- e_d)]$$
 and so for all $i \leq d-2$ we have 
$$\Theta(w(MA(x,n)),e_1\oplus....\oplus e_{i-1}\oplus e_{i+1}\oplus....\oplus e_d)\in [\frac{\Theta(w(M),e_{d-1}-e_d)}{CN},CN\Theta(w(M),e_{d-1}- e_d)].$$  

From this we have Proposition \ref{prop:big face} with the constant comparable to $(CN)^2(c')^{-1}$. Indeed, as before
$$\lambda_{d-2}(Conv_1)\geq \tilde{c} \sigma_1\cdot ...\cdot \sigma_{d-2}\big(\frac {\Theta(w(M),e_{d-1}\oplus e_d)}{CN}{\sigma_{d-1}} +\sigma_d\big) \zeta^{-d+1}|C_1(MA(x,n))|^{-d+3}|C_d(MA(x,n))|^{-2}$$ which is at least comparable to
$$\lambda_{d-2}(Conv_i)\leq  \sigma_1\cdot ...\cdot \sigma_{d-2}\big(\Theta(w(M),e_{d-1}\oplus e_d){CN}{\sigma_{d-1}} +\sigma_d\big) \zeta^{d-1}|C_1(MA(x,n))|^{-d+3}|C_d(MA(x,n))|^{-2}$$
for $i\leq d-2$. 
\end{proof}

\subsection{Diameters and choice of planes}\label{sec:plane choice}
In this section we will define the family of parallel planes  which we will intersect with simplices to compute diameters,  areas and so forth.  
We will need to discuss  the inverse of a possibly degenerate symplectic matrix $\Omega_{\pi_L}$. 
When we talk about the inverse it is the inverse of $\Omega_{\pi_L}$ restricted to $Im(\Omega_{\pi_L})$ (the orthocomplement of its kernel).  \color{black}

Recall $\Delta_c$ is the subset of $\Delta$ defined by $x_{d-1}+x_d=c$.  Fix $c$.  Let $A_1',B_1$ be the first matrices of restriction on LHS and freedom on the RHS.   Let $u$ be the projection of $\Omega_{\pi_L}^{-1}(C_1(A_1'B_1))$ to $\Delta_c\cap \Omega_{\pi_L}^{-1}C_d(A_1'B_1)^{\perp}.$ 
Let $v$ be the projection of $\Omega_{\pi_L}^{-1}(C_d(A_1'B_1))$ to $\Delta_c \cap \Omega_{\pi_L}^{-1}C_1(A_1'B_1)^{\perp}$.  Consider the plane $P_0$ defined by $$P_0=u\oplus v.$$
Our family of planes $\mathcal{P}$ is the family parallel to $P_0$. 
Since the $\Delta_c$ are parallel for different $c$, $\mathcal{P}$ does not depend on $c$.

In the next theorem  $\omega'$ denotes the second smallest singular value of $M$.

\begin{thm} 
\label{thm:diameter} 
 
 There are positive constants $c_1,c_2$ such that for  $k_0$ is large enough, and $M$ is at the end of freedom  on the left hand side or at the end of freedom on right side  
 \color{black}
then  there is a plane  $P\in\mathcal{P}$ which slices \color{black}
 $M \Delta $ into a polygon $Q$ so that 
 $$\frac{c_1\omega'(M)}{|C_{\max}(M)|}\leq diameter (Q)\leq \frac{c_2}{|C_{min}(M)|^2},$$ 

\end{thm}

We will need a couple of preliminary lemmas.

 \begin{lem} 
 \label{lem:output}
 Let $M$ be as in Theorem~\ref{thm:diameter}.  Then for  any  $w \in \mathcal{P}$, 
  $|M^T\Omega_{\pi_L}w|\geq 
 \frac{|C_{\min}(M)|}{100}|w|$.
 \end{lem}
 \begin{proof} 
 Let $$x=\sum_{i=1}^{d-2}C_i(M) \cdot \Omega_{\pi_L}(u)\  \text{and}\ y=\sum_{j=d-1}^d C_j(M)\cdot \Omega_{\pi_L}(v).$$ Similarly let 
 $$x'= \sum_{j=d-1}^d C_j(M) \cdot \Omega_{\pi_L}(u)\ \text{and}\ y'=\sum_{i=1}^{d-2}C_i(M) \cdot \Omega_{\pi_L}(v).$$ 
  Now by Theorem \ref{thm:nue} we may assume that $\Omega_{\pi_L}u$ is as close as we want to $e_1\oplus...\oplus e_{d-2}$ and similarly for $\Omega_{\pi_L}v$ and $e_{d-1}\oplus e_d$. Moreover, by Condition ** 
   the angles of $C_i(M)$  are close to each other for $i\leq d-2$ and $C_{d-1}(M)$ and $C_d(M)$ are close to each other.
 It follows that   for $k_0$ big enough $$\frac{x}{y'}, \frac{y}{x'}>100.$$ 
It follows that for any $w=\alpha u+\beta v\in P$ that $$|M^T\Omega_{\pi_L}(\alpha u+\beta v)|\geq \max\{|\alpha x+\beta y'|,|\alpha x'+\beta y|\}>\frac 1 2 \max\{\alpha x,\beta y\}.$$ 
\color{black}
The lemma follows from the fact that $x$ and $y$ are at least proportional to $C_{\min}(M)$. Indeed, $\Omega_{\pi_L}u$ is not close to being perpendicular to $C_i(M)$ for $i\leq d-2$ and $\Omega_{\pi_L}v$ is not close to being perpendicular to $C_j(M)$. 
 \end{proof}
 \color{black}
We next  prove a result on relating projective action and linear action. For the purposes of clarity in the next lemma, let $\hat{M}$ denote the projective action of $M$ and $\tilde{M}$ denote its linear action. This is local notation that is only used in this lemma and its proof. Let the pairs (singular value, direction) of $M$ be $(\gamma_1,\theta_1),....,(\gamma_d,\theta_d)$ ordered so that $\gamma_i\geq \gamma_{i+1}$. 
\begin{lem} 
\label{lem:projective}
If $v, w \in \Delta$ are such that  $v-w$ is in the direction $\theta_k$, then 
 $$\frac{\gamma_d|v-w|}{|C_{max}M|}\leq d(\hat{M}v,\hat{M}w)\leq \frac{2\pi \gamma_k}{|C_{min}(M)|}.$$
\end{lem}
\begin{proof} First the lower bound. Let $u_v,u_w \in \mathbb{R}^d_+$ so that $\tilde{M}u_v=\hat{M}v$ and $\tilde{M}u_w=\hat{M}w$. Now $|u_v|,|u_w|\geq  \frac 1 {|C_{max}(M)|}$ and $|u_v-u_w|$ is contracted by at most $\gamma_d$ (by the definition of the smallest singular value).  Then 
$$|\hat M(v)-\hat M(u)|= |\tilde{M}(u_v)-\tilde{M}(u_w)|\geq \gamma_d|u_v-u_w|\geq \gamma_d|v-w|/|C_{max}M|.$$

Now we prove the upper bound. Consider the line $\ell$ through the origin and $\tilde{M}w$. Take the closest point, denoted $t\tilde{M}w$  on $\ell$ to $\tilde{M}v$.  Then $$|t\tilde{M}w-\tilde{M}v|\leq |\tilde{M}v-\tilde{M}w|=\gamma_k|v-w|\leq 2\gamma_k.$$   
We consider the right triangle with vertices at the origin, $t\tilde{M}w$ and $\tilde{M}v$ (the hypotenuse is the line segment from the origin to $\tilde{M}v$). Now we want to know the angle $\psi$ the hypotenuse makes  with  the line from $\tilde{M}v$ to $t\tilde{M}w$. We have $|\tilde{M}v|\geq |C_{min}(M)|$ so $\sin(\psi)\leq \frac{2\gamma_k}{|C_{min}(M)|}$. Since $\psi\leq \frac{\pi}2$  we have 
$\psi\leq \frac {\pi}2{\sin(\psi)}$.  Combining  these inequalities   we have $\psi<\frac{\pi\gamma_k}{|C_{\min}(M)|}$  and changing the angle between two vectors to the distance between the corresponding unit vectors gives us the result after multiplying by an additional factor of 2.
\end{proof}

\begin{proof}[Proof of  Theorem~\ref{thm:diameter}.] 
By Lemma~\ref{lem:output}
 and the fact that $M$ is symplectic, it follows that for $w\in \mathcal{P}$ 
we have 
$$|M^{-1}(w)|=|\Omega_{\pi_L}^{-1}M^T\Omega_{\pi_L}(w)|\geq |C_{min}(M)|\cdot |w|/100.$$   This says that $w$ makes  a definite angle with the space perpendicular to the singular input vectors for $M^{-1}$ with singular value at least $\frac{|C_{\min}(M)|}{200}$ and so makes definite angle with the perpendicular  to the  singular output vectors of $M$ with singular value at most $\frac{200}{|C_{\min}(M)|}$.   
Now the directions of the  plane are fixed
and therefore by the above remark
the diameter of the image under the linear action is bounded above by a multiple of 
 $\frac {1}{|C_{\min}(M)|}$. 
By Lemma~\ref{lem:projective} the image under the projective action has diameter bounded above by
a constant multiple of $\frac{1}{|C_{min}||C_{sing}|}\leq \frac{1}{|C_{min}|^2}$.
\color{black}

To find the plane that intersects in the specified diameter, consider the largest $(d-1)$-dimensional ball one can put in $\Delta$ with center $p$. Let $P\in\mathcal{P}$ be the plane through $M(p)$.
The projective image of the ball under $M$ is an ellipsoid.  The intersection of this ellipsoid with $P$ is an ellipse. 
 The greatest contraction possible for the linear action in the perpendicular subspace to $w_1$ is in the direction of $w_2$.  By Lemma~\ref{lem:projective} the contraction for the projective action is at  most  the second smallest singular value $\omega'$ multiplied by $\frac 1 {|C_{\max}(M)|}$. Thus the   major axis of the ellipse is at least proportional to $\frac{\omega'}{|C_{\max}(M)|}$. 

\end{proof}

\color{black}

\section{Geometry of slices on LHS and Illumination}\label{sec:geom}
 The next two sections are interconnected. The only result in these two sections that is quoted after the end of Section \ref{sec:repeat} is Theorem \ref{thm:big shadow}. Corollary \ref{cor:ready} is used to prove Condition * (1). \color{black}

The next two sections control the geometry of simplices during freedom on LHS. This is used in Section \ref{sec:restrict lhs} to show that even though we lose most of the measure during restriction, we keep enough 
(in all but an exponentially small proportion of planes) to verify the assumptions of Theorem \ref{thm:planes}.
The current section shows that as a first step, if we have a simplex  at the beginning of freedom on LHS we can find a fixed finite collection of subsimplices  which intersect the planes in $\mathcal{P}$ nicely. What this means is that for every point in the subsimplex,  the plane through it intersects the face  $F_1$, which is the image under  the ancestor matrix of the face $\{x_1=0\}\subset  \Delta$.  We say the point is {\em illuminated}.     
 We do this, because under restriction on LHS, the interval $I_1$ always loses so our future simplices lie in a neighborhood of this face. The following  section, Section \ref{sec:repeat}, shows that we may iterate this argument so that all but an exponentially small proportion of the  points that we still have at the end of restriction on RHS have this illumination property.  Then Section \ref{sec:restrict lhs}  deals with restriction. 

Given any freedom LHS matrix $A$ at stage $k+1$, and  sequence $A_1',B_1,B_1',A_2,A_2',\ldots,B_k$ through freedom on RHS at stage $k$, then for each matrix of restriction on RHS $B_k'$ at stage $k$, set $$M_{B_k'}=A_1'B_1B_1'\ldots B_kB_k' A,$$ to be the product of matrices and  set 
$$\mathcal{M}_A=\{ M_{B_k'}\}
$$ to be the collection as $B_k'$ varies. 
Observe that one can order $\mathcal{M}_A$ by ordering the $C_{d-1}(M_{B_k'})$ for $M_{B_k'} \in \mathcal{M}_A$ (they all lie on a line).
In a mild abuse of notation set $$\mathcal{M}_A\Delta=\cup_{M_{B_k'} \in \mathcal{M}_A}M_{B_k'}\Delta \cap (\cup_{c\in (.1,.9)}\Delta_c).$$

\begin{lem}
\label{lem:RHScomparable}
For any two matrices $M_1,M_2\in\mathcal{M}_A$, the column lengths  $|C_d(M_1)|$ and $|C_d(M_2)|$ are uniformly comparable. The same holds  for $|C_{d-1}(M_1)|$ and $|C_{d-1}(M_2)|$. 
\end{lem}

\begin{proof}
After finishing freedom on RHS we have columns $C_{d-1},C_d$.  Then during restriction on RHS  columns  $d-1$ and $d$ are of the form $C_{d-1}$ and $C_d+sC_{d-1}$,  where $s\in [s_k,2s_k]$, (and $s_k=10^{(2k+2+k_0)^6 )+(k+k_0)^4}$.) Thus any pair of  columns $C_d$  have comparable size.  This holds  even after restriction ends on RHS when $C_d$ is added to $C_{d-1}$.  This is also true of columns $C_{d-1}$. Then  during freedom on LHS,  the same vectors are added to both $C_{d-1}$ and $C_d$.  
 Thus they remain comparable in size.
\end{proof}

Now let $\phi$ be the direction of the vector $u$ defined in  Section \ref{sec:plane choice}  as one of the pair of vectors defining the plane $\mathcal{P}$. 
 For any freedom LHS matrix $A$,  
let $$\mathcal{S}_{\phi}(\mathcal{M}_A)=\{y \in \cup_{M \in \mathcal{M}_{A}}M\Delta\ \text{such that  there exists a line in direction}\ \phi\ \text{joining y to a point in}\  F_1(\mathcal{M}_A)\}.$$

We call $\mathcal{S}_\phi(\mathcal{M}_{ A})$ the set {\em illuminated} by  $\mathcal{M}_{A}$.
The point of this definition is that when we move to restriction  on LHS, the first entry does not win  and so our simplices will have a face contained in $F_1$. We want the lines  to intersect those simplices; hence the face $F_1$. \color{black} The main  Theorem below says that except for an exponentially small set, at each stage $k$ we can insure all points are illuminated by 
 some $\mathcal{M}_{A}$ for $$\|A\| \in [10^{(2(k+1)+k_0)^6-(k+k_0)^4},10^{(2(k+1)+k_0)^6-(k+k_0)^4+(k+k_0)^2}].$$
 
 We first  define for $0\leq c\leq 1$, $$\Delta_c=\{x\in\Delta:x_{d-1}+x_d=c\}.$$
This definition will be crucial in the discussion of conflicted sets in next section.

\begin{thm}\label{thm:big shadow} There is $\rho<1$ 
so that for all $k_0$ large enough and $c\in [.1,.9]$\color{black}, given $M$ and $\hat A$ in freedom on LHS at stage $k+1$, there are   matrices $A_1,\ldots, A_\ell$  in freedom on LHS such that
$$\lambda_{d-2}\left(\cup_{j=1}^\ell \mathcal{S}_\phi(\mathcal{M}_{\hat AA_j})\cap \Delta_c\right)>(1-\rho^{(k+k_0)^{1.1}\color{black}})\lambda_{d-2}(\mathcal{M}_{\hat A}(\Delta) \cap \Delta_c).$$
\end{thm}

The Theorem will be proved in two stages; the first step in this section and the second in the next.
 The remainder of this section is devoted to proving  (with preliminaries) Proposition~\ref{prop:ready} which says that we obtain a definite proportion of  $\mathcal{M}_{\hat A}\Delta$ in the illuminated set.


\begin{defin}\label{def:ready}  Given $C_0,\theta_0>0$ 
we say a set of matrices $\mathcal{M}$ is LHS $(\phi,C_0,\theta_0)$\emph{-ready for illumination}  if 
\begin{enumerate}
\item $\cup_{M \in \mathcal{M}}M\Delta$ is a single simplex.
\item  $\lambda_{d-2}\cup_{M \in \mathcal{M}}F_1(M)$ 
is at least $C_0$ times the volume of the largest face of $\cup_{M \in \mathcal{M}}M\Delta$.
\item $\Theta(\cup_{M \in \mathcal{M}}F_1(M),\phi)>\theta_0$.
\end{enumerate}
\end{defin}
 Let $N$ be the maximum of the constant in Proposition \ref{prop:big face} and the norms of the matrices $\tilde{A}(x_i,n_i)$ in Corollary \ref{cor:off noninvar}.  
\begin{prop}
\label{prop:illumination}

There exists $\theta_0,C_0,>0$ so that for any matrix $\hat{A}$ of freedom on LHS 
 with $\|\hat{A}\|< \frac 1 N10^{\constFL\color{black}}$
there exists some $\tilde{A}$ with $\|\tilde{A}\|\leq N$ such that the family $\mathcal{M}_{\hat A\tilde{A}}$ is $(\phi,C_0,\theta_0)$ ready for illumination.  

\end{prop}

\begin{proof}
We first show that for each matrix of freedom on LHS $\tilde{A}$, 
$\mathcal{M}_{\hat A\tilde{A}}$ satisfies condition (1).

Note that since  $\hat A$ is a LHS matrix and  $B_k'$ is a matrix of restriction on RHS,  the first $d-2$ columns are not effected by $B_k'$. This implies that  for $i\leq d-2$, these columns $C_i(A_1'...B_kB_k'\hat A\tilde{A})$ do not depend on $B_k'$. 
The next observation is that for fixed $A_1',\ldots, B_k$, 
$$\cup_{B_k'}\spanD\big(C_{d-1}(A_1'...B_kB_k'),C_d(A_1'...B_kB_k')\big)$$
 is a line segment. Indeed  setting 
$$s_k=10^{(2k+2+k_0)^6+(k+k_0)^4}$$ applying the matrix $B_k'$ means that the $C_{d-1}$ column is added to the $C_d$ column between $s_k$ and $2s_k$ times.   Taking the union over all these possible times gives a line segment 
$$\spanD\left(s_kC_{d-1}(A_1'...B_k)+C_{d}(A_1'...B_k),2s_k
C_{d-1}(A_1'...B_k)+C_{d}(A_1'...B_k)\right).$$\color{black}  Now since $\hat A\tilde{A}$ is a LHS matrix, each time it adds to  $C_j(A_1...B_kB_k')$ for  $j \in \{d-1,d\}$ it adds the same vector  to both. 
So $\cup_{B_k'}\spanD\left(C_{d-1}(A_1'...B_kB_k'\hat A\tilde{A}),C_d(A_1'...B_kB_k'\hat A\tilde{A})\right)$ is the result of adding the same vector to each point on a line segment. That is, it is a line segment 
$$\spanD(s_kC_{d-1}(A_1'...B_k)+C_{d}(A_1'...B_k)+v,(2 s_k+1)C_{d-1}(A_1'...B_k)+C_{d}(A_1'...B_k)+v)$$ for some $v$\color{black} . This proves Condition (1).

We now verify (2).
We apply  Proposition \ref{prop:big face} which says we  may choose $\tilde{A}$, with $\|\tilde{A}\|<N$ so that $\lambda_{d-2}(F_1(\mathcal{M}_{\hat A\tilde{A}}))$ is at least comparable to $\lambda_{d-2}(F_i(\mathcal{M}_{\hat A\tilde{A}}))$ for all $i\leq d-2$.

We now wish to make the comparison for $F_j$ when $j\geq d-1$. 
If the smallest singular direction is not close to $e_{d-1}- e_d$ then this case is covered as in Proposition \ref{prop:big face}  Case 1  and we see that $\lambda_{d-2}(F_1(\mathcal{M}_{\hat A\tilde{A}}))$ is at least comparable to $\lambda_{d-2}(F_j(\mathcal{M}_{\hat A\tilde{A}}))$. 
If it does make a small angle with $e_{d-1}-e_d$, then since the largest singular input direction is close to $e_{d-1}-e_d$, as in the proof of \eqref{eq:movable} 
 the second smallest input singular direction makes small angle with $e_1\oplus...\oplus e_{d-2}$. 
 
So by Corollary \ref{cor:off noninvar} (as in the proof of Proposition \ref{prop:big face}) we may assume the second smallest input direction of $MA(x,n)$ makes an  angle with $F_1(MA(x,n))$ that is bounded away from $0$. 
Now for 
$j\geq d-1$, $F_j(\mathcal{M}_{A(x,n)})=F_j(MA(x,n))$ for some $M \in \mathcal{M}$, while $F_1$ is formed by the union of more than $s_k=10^{(k+k_0)^6}$ 
such simplices. Since the ratio of second smallest to smallest  satisfies for some $C$,  
$$\frac{\omega'(M)}{\omega(M)}\leq  C10^{(2k+k_0)^4},$$  for all $M \in\mathcal{M}$, and $j\geq d-1$
  we have that
$$\lambda_{d-2}(F_1(M\hat{A}A_i))>10^{-(2k+k_0)^4}\lambda_{d-2}(F_{j}(M\hat{A}A_i)).$$ Since $F_1(\mathcal{M}_{\hat{A}A_i})$ is made of $s_k>>10^{(2k+k_0)^4}$
 such subsimplices and $F_j(\mathcal{M}_{\hat{A}A_i})=F_{j}(M\hat{A}A_i)$ for $j\in \{d-1,d\}$ and each $M\in \mathcal{M}$ we have  that 
 $$\lambda_{d-2}(F_1(M\hat{A}A_i))>s_k10^{-(2k+k_0)^4}\lambda_{d-2}(F_j(M\hat{A}A_i))>\lambda_{d-2}(F_j(M\hat{A}A_i)).$$

We verify  the third condition of ready for illumination.
We chose $\phi$ to be the vector $u$ which was defined to be the projection of $\Omega_{\pi_L}^{-1}(C_1(A_1'B_1))$ to $\Delta_c\cap \Omega_{\pi_L}^{-1}(C_d(A_1'B_1))^\perp$.  For $k_0$ large enough,  by Theorem \ref{thm:nue}, under Condition ** \color{black}  $u$ is close in angle to the projection of $\Omega_{\pi_L}^{-1}(C_1(M))$. By  Proposition \ref{prop:closeangle}  and (\ref{eq:Viana})
 this direction  itself makes small angle with  $w'$, the second smallest output vector of $M$.  Since   by Corollary \ref{cor:off noninvar}, we may assume  $F_1(\cup_{i=1}^l\mathcal{M}_{\hat A\tilde{A}})=M(e_2\oplus....\oplus e_d)\cap \Delta$ makes a definite angle  with either the smallest or second smallest output direction of $M$, we conclude that it makes a definite angle with $u$.

Lastly, we need that the product, $\hat{A}\tilde{A}$, is a matrix of freedom on LHS at step $k+1$. By our choice of $N$ we have that $\|\hat A\tilde{A}\|\leq \|\hat A\|\cdot \|\tilde{A}\|\leq N\|\hat{A}\|$ and so by our bound on $\|\hat{A}\|$ the product is a matrix of freedom on LHS. 
\end{proof}

\begin{prop}
\label{prop:ready}
For all large enough $\zeta>0$, 
 there exists $C>0$ so that given matrix $\hat{A}$ of freedom on LHS satisfying the bound assumption  in Proposition~\ref{prop:illumination} and also satisfies  $\frac 1 \zeta< \frac{|C_i(M\hat{A})|}{|C_{i'}(M\hat{A})|}<\zeta$ for $1\leq i,i'\leq d-2$, then 
$\lambda_{d-1}\mathcal{S}_\phi({\mathcal{M}}_{\hat A\tilde{A}})>C\lambda_{d-1}(\mathcal{M}_{\hat A}\Delta)$,
where $\tilde A$ is given by  Proposition~\ref{prop:illumination}. 
 \end{prop}

 \begin{proof}[Proof of Proposition \ref{prop:ready}] 
 Proposition~\ref{prop:illumination}  says $\lambda_{d-2}F_1(\mathcal{M}_{\hat A\tilde{A}})$ is comparable to $\lambda_{d-2}V(\mathcal{M}_{\hat A\tilde {A}})$.  

Since $\|\tilde A\|\leq N$ and 
$\frac 1 \zeta<\frac{|C_i(M\hat{A})|}{|C_{i'}(M\hat{A})|}<\zeta$  for all $1\leq i,i'\leq d-2$,   Lemma \ref{lem:face jacobian}  implies     $\lambda_{d-2}V(\mathcal{M}_{\hat A\tilde {A}})$ is comparable to  
 $\lambda_{d-2}V(\mathcal{M}_{\hat A})$.
Now  
the first conclusion of Theorem \ref{thm:most lhs} says that there is a constant $C$ such that  $
\lambda_{d-3}\left(V({\mathcal{M}}_{\hat A\tilde{A}})\right)>\delta
V(\mathcal{M}_{\hat A})$ implies $$\lambda_{d-1}\left(({\mathcal{M}}_{\hat A\tilde{A}})\Delta\right)>\frac{\delta}{C}\lambda_{d-1}(\mathcal{M}_{\hat A}\Delta).$$ \color{black}

Since  the direction $\phi$ makes an angle bounded away from the face $F_1$, we conclude  that the measure of $\mathcal{S}(\mathcal{M}_{\hat{A}\tilde{A}})$ is a definite proportion of the measure of $\mathcal{M}_{\hat{A}\tilde{A}}\Delta$ and hence of  the measure of $\mathcal{M}_{\hat{A}}\Delta$ which  finishes the proof of Proposition~\ref{prop:ready}.
\end{proof}
By Theorem \ref{thm:nue} we may assume that $V(\mathcal{M}_{\hat{A}}) \subset \cup_{c<.05}\Delta_c$ and $W(\mathcal{M}_{\hat{A}})\subset \cup_{c>.95}\Delta_c$. Since $F_1(\mathcal{M}_{\hat{A}})$ is a simplex, it follows that the volume of $\Delta_c\cap F_1(M)$ is comparable for all $c\in [.1,.9]$. Since our directions are parallel to $\Delta_c$ we obtain analogously to above:
\begin{cor} 
\label{cor:ready} There exists  $\zeta',K$ and 
 $\hat{c}>0$ so that for all $c\in [.1,.9]$ and large enough $k_0$, for any matrix $\hat{A}$ of freedom on LHS, with $\frac 1 \zeta< \frac{|C_i(M\hat{A})|}{|C_{i'}(M\hat{A})|}<\zeta$ for $1\leq i,i'\leq d-2$
 and $\|\hat{A}\|< \frac 1 K10^{\constFL\color{black}}$ 
  there exists  matrix 
 $\tilde {A}$ so that 
 \begin{enumerate}
\item\label{cond:still freedom} $\hat{A}\tilde{A}$ is a matrix of freedom on LHS
\item
$\lambda_{d-2}\left(\mathcal{S}_\phi(\mathcal{M}_{\hat A\tilde{A}}) \cap \Delta_c\right)>\hat{c}\lambda_{d-2}(\mathcal{M}_{\hat A}\Delta \cap \Delta_{c})$
\item $\frac{|C_{i}(M\hat{A}A_j)|}{|C_{i'}(M\hat{A}\tilde{A})|}<\zeta'$ for all $i,i'\leq d-2$. 
\end{enumerate}
\end{cor}

The third condition is the only condition that is not immediate. This follows because $\hat{A}$ is $\zeta$-balanced and the matrices in $\mathcal{A}$ have norm at most $N$. (In particular, we can let $\zeta'$ be $N\zeta$ with $N$ as before Definition \ref{def:ready}.)

\section{Repeating to illuminate}\label{sec:repeat}
The result in the last section says a definite proportion of $\mathcal{M}_{\hat A}$ is illuminated.  We have to improve this to all but $\rho \color{black}^{(k+k_0)^{\frac 1 2 }}$ of our simplex is illuminated for some $\rho<1$. We begin with a lemma. 

\begin{lem}  There is $b=b(d) \in \mathbb{N}$ so that $\mathcal{S}(\mathcal{M_{\hat A}})$ is a convex set bounded by at most $b$ faces. 
\end{lem}
\begin{proof} The simplex $\mathcal{M}_{\hat A}\Delta$ is a convex set bounded by $d$ faces. 
 The illuminated set  is formed by intersecting it with the affine subspaces through the faces of $F_1$ (which are codimension 2 in $\mathcal{M}_{\hat A}\Delta$) that also contain the direction $\phi$. 
\end{proof}

This lemma implies that $\mathcal{S}(\mathcal{M_{\hat A}})$ is the intersection of at most $b$ half spaces each of which is bounded by a hyperplane $H$. This motivates us to study such regions intersected with $\mathcal{M}_{\hat A}(\Delta)$ and $\mathcal{M}_{\hat A}\Delta \cap \Delta_c$. 
For $H$ a bounding hyperplane let $H_c=H \cap \Delta_c$.

 Let $$W_{\hat A}=
   \cup_{M \in \mathcal{M}_{\hat A}}\spanD(C_{d-1}(M),C_d(M)).$$  
 In other words 
 $W_{\hat A}$ is the  line segment such that the span of it with $V(\mathcal{M}_{\hat A})$  is   $\mathcal{M}_{\hat A}\Delta$.  
  Consequently for each $p\in W_{\hat A}$,  with $p \neq C_j(M)$ for all   $j\geq d-1$ and $M \in \mathcal{M}_{\hat{A}}$,  there is a unique $M\in\mathcal{M}_{\hat A}$ such that $p\in \spanD(C_{d-1}(M),C_d(M))$

We say  $M\in\mathcal{M}_{\hat A}$ is {\em conflicted}  if there is $p\in \spanD(C_{d-1}(M),C_d(M))$
 and $x\in V(\mathcal{M}_{\hat A})$ such that $\spanD(x,p)\cap H_c\neq\emptyset$.
Let
$\mathcal{W}_c(\mathcal{M}_{\hat A})$ be the set  of conflicted $M$. 

For each  $M\in\mathcal{M}_{\hat A}$ 
we can define a function on $V(M)$ to   subsets of $W_{\hat A}$ by 
$$\beta_{c,\hat A}(x)=\{p\in W_{\hat A}:\spanD(x,p)\cap H_c\neq\emptyset\}$$
if this set is nonempty. Otherwise $\beta_{c,\hat A}(x)$ is not defined.  
Now for each $p\in W_{\hat A}$ let   $$X_c(p)=\{x\in V(M\hat{A}): \beta_c(x)=p\}.$$
 \begin{lem}\label{lem:1interval} For any subinterval $I\subset W_{\hat A}$  we have $\cup_{p\in I} X_c(p)$ is  connected.
 \end{lem}

 \begin{proof} It suffices to prove that for each line segment $\vec{\ell}$ it is the case that   $\vec{\ell} \cap \cup_{p\in I} X_c(p)$ is connected. This follows from  the fact that on  line segments $\beta_{c,\hat A}$ changes monotonically. 
 \end{proof}

For the next definition, recall that $V(M\hat A)$ is independent of $M \in \mathcal{M}$.
Now  given $\hat A, K,\zeta,\tau$  and $\ell$, 
let $G(\hat A,K,\zeta,\ell,\tau)$ be the set of $y \in V(M\hat A)$  such that there exists $m$   so that for some  $M\hat A\in\mathcal{M}_{\hat A} $, with $M\hat{A}=M(x,r)$  
\begin{itemize}
\item $\|A(R^ry,m)\|<K^\ell$,
\item $ |\mathcal{W}_c(\mathcal{M}_{\hat AA(R^ry,m)})|<\max\{\frac 2 3 ^{-\frac{\tau}{2}\ell}|\mathcal{M}_{\hat A}|,2\}$
\item $\frac{|C_j(M\hat AA(R^ry,m))|}{|C_i(M\hat AA(R^ry,m))|}<\zeta$ for all $1\leq i,j\leq d-2$
\item the permutation of $R^ry$ is $\pi_L$
\end{itemize}

\begin{prop}\label{prop:shadow technical}There exists $K,\zeta,C>1,\tau$ and $\rho<1$ so that
for all $\ell$ and $\hat A$ \color{black}

$$\lambda_{d-3}\big(G(\hat A,K,\zeta,\ell,\tau)\big)>(1-C\rho^\ell) \lambda_{d-3}(V(M\hat A))$$
\end{prop}

\begin{proof}

 We will  need the following family of  paths.
    Let $\pi'$ be the permutation   on the LHS  $$\begin{pmatrix}1&3&\dots &d&2\\d&d-1&\dots&2&1\end{pmatrix},$$
one step before $\pi_s$.   We consider paths $\gamma(x,n)$ of permutations  of length $n$ starting at $x$ with permutation   $\pi_L$ that go through $\pi'$ and return to $\pi_s$ in one step with $1$ beating $2$, and this is the only  time of going from $\pi'$ to $\pi_s$.  We call these $\pi_s$ via $\pi'$ {\em isolated}. The point of this definition is that no columns are added to the last two columns. 
Thus the interval  $W_{\hat A}$ does not change.  
Let $M_{mid}(\hat A)\in\mathcal{W}_c(\mathcal{M}_{\hat A})$ be chosen so that  
\begin{multline} \label{eq:mid division}
 |\{M\in \mathcal{W}_c(\mathcal{M}_{\hat A}):M \geq M_{mid}(\hat A)\}|-1\leq
 \\ |\{M\in \mathcal{W}_c(\mathcal{M}_{\hat A}):M< M_{mid}(\hat A)\}|\leq |\{M\in \mathcal{W}_c(\mathcal{M}_{\hat A}):M\geq M_{mid}(\hat A)\}|.
 \end{multline}

Note $\beta_c^{-1}(C_d(M_{mid}(\hat A))\subset V(M\hat A)$ is a hyperplane

Fix  $K'$ to be determined later.  For any $n$, let $E_n$ be the set of $y \in V(M\hat A):\exists p \text{ so that }$
\begin{itemize}
\item The permutation $R^{p+r-2}y$ is  $\pi_s$ and  $R^{p+r}y$ is  $\pi_L$ 
\item $\frac{|C_j(\hat AA(R^ry,p))|}{|C_i(\hat AA(R^ry,p)|}<\zeta$ for $1\leq i,j\leq d-2$
\item $\|A(R^ry,p)\|\in [K'^{2n},K'^{2n+1}]$
\end{itemize}
Now given $E_n$,  let $\hat{E}_n$ be the set of $ x  \in E_n:\exists m $ so that $x\in V(M\hat{A}A(y,p))$ 
with $y,p$ as in the definition of $E_n$ and so that 
\begin{enumerate}[label=\roman*]
\item the path $\gamma(m+2,R^{r+p-2} x  )$ is $\pi_s$ via $\pi'$ isolated.
\item $\|A(R^{p+r}x  ,m)\|<K'$
 \item  \label{conc:avoid mid}  $\beta_{c,\hat A}^{-1}(C_d(M_{mid} (\hat AA(R^rx \color{black} ,p))$ 
 and $V(M\hat AA(R^rx  ,p+m))$ are disjoint

  \color{black}
 
\end{enumerate}

If $K'$ is large enough, then by  Lemma~\ref{lem:bal often}   there exists $\tau_1>0$ so that given any outcome of $\hat{E}_1,...,\hat{E}_{i-1}$, the conditional probability is at least $\tau_1$ that $ x \in E_i$.

We now show that there exists $\tau_2>0$  so that for $K''$ large enough, given  the  hyperplane 
 $\beta_c^{-1}(C_d(M_{mid}(\hat AA(R^r x  ,p))$, the conditional probability that  $y\in \hat{E}_i$,  given any outcomes of $\hat{E}_1,....,\hat{E}_{i-1}$ and $x  \in E_i$ 
is at least $\tau_2>0$. Indeed, we apply Corollary \ref{cor:avoid hyp} in the appendix to find the matrices that avoid $\beta_c^{-1}\bigl(C_d(M_{mid}(\hat{A}A(R^r x  ,p))\bigr)$. 
Now because $ x  \in \hat{E}_i$, Lemma \ref{lem:face jacobian} says that this is a definite proportion. 

We now let
$$F_i(x )=\begin{cases}1 & \text{if } x  \in \hat{E}_i\\ 0& \text{otherwise}\end{cases}.$$ Let $$\tau=\tau_1\tau_2$$ and let $$\tilde{K}=\max\{K',K''\}.$$  Finally let  $K=
\tilde{K}^{3}$.
 We apply Proposition~\ref{prop:prob decay} with  $\epsilon=\frac{\tau}{2}$. 
 There exists $\rho<1$ so for each $\ell$, for all but a percentage of at most  $C\rho^\ell$  of the points $x  \in V(M\hat A)$  
 have the property that they belong to at   least $\ell \frac{\tau}{2}$  of the sets $\hat E_n$ with
$n\leq \ell$.

Now suppose $ x   \in \hat{E}_n$. 
Then by definition 
 $\max_{j\leq d-2} |C_j(A(R^r x \color{black} ,p))|<\tilde{K}^{2n+1}$ and there exists $m$ so that $\max_{j\leq d-2}|C_j(A(R^{r+p} x \color{black} ,m))|<\tilde{K}$.  We claim  
\begin{equation}\label{eq:numbers}|\mathcal{W}_c(\mathcal{M}_{\hat{A} A(R^r x \color{black} ,p+m)})|\leq \lceil \frac1 2 |\mathcal{W}_c(\mathcal{M}_{\hat{A}A(R^r x \color{black} ,p)})|\rceil\leq  \frac1 2 |\mathcal{W}_c(\mathcal{M}_{\hat{A}A(R^r x \color{black} ,p)})| +1.
\end{equation}
To see this, by Conclusion (\ref{conc:avoid mid}) we have that  $$\beta_{c,\hat A}^{-1}(C_d(M_{mid}(\hat AA(R^r x \color{black} ,p)))\cap V(A(R^rx \color{black} ,p+m)=\emptyset.$$  By \eqref{eq:mid division} \color{black} all the $C_d$ on one side of 
$C_d(M_{mid}(\hat AA(R^rx \color{black} ,p)))$
have the property that $$\beta_{c,\hat A}^{-1}(C_d)\cap V(A(R^rx \color{black} ,p+m))=\emptyset.$$ 
By our choice of $M_{mid}$  \eqref{eq:numbers} follows.
Now for any $A$, if $M\notin \mathcal{W}_c(\mathcal{M}_{\hat A})$ then $M \notin \mathcal{W}_c(\mathcal{M}_{\hat AA})$,   
Given any $\ell$ we have shown that except for a set of $x  \in V(M\hat A)$ of measure $C\rho^\ell\lambda_{d-3}V(M\hat A)$, for at least $\frac{\tau}{2}\ell$ values of $n\leq \ell$ we have   $x \color{black} \in \hat{E}_n$.  For such  $ x \color{black} $ the corresponding matrix satisfies $|C_{\max}(A(R^{r+p} x  ,m))|<K'^{2n+1}\leq K^\ell$.  Furthermore 
\begin{equation}\label{eq:cardinality}|\mathcal{W}_c(\mathcal{M}_{\hat AA(R^rx,n)})|\leq 2^{-\ell\frac{\tau}{2}}|\mathcal{W}_c(\mathcal{M})|+2
\end{equation}

This finishes the proof. 

\color{black}
\end{proof}
Based on the last result we make the following definition. 
 Given   $\hat A,k,k_0,N_0$   
let $H_{k+k_0}(\hat A,N_0)$ be the set of $y\in V(M\hat A)$ such that there exists $A(R^ry,m)$ so that 
\begin{enumerate}[label=(\roman*)]
\item\label{cond:H bal} $\frac{|C_{i'}(M\hat AA(R^ry,m))|}{|C_i(M\hat AA(R^ry,m))|}<\zeta$ for $1\leq i,i'\leq d-2$
\color{black}
\item\label{item:conflict bound} $|\mathcal{W}_c\mathcal{M}_{\hat AA(R^ry,m)}|<(\frac  23)^{(k+k_0)^{1.1\color{black}}}|\mathcal{M}_{\hat A}|+2$
\item $\|A(R^ry,m)\|<N_0^{(k+k_0)^{1.1\color{black}}}.$
\item  the permutation of $R^ry$ is $\pi_L$.
\end{enumerate}

\begin{cor}\label{cor:choose size}
There is $C, N_0$, $\rho'$ such that for all $k,k_0$   

\begin{equation}
\label{eq:mostsmallshadow}
\lambda_{d-3}(H_{k+k_0}(\hat A,N_0))\geq (1-C\rho'^{(k+k_0)^{1.1 }})\lambda_{d-3}V(M\hat A).
\end{equation}

\end{cor}

\begin{proof}
   Let $ K,\, \tau,\, \rho$ be a triple so that Proposition \ref{prop:shadow technical} is satisfied with this triple and some $ C,\, \zeta$. 
Choose  $\rho'=\rho^{\frac{2}{\tau}}$. 
 This choice says that  for any $k,\, k_0$, if we set $\ell=\frac{2}{\tau}(k+k_0)^{1.1\color{black}}$, then $\rho^\ell=\rho'^{(k+k_0)^{1.1\color{black}}}.$

Similarly, choose $N_0$ so that if $\ell=\frac {2}{\tau}(k+k_0)^{1.1}$, then  $N_0^{(k+k_0)^{1.1\color{black}}}>K^\ell$ (so $N_0=K^{\frac 2\tau}$).
Apply Proposition~\ref{prop:shadow technical} with $\ell=\frac{2}{\tau}(k+k_0)^{1.1\color{black}}$ and we obtain the Corollary. 
     \end{proof}
     

Next let 
$\mathcal{H}_{k+k_0}(\hat{A})$ 
be the set of matrices of the form $A(R^\ell y,m)$ where  $y \in H_{k+k_0}(\hat A,N_0)$ and $m$ is as above. 
These are matrices that give a small conflicted set of simplices.

The next lemma says that $\mathcal{H}_{k+k_0}(\hat{A})\color{black}\Delta$ covers most of $\mathcal{M}_{\hat A}\Delta$ and that 
for matrices in $\mathcal{H}_{k+k_0}$ the corresponding conflicted matrices only cover a set of small measure.

\begin{lem}
There is $\hat{C}'$  such that  for all $k$ 
\color{black}
\begin{multline}\label{eq:next one}\lambda_{d-1}(\mathcal{M}_{\hat A}\Delta \setminus \mathcal{H}_{k+k_0}(\hat A)\Delta)+
\lambda_{d-1}\bigl(\cup_{A\in \mathcal{H}_{k+k_0}(\hat A)}(\mathcal{W}_c(\mathcal{M}_{\hat AA})\color{blue}\hat A A\color{black}\Delta )\bigr)
\\<\hat{C}'(\rho'^{(k+k_0)^{1.1\color{black}}} +10^{-(k+k_0)^{1.1\color{black}}}
+(\frac  23 ) ^{(k+k_0)^{1.1\color{black}}})\lambda_{d-1}(\mathcal{M}_{\hat A} \Delta).
\end{multline}
\end{lem}
\begin{proof}

First of all applying (\ref{eq:mostsmallshadow})  we find $\rho',C,N_0$ such that the simplices corresponding to matrices  $A\in\mathcal{H}_{k+k_0}(\hat A)$
 cover a subset  of $V(M\hat A)$ whose complement has measure at most $C\rho'^{(k+k_0)^{1.1\color{black}}}$.  Each such 
 $A$ satisfies $\|A\|\leq N_0^{(k+k_0)^{1.1\color{black}}}$.   Moreover for each such $A$ and  for $i\leq d-2, j\geq d-1$, and $k_0$ large enough, we have $$\frac{N_0^{(k+k_0)^{1.1\color{black}}}}{10^{-2(k+k_0)^{1.1\color{black}}}}\leq   
10^{(k+k_0)^{2.3\color{black}}}\leq \frac{|C_j(A)|}{|C_i(A)|}$$ so we can apply 
the second conclusion of Theorem \ref{thm:most lhs} with $N=N_0^{(k+k_0)^{1.1\color{black}}}$, $\epsilon =10^{-(k+k_0)^{1.1\color{black}}}$ and $\delta=\rho'^{(k+k_0)^{1.1\color{black}}}$.
This bounds the first term on the left by the first two terms on the right in \eqref{eq:next one}.

We now bound  the second term on the left. First we note  that the bound   in  
\ref{item:conflict bound}
\color{black} on  the cardinality of $\mathcal{W}_c(\mathcal{M}_{\hat AA})$ in the definition of $\mathcal{H}$ in says that our conflicted set has a small number of simplices compared to the non-conflicted set. To obtain a measure estimate, we apply Lemma~\ref{lem:RHScomparable} and Veech's volume estimate (Lemma \ref{lem:volume}) which together say that the $\lambda_{d-1}$ volume of the different $\spanD(V(M), C_{d-1}(M), C_d(M))$ as $M$ varies in $\mathcal{M}_{\hat A}$  are uniformly comparable,   to conclude that 
$$\lambda_{d-1}\bigl(\cup_{A\in \mathcal{H}_{k+k_0}(\hat A)}(\mathcal{W}_c(\mathcal{M}_{\hat AA})\hat A A\Delta )\bigr)\leq (\frac  23 ) ^{(k+k_0)^{1.1}}$$ 
\color{black}
\end{proof}
\color{black}

In the next lemma we take the estimates of the last lemma and intersect with the sets $\Delta_c$.

\begin{lem}\label{lem:next two}
There are  $N_0,C,\rho''$, so that for $k_0$ large enough,  for all $k$ \color{black}
and  
for each $c$ satisfying $.1\leq c\leq .9$ \begin{multline}\label{eq:next two}\lambda_{d-2}\left(\mathcal{M}_{\hat A}\Delta \cap \Delta_c \setminus \mathcal{H}_{k+k_0\color{black}}(\hat A)\Delta \cap \Delta_c\right)+
\lambda_{d-2}\left(\cup_{A\in \mathcal{H}_{k+k_0\color{black}}(\hat A)}\mathcal{W}_c(\mathcal{M}_{\hat AA}\color{blue}A\color{black})\Delta \cap \Delta_c\right)\\<C(\rho''^{(k+k_0)^{1.1\color{black}}} +10^{-(k+k_0)^{1.1\color{black}}}+(\frac  23\color{black}) ^{(k+k_0)^{1.1\color{black}}})\lambda_{d-2}(\mathcal{M}_{\hat A} \Delta \cap \Delta_c).
\end{multline}
\end{lem}
\begin{proof}
By Theorem \ref{thm:nue}, for $k_0$ large enough  
for $M\in \mathcal{M}_{\hat A}$ and for $1\leq i\leq d-2$,  the columns $C_i(M)$ lie within $.05$ of  $e_1\oplus\ldots  \oplus e_{d-2}$ and the last two columns lie within $.05$ of $e_{d-1}\oplus e_d$.
Thus $V(\mathcal{M}_{\hat A})\subset \cup_{s\in [0,.05]}\Delta_s$ and $\mathcal{W}_{\hat A} \subset \cup_{s\in [.95,1]}\Delta_s$. 
 Thus, for any $s,s'\in [.1,.9]$  for each  $M\in\mathcal{M}_{\hat A}$ we have that 
$$\frac{\lambda_{d-3}\big(\spanD(V(M),p)\cap \Delta_s\big)}{\lambda_{d-3}\big(\spanD(V(M),p)\cap \Delta_{s'}\big)}<C,$$ where $C$ depends only on the dimension.  Indeed, $\spanD(V(M),p)$ is a simplex and $\Delta_s$, $\Delta_{s'}$ are parallel planes that are at least $\frac 5 {100}$ of the diameter of the simplex away from any of the extreme points of the simplex. We now have the bound 
\begin{equation}\label{eq:intersect bound}\frac{\lambda_{d-2}(M\Delta \cap \Delta_s)}{\lambda_{d-2}(M\Delta\cap \Delta_{s'})}=\frac{\lambda_{d-2}(\cup_{p\in W_A}\spanD(V(M),p) \cap \Delta_s)}{\lambda_{d-2}(\cup_{p\in W_A}\spanD(V(M),p)\cap \Delta_{s'})}<C'
\end{equation}
 by Fubini's theorem. Indeed, by Fubini's theorem, there exists $c_s$ so that 
$\lambda_{d-2}(M\Delta \cap \Delta_s)=c_s\int_{W_a} \lambda_{d-3}\big(\spanD(V(M),p)\cap \Delta_{s})dp$. Moreover, since $W_A,V(M) \subset \cup_{t\in [0,.05]\cup [.95,1]}\Delta_t$, we have that $c_s$ changes polynomially as $s$ ranges from $.95$ to $.05$ and so $c_s,c_{s'}$ are comparable for all $s,s'\in [.1,.9]$. 
Applying \eqref{eq:intersect bound} to each simplex in \eqref{eq:next one} we obtain the lemma. \color{black} 
\end{proof}


Applying  Equation (\ref{eq:next two}) to each hyperplane that cuts out the illuminated set  gives 
\begin{multline}\label{eq:shadow} \lambda_{d-2}\Big(\mathcal{M}_{\hat A}\Delta  \cap \Delta_c\setminus \big(\mathcal{S}_\phi(\mathcal{M}_{\hat A})  \cup_{A \in \mathcal{H}_{k+k_0}(\hat A)} \mathcal{W}_{c}(\hat AA)\big)\cap\Delta_c\Big)\\<Cb(d)((\frac  23 \color{black}) ^{(k+k_0)^{1.1\color{black}}} +\rho''^{(k+k_0)^{1.1\color{black}}} +10^{-(k+k_0)^{1.1\color{black}}})\lambda_{d-2}(\mathcal{M}_{\hat A}\Delta\cap \Delta_c).
\end{multline}

\begin{proof}[Proof of Theorem \ref{thm:big shadow}] We are at stage $k$  but suppress  it in the definition of matrices. 
We restrict our attention to $\alpha \in [.1,.9]$ and we are given the constants $\hat{c}$ and $K$  from Corollary~\ref{cor:ready}.

Choose $N$ so that $$(1-\hat{c})^N<\frac 1 {10},$$  
We choose  $k_0$ large enough to guarantee  that for all $k$ and $N_0$ as in Corollary \ref{cor:choose size}, 

\begin{equation}\label{eq:k0 big}
(N_0^{(k+k_0)^{1.1\color{black}}}+K)^{N(k+k_0)^{1.1\color{black}}}< \frac 1 2 10^{(k+k_0)^{2.3\color{black}}}.
\end{equation}
\color{black}

\color{black}
The proof is by an inductive procedure.   For the first step  we are given $\hat A$ a LHS matrix with $\|\hat{A}\|\in [10^{(2k+2+k_0)^6-(2k+1+k_0)^4}, 2\cdot 10^{(2k+2+k_0)^6-(2k +1+k_0)^4}]$, and a set of matrices $\mathcal{M}_{\hat A}$.  We apply Corollary~\ref{cor:ready} to produce   matrix $A_1$ so that  $\mathcal{M}_{\hat AA_1}$ is ready for illumination.  
For a constant $\hat{c}>0$, the simplices   $\mathcal{M}_{\hat AA_1}\Delta \cap \Delta_\alpha$ satisfy  that $X_1:=\cup_\alpha\mathcal{S}(\mathcal{M}_{\hat{A}A_1}) \cap \Delta_\alpha$ is a set of measure at least $\hat{c}\lambda_{d-1}(\mathcal{M}_{\hat A}\Delta \cap \Delta_\alpha)$.  

 Now  let
 $$W_1:= \left((\mathcal{M}_{\hat{A}A_1}\Delta \cap \Delta_\alpha) \setminus (\cup_{A \in \mathcal{H}_{k+k_0}(\hat{A}A_1)}\mathcal{M}_{\hat{A}A_1A}\Delta)\right) \cup \cup_{A \in \mathcal{H}_{k+k_0}(\hat{A}A_{1,j})}\mathcal{W}_\alpha(\mathcal{M}_{\hat{A}A_1A}\Delta \cap \Delta_\alpha).$$ 
 
 This is the set  not covered by simplices of matrices $A$ with a small conflicted set  union the conflicted subset of those $A$ that do have a small conflicted set.
 By  Lemma \ref{lem:next two},  $X_1 \setminus W_1$ has measure at least  
$$\hat{c}\big(1-C(\rho''^{(k+k_0)^{1.1\color{black}}} +10^{-(k+k_0)^{1.1\color{black}}}+(\frac  23\color{black}) ^{(k+k_0)^{1.1\color{black}}})\big)\lambda_{d-1}(\mathcal{M}_{\hat A}\Delta \cap \Delta_\alpha).$$
\color{black}


We next consider  $\mathcal{M}_{\hat A}\Delta \setminus ( X_1\cup W_1)$.  Notice that it is disjoint from the illuminated set,  and it can be 
partitioned   into simplices of the form  $\mathcal{M}_{\hat A\tilde{A}_\beta}\Delta$ where  $\|\tilde{A}_\beta\|<(N_0^{(k+k_0)^{1.1\color{black}}}+K).$
We now again have a collection of families $\mathcal{M}_{{\hat A\tilde{A}}_{\beta}}$.
  Using  Corollary~\ref{cor:ready}  we first make each of them ready for illumination by a matrix $A_2$. In this way we obtain an illuminated subset 
  $$X_2\subset \mathcal{M}_{\hat A}\Delta \cap \Delta_\alpha \setminus (X_1\cup W_1)$$ of measure at least 
  $\hat{c}\lambda_{d-2}\big(\mathcal{M}_{\hat A}\Delta \cap \Delta_\alpha \setminus (X_1 \cup W_1)\big)$. 
  As before, to each $A_2$ we obtain the set $W_2$ not covered by simplices of matrices $A$ with small conflicted set union the conflicted subset of those $A$ that do have a small conflicted set.  We remove $X_2,W_2$ and
  repeat this procedure $N(k+k_0)^{1.1\color{black}}$  total times, constructing  disjoint illuminated sets $X_j$ and removed sets $W_j$. We now show that 
  $$\cup_{i=1}^{N(k+k_0)^{1.1\color{black}}}X_i \setminus W_i = \sqcup_{i=1}^{N(k+k_0)^{1.1\color{black}}} X_i \setminus W_i $$ 
is covered by simplices from matrices of freedom on LHS.  First note that since we performed a procedure $N(k+k_0)^{1.1\color{black}}$ times that increased the norm by at most $N_0^{(k+k_0)^{1.1\color{black}}}+K$, our assumption on $\|\hat{A}\|$ and Inequality (\ref{eq:k0 big}) implies that our set is contained in the matrices of freedom on the left hand side  (assuming $k_0$ is large enough) at step $k$.

  We now bound $\lambda_{d-2}(\sqcup_{i=1}^{N(k+k_0)^{1.1\color{black}}} X_i \setminus W_i)$ from below by first bounding $\lambda_{d-2}(\sqcup_{i=1}^{N(k+k_0)^{1.1\color{black}}}W_i)$ from above and then bounding 
  $\lambda_{d-2}(\mathcal{M}_{\hat{A}} \Delta \cap \Delta_\alpha \setminus (\cup_{i=1}^{N(k+k_0)^{1.1\color{black}}}X_i \cup W_i))$
from above.  Now 
  \begin{eqnarray}\label{eq:small toss}\nonumber \sum_{i=1}^{N(k+k_0)^{1.1\color{black}}}\lambda_{d-2}(W_i)& &\\  \nonumber& \leq &
(\rho''^{(k+k_0)^{1.1\color{black}}}+(\frac{2}3\color{black})^{(k+k_0)^{1.1\color{black}}}+10^{-(k+k_0)^{1.1\color{black}}}) \sum_{i=1}^{N(k+k_0)^{1.1\color{black}}}\lambda_{d-2}(\mathcal{M}_{\hat{A}}\Delta \cap \Delta_\alpha \setminus \cup_{j=1}^{i-1}X_j\cup W_j)
  \\\nonumber&\leq & 
(\rho''^{(k+k_0)^{1.1\color{black}}}+(\frac{2}3\color{black})^{(k+k_0)^{1.1\color{black}}}+10^{-(k+k_0)^{1.1\color{black}}}) \sum_{i=1}^{N(k+k_0)^{1.1\color{black}}}(1-\hat{ c})^{i-1}\lambda_{d-2}(\mathcal{M}_{\hat{A}}\Delta \cap \Delta_\alpha)\\  & < &
 C(\rho''^{(k+k_0)^{1.1\color{black}}}+(\frac{2}3\color{black})^{(k+k_0)^{1.1\color{black}}}+10^{-(k+k_0)^{1.1\color{black}}}) \lambda_{d-2}(\mathcal{M}_{\hat{A}}\Delta \cap \Delta_\alpha)
  \end{eqnarray}
  
  On the other hand  inductively we see that 
$$\lambda_{d-2}(\mathcal{M}_{\hat{A}}\Delta \cap \Delta_\alpha\setminus (\cup_{i=1}^jX_i\cup W_i))\leq (1-\hat{c})^j \lambda_{d-2}(\mathcal{M}_{\hat{A}}\Delta \cap \Delta_\alpha).$$ 
  By our choice of $N$ and taking $j=N(k+k_0)^{1.1\color{black}}$ this implies that 
  $$\lambda_{d-2}(\cup_{i=1}^{N(k+k_0)^{1.1\color{black}}}X_i\cup W_i)\geq (1-(\frac 1 {10})^{(k+k_0)^{1.1\color{black}}})\lambda_{d-2}(\mathcal{M}_{\hat{A}}\Delta \cap \Delta_\gamma).$$
  Combining this with Inequality (\ref{eq:small toss}) proves the Theorem (if $k_0$ is large enough so that there exists $\rho<1$ so that $ C(\tau''^{(k+k_0)^{1.1\color{black}}}+(\frac{2}3\color{black})^{(k+k_0)^{1.1\color{black}}}+10^{-(k+k_0)^{1.1\color{black}}})<\rho^{(k+k_0)^{1.1\color{black}}}$ for all $k\geq 1$). 
   \end{proof}

\color{black}
\section{Restriction on left side}\label{sec:restrict lhs}
The point of this section is to prove Theorem \ref{thm:neighborhood} below. 
In this section $R$ refers to  normalized Rauzy induction and $\hat{R}$ for non-normalized.

Recall  $\pi_L=\begin{pmatrix}1& d-1&d&2&\dots&d-3& d-2\\d&d-1&d-2&d-3&\dots&2&1 \end{pmatrix}$, the matrix sizes at stage $k$ during restriction are $$\|A_k'\|\in [10^{(2k+1+k_0)^6+(k+k_0)^4},10^{(2k+1+k_0)^6+(k+k_0)^4+(k+k_0)^2}],$$and are given by  $1$ losing to  $d-2$ and $1$ not winning until returning to $\pi_L$ with norm in the range given above. (This is followed by transition, a path from $\pi_L$ to $\pi_s$.)  Let $\mathcal{A}_k'$ be the set of these matrices satisfying Condition **. 
 

\color{black}

Now after finishing freedom on LHS we have  families $\mathcal{M}_{\hat A}$ for some LHS matrix $\hat A$.

In this section  balanced and positive refer to $d-3\times d-3$ submatrices $A$ consisting of columns $C_2,\ldots, C_{d-2}$ and the entries 2 through $d-2$ in these columns. 
 
 Given $B,L$  let 
$\mathcal{A}(B,L)$ be the set of $A\in \mathcal{A}'_k$  
such that 
\begin{itemize}
\item $A$ is of form $A=A_1PA_2$ where  $\|P\|<B$ 
\item the upper $d-2\times d-2$ submatrix, $P$ has all but the 1st row positive
\item   $\|A_2\|<10^L.$

\end{itemize}

The idea is that $A$ has a fixed bounded size matrix $P$ not long before the end of Rauzy induction.

Given $M\in\mathcal{M}_{\hat A}$, 
let $$U_{d-2}(M)=\spanD(C_2(M), \ldots C_{d-2}(M))\subset V(M),$$  
$$U_d(M)=\spanD(C_2(M),\ldots, C_{d-2}(M),C_{d-1}(M),C_d(M))$$
 $$d_M=d(C_1(M),U_{d-2}),$$ 
 \begin{equation}\label{eq:tk def} t_k=10^{-[(2k+1+k_0)^6+ (k+k_0)^4+ \frac{1}{2}(k+k_0)^2]\color{black}},
 \end{equation}
  $$\mathcal{N}'=\mathcal{N}_{t_kd_{M}}U_{d-2}  \cap V(M)\color{black}$$ and 
let $$\mathcal{N}=\spanD(\mathcal{N}',C_{d-1}(M),C_d(M)).$$
\color{black}
Note this neighborhood depends on $M$ as well as $k$.  We denote it $\mathcal{N}(M)$ when $k$ is understood. \color{black}
\begin{thm}
\label{thm:neighborhood}For any $\zeta$  there exist constants $\alpha>0$ and  $B$  so that for all large enough $k_0\color{black}$  if $M$ is $\zeta$-balanced then  for all $c\in (.1,.9)\color{black}$
\begin{equation}\label{eq:restriction}
\lambda_{d-2}(M\mathcal{A}(B, \frac 1 {16}(k+k_0)^{2}\color{black})\Delta \cap \mathcal{N} \cap \Delta_c)\geq 
(1-\alpha^{(k+k_0)^{2}})\lambda_{d-1}(\mathcal{N} \cap \Delta_c)
\end{equation}
\end{thm}
Note that $B$ can be chosen to only depend on the Rauzy class. 

This says the image of $\Delta$ under elements of $\mathcal{A}(B,\frac 1 {16}(k+k_0)^{2}\color{black})$ takes up most of  $\mathcal{N}$.


To prove this result consider the following three sets of matrices  whose dependence on $k$ is suppressed.

$\mathcal{E}_1=\mathcal{A}'_k\setminus \mathcal{A}(B, \frac 1 {16}(k+k_0)^{2}\color{black})$ ,

$\mathcal{E}_2$ be the set of matrices of restriction of the LHS with norm at least $10^{(k+k_0)^6+(k+k_0)^4+\frac{3}{4}(k+k_0)^2}$ and  

$\mathcal{E}_3$ the set of matrices $A$ that leave restriction on the LHS with 
$\|A\|\leq 10^{(k+k_0)^6+(k+k_0)^4}$. 
\color{black}

The Theorem will follow from establishing the following results:
\begin{lem}\label{lem:hat 1}There exists $\rho<1$ such that $\lambda_{d-3}(M\mathcal{E}_1\Delta \cap \mathcal{N}')<\rho^{(k+k_0)^2\color{black}}\lambda_{d-1}(\mathcal{N}')$.
\end{lem}
\begin{prop}\label{prop:not late}There exists $\rho<1$ such that $\lambda_{d-3}(M\mathcal{E}_2\Delta \cap \mathcal{N}')<\rho^{(k+k_0)^2\color{black}}\lambda_{d-1}(\mathcal{N}')$.
\end{prop}
\begin{prop}\label{prop:not early}  There exists $\rho<1$ such that $\lambda_{d-1}(M\mathcal{E}_3\Delta \cap \mathcal{N})<\rho^{\frac 1 2 (k+k_0)^2}\lambda_{d-1}(\mathcal{N})$.
\end{prop}
\begin{proof}[Proof of Theorem \ref{thm:neighborhood} assuming previous 3 results] First observe that if $y \in \mathcal{N}$ and 
$y\notin M\mathcal{A}(B,(k+k_0)^2)\Delta \cap \mathcal{N}$ then $y\in M\mathcal{E}_i\Delta$ for some $i\in\{1,2,3\}$.  So it suffices to prove that Lemma \ref{lem:hat 1} and Proposition \ref{prop:not late} imply analogous bounds for $\lambda_{d-1}(M\mathcal{E}_i\Delta \cap \mathcal{N})$ for $i \in \{1,2\}$. 
Now $C_{d-1}$ and $C_d$ are unchanged during restriction on LHS and so our sets $\mathcal{N}$, $M\mathcal{E}_1\Delta \cap \mathcal{N}$ and $M\mathcal{E}_2\Delta \cap \mathcal{N}$ are obtained from $\mathcal{N}'$, $M\mathcal{E}_1\Delta \cap \mathcal{N}'$ and $M\mathcal{E}_2\Delta \cap \mathcal{N}'$ by taking the convex combinations with the same line. So the analogous estimates hold. 
\color{black}
\end{proof}
\begin{proof}[Proof of Lemma \ref{lem:hat 1}] 
 Write $M$ as $M(x,r)$.  We apply Proposition \ref{prop:positive} to the Rauzy class on $d-3$ symbols with $P$ a fixed positive matrix to get that there exists $\rho<1$ such that 
 \begin{equation}
\label{eq:neighborhood}
\lambda_{d-4}(\{y\in U_{d-2}:\not\exists n: \|A(R^ry,n)\|<10^{\frac 1 {16}(k+k_0)^2\color{black}} \text{ and }A(R^ry,n)=P \})\\
\color{black}<
C\rho^{(k+k_0)^2\color{black}}\lambda_{d-4}(U_{d-2}).
\end{equation}
 Let $\mathcal{D}$ denote the set of the matrices $\|A(R^ry,n)\|<10^{\frac 1 {16}(k+k_0)^2\color{black}}$ and $A(R^ry,n)\neq P$.   
Because in restriction on the LHS we do not affect columns $C_{d-1}$ and $C_d$,  moving to $U_d$ we take  the convex combination with the same line segment, 
giving the analogous estimate for $U_d$. 
To complete the lemma,  let $H$ be a hyperplane in $V(M)$ \color{black} parallel to $U_{d-2}$.  Then if $A \in \mathcal{D}$   
\begin{equation}\label{eq:U slice bound} \lambda_{d-4}\bigl(H\cap MA\Delta)\leq \lambda_{d-4}(\spanD(C_2(MA),...,C_{d-2}(MA))\bigr).
\end{equation}
 
 We now take an orthogonal transversal in $V(M)$ to $U_{d-2}$ 
  and exhaust  $M\mathcal{D}\Delta\cap \mathcal{N}'$ by taking a hyperplane, $H$, through each point of the transversal.  
  Applying Inequalities \eqref{eq:U slice bound} and \eqref{eq:neighborhood} and observing that for every $H\subset \mathcal{N}'$ we have 
  $\lambda_{d-4}(H\cap \mathcal{N}')$ is proportional to $\lambda_{d-4}(U_{d-2})$ we obtain the lemma.
\end{proof}

\subsection{Proof of Proposition \ref{prop:not late}}

The proof of Proposition \ref{prop:not late} is similar to the previous lemma (reducing to $\Lambda_{d-3}$ and using Proposition \ref{prop:positive}) but requires a couple of preliminaries. \color{black}
\color{black}
\begin{lem}
\label{lem:lengths}
Let $M=M(\xi,n)$ be a matrix of Rauzy induction.  If $|C_j(M)|>N\color{black}$ for some $j$  then the unnormalized length $\lambda_j$  of the $j^{th}$ interval of $\hat{R}^n\xi$ is at most $\frac 1 {N\color{black}}.$
\end{lem}
\begin{proof}
The $a_{i,j}$ entry of the $C_j$ column of $M$  is the number of visits of points of the $j^{th}$ interval  of $R^n\xi$  to the  $i^{th}$ interval $I_i$ of the original IET before these points return to $R^n\xi$. By assumption   $ \sum_i a_{i,j}\geq N\color{black}$. We conclude   $N\lambda_j\leq \sum_i a_{i,j} \lambda_j\leq 1$, the length of our initial interval.
\end{proof}

\begin{lem}\label{lem:too close} Let $M(x,r)$ be a matrix of Rauzy induction so that $\frac{|C_i(M)|}{|C_{i'}(M)|}<\zeta$ for all $1\leq i,i'\leq d-3$. For any path   $\gamma$  of Rauzy induction of some length $p$ where at the end  $d-2$ beats $1$ and such that  
\begin{itemize}
\item $A(R^ry,n)$ is a matrix of restriction on the LHS, 
\item  $|C_{d-2}(A(R^ry,n))|>2N\zeta$ and 
\item $R^{r+n}y$ follows $\gamma$
\end{itemize} 
then $MA(R^{r+n+p}y) \Delta \subset \mathcal{N}_{\frac{d_M}N}U_d$. 
\end{lem}
\begin{proof} Because $1$ loses during restriction and $d-2$ beats $1$ on the last step of $\gamma$,   it follows that every $z\in MA(R^{r}y,n+p-1)\Delta$ satisfies  $$({R}^rz)_1=(\hat{R}^{p+n-1}R^rz)_1<(\hat{R}^{p+n-1}R^rz)_{d-2}.$$
However by our assumptions we have  $|C_{d-2}(A(R^{r}y,n)\color{black})|>{2N\zeta}$ and so by Lemma \ref{lem:lengths} we have $(\hat{R}^{n}R^ry)_{d-2}<\frac 1 {2N\zeta}$. So if 
$z \in MA(R^{r}y,n+p)$ then $(R^rz)_1<\frac 1 {2N\zeta}$ and 
$$d(z,U_{d})<d_M\frac{|z_1C_1(M)|}{\sum_{i=2}^{d}|z_iC_i(M)|}\leq \frac{d_M}{2N(1-z_1)}.$$
\end{proof}

\begin{proof}[Proof of Proposition \ref{prop:not late}]  Let  $$\mathcal{N}''=\mathcal{N}_{10^{-\frac 1 8 (k+k_0)^2}t_kd_{M}}U_{d-2}\cap V(M)\subset \mathcal{N}'=\mathcal{N}_{t_kd_{M}}U_{d-2}\cap V(M).$$   We partition $\mathcal{E}_2$ into two parts.
  The first subset consists of those    matrices $A$ such that $MA\Delta\subset \mathcal{N}''$.  This gives  an exponentially small part of $\mathcal{N}'$ by simple geometry.  The second subset consists of those $A\in \mathcal{E}_2$ such that $MA\Delta_d\cap \mathcal{N}''^c  \neq\emptyset$  which we now treat. That is, let  $\tilde{\mathcal{A}}$ be the set of matrices $A$ of restriction  on the LHS with $$\|A\|\geq 
10^{(k+k_0)^6+(k+k_0)^4+ \frac{3}{4}(k+k_0)^2}$$ and such  that 
$$MA\Delta \cap \mathcal{N}''^c \neq \emptyset.$$

We now show there exists $\rho<1$ such that 
\begin{equation}
\label{eq:late key}\lambda_{d-4}(M\tilde{\mathcal{A}}\Delta \cap U_{d-2})<\rho^{(k+k_0)^2}
\lambda_{d-4}(U_{d-2})
\end{equation}
  by showing that $\lambda_{d-4}(M\tilde{A}\Delta \cap U_{d-2})$ is an exponentially small amount in $\frac 1 8 (k+k_0)^4$ multiplied by the measure of $U_{d-2}$. To do this we show that its measure is smaller than the volume of $\mathcal{N}''$.  Set  
$$N= \frac 1 {2\zeta} 10^{(k+k_0)^6+(k+k_0)^4+\frac 5 8 (k+k_0)^2}.$$ 
 By Lemma \ref{lem:too close}, once a matrix $A$ of norm at least  $2\zeta N$  becomes $\zeta$-balanced it can not be in $\tilde{\mathcal{A}}$ because 
 \begin{equation}\label{eq:getting containment}
 MA\Delta\subset  \mathcal{N}_{\frac{d_M}N}U_d\subset \mathcal{N}'',
 \end{equation}
  the last inclusion by the choice of $t_k$.
Let $M''$ be a matrix in $\mathcal{R}_{d-2}$ that arises from a positive matrix of $\mathcal{R}_{d-3}$ followed by a fixed path where at the end $d-2$ beats 1, so that $MM''$ is $\zeta$-balanced for every (non-negative) matrix $M$. 
 We now apply Proposition \ref{prop:determined} to this $M''$ and obtain that regardless of our past, off of an exponentially small proportion of $M\Delta$ we produce matrices with ratio of columns at most 
 $\zeta$ before the norm has increased by more than $10^{\frac 1 8 (k+k_0)^2}$. Re-interpreting this we produce matrices of restriction in $\mathcal{R}_d$ coming from points in $V(W)$ that are $\zeta$-balanced 
 (in columns 2,...,$d-2$), whose intersections with $U_{d-2}$  are contained in $\mathcal{N}''$ (via \eqref{eq:getting containment}).  This shows that the measure of $M\tilde{\mathcal{A}}\Delta \cap U_{d-2}$ is less than the measure of $\mathcal{N}''$   obtaining     \eqref{eq:late key}. With our initial remark we have that 
 $\lambda_{d-2}(\mathcal{E}_{2}\Delta\cap\mathcal{N'})$ is an exponentially small proportion of $\mathcal{N}'$. 
\end{proof}

\subsection{Proof of Proposition \ref{prop:not early}}
Let $\hat{\mathcal{A}}$ be a set of matrices of restriction on the LHS so that 
\begin{enumerate}
\item $\|A\|\in [10^{(k+k_0)^6+(k+k_0)^4},2\cdot 10^{(k+k_0)^6+(k+k_0)^4}]$ for all $A \in \mathcal{\hat{A}}$.
\item $U_{d}\subset \cup_{A\in \mathcal{\hat{A}}}MA\Delta $
\item $A \neq A' \in \mathcal{\hat{A}}$ then $A\Delta \cap A'\Delta=\emptyset.$
\end{enumerate}
Parametrize affine hyperplanes parallel to $U_{d}$ and non-trivially intersecting $M\Delta$ by their distance from $U_{d}$. So $H_\alpha$ is the affine hyperplane parallel to $U_{d}$ whose distance from $U_{d}$ is $\alpha$. 
To prove Proposition \ref{prop:not early} it will suffice to show 
there exists $\rho<1$ so that for all $A\in \mathcal{\hat{A}}$ and $\alpha\leq t_kd_M$ 

\begin{equation}\label{eq:suff not early} \lambda_{d-2}(H_\alpha \cap MA\Delta)>(1-\rho^{\frac 1 2 (k+k_0)^2})\lambda_{d-2}(MA\Delta \cap U_d)
\end{equation}

We prove this inequality.  For each $A$ 
let $\alpha_0=d(C_1(MA),U_d)$ so  $H_{\alpha_0}\cap MA\Delta$ is a single point.
\color{black}
 
Notice that $$\alpha_0\geq\frac 1 {2\zeta} d_M10^{-[(k+k_0)^6+(k+k_0)^4]}.$$ 

Now the side lengths  of parallel hyperplanes $H_\alpha$ intersecting a simplex 
vary linearly and angles are constant \color{black} so $\lambda_{d-2}(H_\alpha \cap MA \Delta)=(1-\frac{\alpha}{\alpha_0})^{d-2}\lambda_{d-2}(MA\Delta \cap U_d)$. 

By the bound $\alpha<t_kd_M$ and the definition of $t_k$ we have   $$\frac{\alpha}{\alpha_0}<2\zeta \cdot 10^{-\frac 1 2 (k+k_0)^2}.$$
We have established  (\ref{eq:suff not early}) for an appropriate $\rho$.

We now finish 
the proof of 
Proposition \ref{prop:not early}.
Because $ \cup_{A\in \hat{\mathcal{A}}}MA\Delta \supset U_{d}$,   inequality (\ref{eq:suff not early}) implies that 
$$\lambda_{d-2}(H_\alpha \cap \cup_{A \in \mathcal{\hat{A}}}MA\Delta)>(1-\rho_3^{\frac 1 2 (k+k_0)^2})\lambda_{d-2}(U_d)>(1-\rho_3^{\frac 1 2 (k+k_0)^2})\lambda_{d-2}(H_\alpha\cap M\Delta) .$$
This establishes that all but an exponentially small portion of $\mathcal{N}$ is not in $M\hat{\mathcal{A}}_3\Delta$.

\section{Transition, freedom and restriction on the right hand side}\label{sec:rhs}

 Theorem \ref{thm:RHS} is used in the proof of Theorem \ref{thm:planes}. The other results used outside of this section are Lemmas \ref{lem:second}  and \ref{lem:remaining RHS}, which are used to establish Condition * (2) and (3). 
We now establish the second condition of Condition *.

\begin{lem}
\label{lem:second}
\color{black}
There exists $C>1$ and $\rho<1$ so that if  $M(x,r)$ is at end of restriction on LHS,  then the measure of  the set of $y\in M\Delta$ such that there exists $n$  with $A(R^ry,n)$ reaches $\pi_s$ and $\|A\|\leq 10^{(k+k_0)^{2}}$ is at least 
 $(1-C\rho^{(k+k_0)^{2}})\lambda_{d-1}(\hat{M}\Delta)$
 \end{lem}
 \begin{proof}
By Proposition \ref{prop:balanced} we have that there exists $\zeta,C,\rho$ so that 
\begin{multline*}
\lambda_{d-1}(\{y \in \hat{M}\Delta: \exists n \text{ with } |C_{\max}(M(R^ry,n))|<10^{(k+k_0)^{2}}|C_{\max}(\hat M)|\\
 \text{ and } M(R^ry,n) \zeta\text{-balanced}\})>(1-C\rho^{(k+k_0)^{2}})\lambda_{d-1}(\hat{M}\Delta).
\end{multline*}
For the matrix to become balanced (as $d \times d$ matrix),  $d$ has to be compared to $i$ for $i\leq d-2$, because these columns are much larger than $C_d$. For this to occur, we have to reach $\pi_s$.
\end{proof}

\color{black}
With Theorem \ref{thm:neighborhood} 
\color{black}
this establishes Condition* (2).

\color{black}

Given $\zeta$ 
let $\mathcal{B}_k$ be a collection of matrices $B$ 
 so that 
\begin{itemize}
\item $B$ is a matrix of Rauzy induction corresponding  to a  path starting at  $\pi_s$ and ending at $\pi_R$.  Anytime it returns to $\pi_s$ we have $d$ beating  $1,\ldots, d-2$. 
\item For $B_1,B_2\in \mathcal{B}_k$, $B_1\Delta\cap B_2\Delta=\emptyset$.
\item $\|B\|\in [10^{(2k+1+k_0)^6-(k+k_0)^4},10^{(2k+1+k_0)^6-(k+k_0)^4+(k+k_0)^2}] $
\item  $\frac 1 \zeta<\frac{|C_{d-1}(B)|}{|C_d(B)|}<\zeta.$
\end{itemize}
\begin{lem}\label{lem:remaining RHS} There exists $ c, \rho<1,$ C and $\zeta$ so that for $k_0$ sufficiently large   
\begin{enumerate}
\item$\lambda_{d-1}(\cup_{B\in \mathcal{B}_k}MB\Delta)>(1-C\rho^{(k+k_0)^2\color{black}})\lambda_{d-1}(M\Delta)$
\item
For all but a set of  planes $P$ of measure $\rho^{c(k+k_0)^2}$  
$$\lambda_2(P\cap (M\Delta\setminus \cup_{B\in \mathcal{B}_k}MB\Delta)<\rho^{\frac 12 (k+k_0)^2}\lambda_{d-1}(P \cap M\Delta)$$
\end{enumerate}

\end{lem}
\begin{proof}
We first show that the  that the subset of $M\Delta$ for which there is no $B\in \mathcal{B}_k$ is exponentially small part of $M\Delta$. 
We restrict to $W(M)$ where the first bullet  of $\mathcal{B}_k$ automatically holds. By applying Proposition \ref{prop:balanced} to $\mathcal{R}_2$  (which is justified by Lemma \ref{lem:face jacobian}) we obtain that the last three bullets  hold  off of an exponentially small subset  . These matrices give us paths in $\mathcal{R}_d$ corresponding to points in $W(M)$. So we obtain
$$\lambda_1(\cup_{B\in \mathcal{B}_k}W(MB))>(1-C\rho^{(k+k_0)^2})\lambda_1(W(M)). $$
 By the second conclusion of Theorem \ref{thm:most rhs} we have  (1) of the lemma. 
We now prove Conclusion (2). By Proposition \ref{prop:concave}, the fact that $M\Delta$ is bounded, and Fubini,  there  is a constant $c'$ so that for $k_0$ large enough,  the set of $P$ with
\begin{equation}\label{eq:tiny intersection}\lambda_2(P \cap M\Delta)<\rho^{\frac 3 4 (k+k_0)^2}\lambda_{d-1}(M\Delta)
\end{equation} 
is at most $\rho^{c'(k+k_0)^2}$ proportion of the planes intersecting $M\Delta$. Now for a plane $P$ not satisfying \eqref{eq:tiny intersection} to fail the conclusion of the lemma we have that 
$$\lambda_2(P\cap (M\Delta\setminus \cup_{B\in \mathcal{B}_k}MB\Delta))>2\rho^{\frac 1 2 (k+k_0)^2}\lambda_{d-1}(M\Delta).$$ 
The set of such planes is at most $2 \rho^{\frac  1 2 (k+k_0)^2}\lambda_{d-1}(\Delta)$ by Fubini's Theorem. We have the lemma for $c$ any number bigger than $\max\{c',\frac 1 2 \}$ 
(provided $k_0$ is large enough).

\end{proof}

Let $\mathcal{B}'_k$ be matrices $B'$ of restriction on RHS starting at $\pi_R$ and ending at $\pi_s$. Let \begin{equation}
\label{eq:defs_k}
s_k=10^{(2k+2+ k_0)^6+ (k+ k_0)^4}.
\end{equation} Recall from our choice of matrix sizes   $$s_k\leq \|B'\|\leq  2s_k.$$ and that   $d-1$ beats $d$   between  $s_k$ and $2s_k$ \color{black} times,  and then $d$ beats $d-1$. 
The purpose of this section is to prove the following:


\begin{thm}\label{thm:RHS}
There exists $c>0$, $\rho<1$ and $C$  such that for all but a  proportion at most $C\rho^{(k+k_0)^2}$, of planes $P\in\mathcal{P}$ that intersect $M\Delta$, 
$$\lambda_2(\cup_{B \in \mathcal{B}_k, B'\in \mathcal{B}_k'}MBB'\Delta\cap P)\geq \frac{c}{s_k} \lambda_2(M\Delta \cap P).$$ 
\end{thm}

\subsection{Freedom on RHS }


\begin{lem}\label{lem:face same} Let $B$ be a matrix on the right hand side. Then $F_i(MB)\subset F_i(M)$ for all $i\leq d-2$. 
\end{lem}
\begin{proof}
Since every entry that wins is either $d-1$ or $d$, for all $i\leq d-2$, $C_i$ is not added to another column. So 
$$\spanD(C_{i_1}(MB),...,C_{i_r}(MB),C_{d-1}(MB),C_{d}(MB)) \subset \spanD(C_{i_1}(M),...,C_{i_r}(M),C_{d-1}(M),C_d(M))$$
for any $i_1,...,i_r \subset \{1,...,d-2\}$. 
 This implies the result.
\end{proof}
\begin{cor}\label{cor:see sub} If $P$ is a plane such that  for all $B$ we have $F_d(MB) \cap P= F_{d-1}(MB)\cap P=\emptyset$, then $MB\Delta\cap P=M\Delta \cap P$. 
\end{cor}

Consider the partition of $M\Delta$ into $MB\Delta$ where $B$ are matrices of the RHS and let $E$ be  the complement of these partition elements. 

\begin{lem}\label{lem:all d}
The boundary of these partition elements are subsets of $F_i(M)$ for $i\leq d-2$  or $F_d(MB)$ or $F_{d-1}(MB)$.
\end{lem}

\begin{proof}Since we are on the right hand side, either $d$ beats some other letter or $d-1$ beats $d$. When $d$ beats  some other letter  the new boundary  is the face $F_d$. When $d-1$ beats $d$ the new face  is $F_{d-1}$. 
\end{proof}
Given a plane $P$ and   a matrix  $B\in\mathcal{B}_k$, we say $P$ is \emph{standard} for $MB$, if $P \cap F_{d-1}(MB)\neq \emptyset$ and $P \cap F_d(MB)\neq \emptyset$. 

Let  $\hat{\mathcal{B}}_k$ be  the set of $B$ for which $P$ is standard for $MB$. 

\begin{lem}\label{lem:2 bad} For any $P$, there are at most 2 different $B \in \mathcal{B}_k$ so that $P\cap MB\Delta \neq \emptyset$ and $P$ is not standard for $MB$.
\end{lem}
\begin{proof}If there is $B$ such that  $MB \cap P=M\cap P$, then it is unique and the lemma holds. So we assume that this is not the case. Notice 
$$\{x \in \spanD(C_{d-1}(M),C_d(M)):\spanD(V(M),x)\cap P \neq \emptyset\}$$ is an interval. Let $p_a$ and $p_b$ be the two extreme points of this interval and $q_a$ and $q_b$ be points in $\spanD(V(M),p_a)\cap P$ and $\spanD(V(M),p_b)\cap P$ respectively. If $B \in \mathcal{B}_k$ and $MB \Delta \cap P \neq \emptyset$ choose $q \in MB\Delta \cap P$ and consider the line segments connecting $q_a$ to $q$ and $q_b$ to $q$. If $q_a$ or $q_b$  are not in $MB \Delta$, then the line segments cross the boundary of $MB \Delta$ at $F_j(MB)$ for $j\geq d-1$. 
To see this, if the line segment crossed $F_i(MB)
            \subset F_i(M)$ for $i\leq d-2$ 
            it would be  leaving $M\Delta$ (by Lemma \ref{lem:face same}) and so can not be entering
            $MB\Delta \cap P$. Also the two line segments have to cross at different faces (so one at $F_{d-1}(MB)$ and one at $F_d(MB)$).  It follows that all but possibly the two $B \in \mathcal{B}_k$ whose projection to $\spanD(C_{d-1}(M),C_d(M))$ is most extreme are standard. 
\end{proof}

\begin{cor}\label{cor:2 bad} For any line $\ell$ contained in $P$, we have there are at most $2$ different $B$ so that  $\ell \cap MB\Delta\neq \emptyset$ and $B$ is not standard for $P$.
\end{cor}

\begin{prop}\label{prop:RHS decay}
 There are constants $C$ and $\rho<1$  so that for $k_0$ large enough, if  $M$ is a matrix at start of freedom on RHS, 
then  for all $B \in \mathcal{B}_k$,  except for  a proportion at most $C\rho^{(k+k_0)^{2 \color{black}}}$ of planes $P\in\mathcal{P}$ intersecting $M\Delta$, 
we have $\lambda_2(P\cap MB\Delta)<10^{-\frac 1 2 (k+k_0)^5}\lambda_2(P \cap M\Delta)$.
\end{prop}

\begin{proof}  We will use the lower bound in Theorem \ref{thm:diameter} to bound from below the diameter of the  intersection of $M\Delta$ with $P$ and the upper bound for the diameter of intersection of $MB\Delta$ with $P$.
By Proposition \ref{prop:RHS sing} at the start of RHS the second smallest  singular value   is at least $$\|M\|^{-1} 10^{(k+k_0)^5}$$

 So by Proposition \ref{prop:concave}, for a constant $c>0$, off of a proportion $\rho^{{c(k+k_0)}^2}$ of planes intersecting the simplex the intersection is at least $\|M\|^{-1} 10^{(k+k_0)^5}\rho^{(k+k_0)^2}$. 

 Now by Proposition \ref{prop:sizes} and Condition * 
$$\|M\|\leq 10^{2(k+k_0)}u_k\leq 10^{2(k+k_0)}2^k 10^{\sum_{i=3}^{2k+1} (i+k_0)^6+2\sum_{i=2}^k(i+k_0)^2+\sum_{i=2}^k(i+k_0)^{2.3}}.$$  
This gives us that the second smallest singular value of   $M$ is at least
\begin{equation}\label{eq:ratio sec sing}
 2^{-k}10^{-\sum_{i=3}^{2k+1}(i+k_0)^6-2\sum_{i=2}^k(i+k_0)^2-\sum_{i=2}^k(i+k_0)^{2.3}}10^{-2(k+k_0)+(k+k_0)^5}.
\end{equation} 
 
 At the end of freedom on the right hand side by Lemma~\ref{lem:output} we have that the second smallest singular value  of $MB$ is at most  $\frac{C}{|C_{\min}(MB)|}$ for some $C$, which by  the lower bound for $V_k$ in  Proposition \ref{prop:sizes}  is bounded by $$C  (2\zeta)^{-k}10^{-\sum_{i=3}^{2k+1}(i+k_0)^6}10^{-(k+k_0)^4}$$ (because $C_d$ and $C_{d-1}$ have ratio at most $\zeta$). Comparing this with Inequality (\ref{eq:ratio sec sing}), 
we see that for $k_0$ large enough, the second smallest singular value of $MB$ is at most $10^{-\frac 23  (k+k_0)^5}$ multiplied by the second smallest singular value of $M$. 
If $P$ is a plane so that  $\lambda_2(P \cap M\Delta)\geq 10^{-3(k+k_0)^2}\lambda_{d-1}(M\Delta)$, then since the smallest singular value is  nonincreasing,   the area has decayed proportionally to at least the decay in the second smallest  singular valued (that is, $10^{-\frac 23  (k+k_0)^5}$) multiplied by 
$10^{3(k+k_0)^2}$ giving a decay of at least $$10^{-\frac 23  (k+k_0)^5}
10^{3(k+k_0)^2}\geq  10^{-\frac 1 2 (k+k_0)^5},$$ for $k_0$ large enough.
 By Proposition \ref{prop:concave} off of a set of planes of proportion $\rho^{(k+k_0)^2}$ of the planes intersecting $M\Delta$  we have $\lambda_2(P \cap M\Delta)\geq 10^{-3(k+k_0)^2}\lambda_{d-1}(M\Delta)$, \color{black}  establishing the proposition.  
\end{proof}

\subsection{End of restriction}

Let $B\in \mathcal{B}_k$ and $M=A_1'B_1...A'_k$ be its ancestor. Consider the subinterval $W_{MB}\subset [0,1]$ defined by   $$W_{MB}= 
[\frac {s_kC_{d-1}(MB)+C_d(MB)}{|s_kC_{d-1}(MB)+C_d(MB)|},\frac{2s_kC_{d-1}(MB)+C_d(MB)}{|2s_kC_{d-1}(MB)+C_d(MB)|}].$$

Now fix $k$ and for  $\ell$ such that $s_k\leq l\leq 2s_k$, let $B'_{\ell}$ be the matrix in restriction given by $d-1$ beating $d$ exactly  $\ell$ times and then $d$ beating $d-1$ (to return  to $\pi_s$).
The following statements follow immediately from the definition.

\begin{multline}\label{eq:RHS restriction survive}
\cup_{\ell \in [s_k,2s_k]}MBB'_\ell\Delta=
\color{black}\{y\in MB\Delta:s_k<\frac{(R^ny)_{d-1}}{(R^ny)_d}<2\cdot s_k+1\}.
\end{multline}

At the end of restriction, the set that is left is  $\spanD(C_1(MB),\ldots,C_{d-2}(MB), W_{MB})$.  
\color{black}

\begin{lem}\label{lem:restrict size}  Suppose  $B$ is a matrix after freedom on RHS such that 
 $\frac 1 \zeta <\frac{|C_{d-1}(MB)|}{|C_d(MB)|}<\zeta$. 
   Then 

$$\frac{1}{9\zeta^2 s_k}\lambda_{d-1}(MB\Delta)\leq 
\lambda_{d-1}(\spanD(C_1(MB),...,C_{d-2}(MB), W_{MB}))
\leq \frac{\zeta^2}{s_k}\lambda_{d-1}(MB\Delta)$$
\end{lem}

\begin{proof} By the fourth condition on matrices in $\mathcal{B}_k$ we have that if $B \in \mathcal{B}_k$ then for each matrix in $B' \in \mathcal{B}'_k$ we have 
$$|C_{j}(MBB')|\in [\frac 1 {\zeta s_k}|C_j(MB)|,(2s_k+1)\zeta |C_j(MB)|]$$ for $j\geq d-1$.
 Moreover, $C_i(MBB')=C_i(MB)$ for $i\leq d-2$. So by Lemma \ref{lem:volume} we have that 
$$(3s_k\zeta)^{-2}\lambda_{d-1}(MB\Delta) \leq \lambda_{d-1}(MBB'\Delta)\leq (\frac{\zeta}{s_k})^{2} \lambda_{d-1}(MB\Delta).$$
There are $s_k$ such disjoint simplices, giving the lemma.
\end{proof}
\color{black}
\begin{lem}\label{lem:restriction survive} There exists $\rho<1$ so that off of a set of planes of measure at most $C\rho^{(k+k_0)^2}$ we  have 
\color{black}
$$\lambda_2(\cup_{B\in \mathcal{B}_k,\, B' \in \mathcal{B}'_k}MBB'\Delta \cap P)$$ is proportional to $\frac 1 {s_k}\lambda_2(\cup_{B \in \mathcal{B}_k}MB\Delta\cap P\color{black}).$ 
\end{lem}
\begin{proof}Consider the direction orthogonal to the smallest singular direction of $M$  that is in the direction $P$. Because $P\cap MB\Delta$ is convex, for all but an exponentially small proportion of these lines $\ell$ we have that $$\lambda_1(\ell \cap M\Delta)>10^{-(k+k_0)^3}diam(M\Delta \cap P)$$(analogously)to Proposition \ref{prop:concave}). By Proposition \ref{prop:RHS decay} we have that all but an exponentially small proportion of these lines, have that at least half of its length is  in segments of size at most $10^{-\frac 1 2 (k+k_0)^5}diam(M\Delta \cap P)$. By Corollary \ref{cor:2 bad} all but at most 2 such segments are cut by $F_{j}(MB_1)$ and $F_{j'}(MB_2)$ where $j,j'\geq d-1.$ We have that on any such segment the part that survives restriction has length at least $\frac 1 {3 \zeta s_k}$ times the length of the segment. The lemma follows. 
\end{proof}

\subsection{Proof of Theorem \ref{thm:RHS}}
 The Theorem follows by combining the estimates in Lemmas \ref{lem:remaining RHS} and \ref{lem:restriction survive}.

\section{Proof of Theorem~\ref{thm:planes}.}

Our simplices are measured while just starting  restriction \color{black} on left.
Note that \ref{A:nest} is satisfied by construction and \ref{A:nue} is satisfied by Theorem \ref{thm:nue}.

The goal now  is to verify \ref{A:close size} of Theorem~\ref{thm:planes}. Suppose we are on freedom left side at stage $k+1$, just having finished restriction on right hand side.  
We apply  Theorem~\ref{thm:diameter}.  Since the second smallest singular value  $\omega'(M)\geq \frac{1}{|C_{max}|}$ 
we have  constants $c_1,c_2>0$  and a plane $P$ such that the diameter $r_k^j(P)$  of 
$M\Delta\cap P$  satisfies
$$\frac{c_1}{|C_d(M)|^2}\leq r_k^j(P)\color{black}\leq \frac{c_2}{|C_1(M)|^2}.$$ 
  Now by our choice of sizes of matrices (Proposition \ref{prop:sizes}) \color{black} we have 
  \begin{equation}\label{eq:diam upper}\bar r_k\leq \frac{c_2}{|C_1(M)|^2}\leq \frac{c_2}{\zeta^k}10^{-2\sum_{j=3}^{2k+1}(j+k_0)^6+(j+k_0)^4+\sum_{j=1}^k4(j+k_0)^2}.
  \end{equation}
  Also, except for an exponentially small proportion  of the planes intersecting the simplex (Proposition \ref{prop:concave}) we have, 
  
   \begin{equation}\label{eq:diam lower}\hat r_k\geq \frac{c_1}{|C_d(M)|^2}\radlower\geq c_12^{-2k}10^{-2\sum_{j=3}^{2k+2}(j+k_0)^6}10^{-(k+k_0)^2},
   \end{equation}
   
where recall $\bar r_k$ is the maximum value for the diameters at stage $k$, and $\hat r_k$ is the minimum at stage $k$.
\color{black}

 \color{black}
  From these bounds it is easy to see that \ref{A:close size}  holds:
 for all $\epsilon>0$ that   
 $$\underset{k \to \infty}{\lim}\, \frac{\bar r_k^{1+\epsilon}}{\hat r_{k+1}}=0.$$  

\subsection{Proof of \ref{A:small loss}}

Let $\mathcal{M}_{k-1}$ be the set of all matrices we keep at  the end of restriction on RHS at the end of stage 
$k-1$. After that, let $\mathcal{M}_k'$ the set of matrices we keep after freedom on LHS at stage $k$ and $\mathcal{M}_k''$ after restriction on  LHS. 
Let $\mathcal{S}_{k}$ denote $\cup_{M\in\mathcal{M}_k} S(M)$ the union of  the illuminated sets at the end of freedom on LHS at stage $k$. Note that by Section \ref{sec:repeat} we may assume that it is a union of simplices.

Recall the definition  $$t_k=10^{-[(2k+1+k_0)^6+(2k+1+k_0)^4+\frac 1 2(k+k_0)^2]}$$
given by \eqref{eq:tk def} and $$s_k=10^{(2k+2+k_0)^6+(k+k_0)^4}$$
 given by (\ref{eq:defs_k}).

Inductive Assumption: There exists a constant $c>0$ so at the end of freedom on LHS at stage  $k$ at least half of the  planes $P \in \mathcal{P}$ satisfy 
\begin{equation}
\label{assumption:1}\tag{$Y_{k-1}$}
\lambda_2(P\cap \mathcal{S}_k)\geq c^{k-1} \prod_{i=1}^{k-1}t_i\prod_{i=1}^{k-1}\frac{1}{s_i}. 
\end{equation}
 Call the planes that satisfy this inequality $\mathcal{P}_{k-1}$. Also, 
there is a constant $C$ such that

\begin{equation}\label{assumption:2}\tag{$Z_{k-1}$}\lambda_{d-1}(\mathcal{S}_{k})\leq 
C^{k-1} \prod_{i=1}^{k-1}t_i\prod_{i=1}^{k-1}\frac 1 {s_i}.
\end{equation}





Now \ref{A:small loss}
follows from the next theorem.

\begin{thm}\label{thm:plane estimates} There exists $\rho<1$ and $c>0$  so that for each plane $P\in\mathcal{P}_{k-1}$ except for a subset of $\mathcal{P}_{k-1}$ of  measure at most $\rho^{k+k_0}$,  
 has the property that for a set of connected components $J\subset P\cap \mathcal{S}_k$  whose union has measure at least $(1-\frac 1 {9^k})\lambda_2(\mathcal{S}_k\cap P)$,
\begin{equation}\label{eq:plane est}
 \lambda_2( \mathcal{M}_P \cap J)\geq c \frac{t_k}{s_k}\lambda_2(J)
 \end{equation}
 where $\mathcal{M}_P$ are the matrices at the end of restriction on LHS at stage $k$  so that ${diam(P\cap M\Delta)>\hat{r}_k}$.
Moreover the inductive assumptions $(Y_k)$ and $(Z_k)$ hold. That is, 
\begin{equation*}
\lambda_2(P\cap \mathcal{S}_{k+1}(\Delta))\geq c^{k}\prod_{j=1}^{k}t_j
\prod_{j=1}^{k}\frac{1}{s_j}
\end{equation*}
and 
\begin{equation*}
\lambda_{d-1}( \mathcal{S}_{k+1}(\Delta))\leq C^{k}\prod_{j=1}^{k}t_j
\prod_{j=1}^{k}\frac{1}{s_j}.
\end{equation*}
\end{thm}

We begin the proof.  
In proving Theorem~\ref{thm:plane estimates} in going from $\mathcal{S}_k$ to $\mathcal{S}_{k+1}$ we proceed in $3$ steps going  through restriction on LHS at stage $k$ where we apply Theorem~\ref{thm:neighborhood},
to treat freedom and restriction on RHS where we will apply Theorem~\ref{thm:RHS} at stage $k$ and finally to treat freedom on LHS (at stage $k+1$)  we use Theorem \ref{thm:big shadow}.

\begin{prop}\label{prop:through restriction} There exists $\rho<1$ and $c'>0$ so that except  for a set of $P \in \mathcal{P}_{k-1}$ of measure at most $\frac{1}{4}\rho^{(k+k_0)}$ we have that 
 for a set of connected components   $J\subset P\cap \mathcal{S}_k$ of total measure  at least $(1-\frac 1 4 \frac 1 {9^k})\lambda_2( P\cap\mathcal{S}_k)$ satisfies:
$$\lambda_2\big(\cup_{M\in\mathcal{M}_k''}(M\Delta\cap J )\big)\geq c't_k \lambda_2(J).$$
\end{prop}
To prove this Proposition we will need the following lemmas.

 \begin{lem}
 \label{lem:neigh}
 Let $\mathcal{N}(M)$ be as defined just before Theorem \ref{thm:neighborhood}.  There exists $c'>0$ so that if $\ell$ is a line in direction $\phi$ intersecting $\mathcal{S}(M)$ then 
$$\lambda_1(\ell \cap \mathcal{N}(M))\geq c'{t_k}\lambda_1(\ell \cap \mathcal{S}(M)).$$
\end{lem}
\begin{proof} $\mathcal{N}(M)$ is a slab about $F_1(M)$ with width $t_kd_M$. If the line segment crosses $\mathcal{N}(M)$ the length of the intersection with $\mathcal{N}(M)$ 
is $\frac{t_kd_M}{\sin(\gamma)}$ where $\gamma$ is the angle $\ell$ makes with $F_1(M)$. 
\end{proof}
From the definition of $\mathcal{N}(M)$ we have: 
\begin{lem}\label{lem:meas decay}There exists a constant $C'$ so that $\lambda_{d-1}(\mathcal{N}(M))<C't_k\lambda_{d-1}(\mathcal{S}(M))$. 
\end{lem}
\begin{proof}[Proof of Proposition \ref{prop:through restriction}] 
 By Lemma~\ref{lem:neigh},
  and the induction hypothesis,  for all of the planes  in $\mathcal{P}_{k-1}$ we have that 
\begin{equation}\label{eq:plane nbhd}
\lambda_2(\cup_{M\in \mathcal{M}'_k}\mathcal{N}(M) \cap P)
\geq {c'}{t_k} c^{k-1}\prod_{j=1}^{k-1}t_j
\prod_{j=1}^{k-1}\frac{1}{s_j}.
\end{equation}
 By applying Theorem \ref{thm:neighborhood} to each simplex in $\mathcal{M}_k'(\Delta)$ we have that 
$$\lambda_{d-1}\Big(\cup_{M \in \mathcal{M}'_k}\mathcal{N}(M)\setminus(\cup_{M\in \mathcal{M}_k''}M\Delta\Big)\leq \alpha^{(k+k_0)^2}\lambda_{d-1}(\cup_{M \in \mathcal{M}_k'}\mathcal{N}(M))\leq \alpha^{(k+k_0)^2} C't_k\lambda_{d-1}(\mathcal{S}_k).$$
Lemma \ref{lem:meas decay} gives the second inequality. 
By a straightforward estimate using Fubini  the measure of planes $P\in \mathcal{P}_k$ so that 
\begin{equation}\label{eq:too bad}
\lambda_2\Big(P \cap \big(\cup_{M \in \mathcal{M}_k'}\mathcal{N}(M)\setminus(\cup_{M\in \mathcal{M}_k''}M\Delta)\big)\Big)>\alpha^{\frac 1 2 (k+k_0)^2}C't_k\lambda_{d-1}(\mathcal{S}_k)
\end{equation}
is at most $C''\alpha^{\frac{1}{2}(k+k_0)^2}$ for some $C''$.  
Appealing to Assumption $(Y_{k-1})$ and $(Z_{k-1})$ which say that $\lambda_{d-1}(\mathcal{S}_k)<(\frac C {c})^{k-1}\lambda_2(P\cap \mathcal{S}_k))$ we have
\begin{equation}\label{eq:too bad}
\lambda_2\Big(P \cap \big(\cup_{M \in \mathcal{M}_k'}\mathcal{N}(M)\setminus(\cup_{M\in \mathcal{M}_k''}M\Delta)\big)\Big)>\alpha^{\frac 1 4 (k+k_0)^2}\lambda_2(\cup_{M \in \mathcal{M}_k'}\mathcal{N}(M)\cap P)
\end{equation}
is at most $C''\alpha^{\frac 1 2 (k+k_0)^2}$ (if $k_0$ is large enough). 
\end{proof}

Analogously to the  previous proposition we have the following:

\begin{prop} There exists $\rho<1, c>0$ so that  except  for a set of $P \in \mathcal{P}_{k-1}$ of measure at most $\frac{1}{4}\rho^{(k+k_0)}$, 
for  
 a set of connected components   $J=\mathcal{M}_k''\Delta\cap P$  whose union has measure at least $(1-\frac 1 4 \frac 1 {9^k})\lambda_2(P\cap\mathcal{S}_k)$ 
 we have  
$$\lambda_2(\cup_{M\in \mathcal{M}_k}M(\Delta)\cap J)\geq \frac{c'}{s_k}\lambda_2(J).$$
\end{prop}

We sketch the proof. By Theorem \ref{thm:RHS} there exists $\rho<1, c'>0$ so that for each simplex $M\Delta$ (with $M \in \mathcal{M}_k''$)
 there exists a subset  $\mathcal{E}(M)\subset M\Delta$  of measure at most $\rho^{(k+k_0)^2}\lambda_{d-1}(M\Delta)$, so that if $P \cap M\Delta \setminus \mathcal{E}(M) \neq \emptyset$ then 
 $$\lambda_2(P \cap M\Delta \cap \mathcal{M}_k)> \frac {c'} {s_k}\lambda_2(P \cap M\Delta).$$ 
 So we want to show that for all planes $P$, except for a set of planes of measure at most $\frac{1}{4}\rho^{k+k_0}$,  we have  
 $$\lambda_2\Big(P \cap (\cup_{M\in\mathcal{M}''_k} \mathcal{E}(M))\Big)<\frac 1 4 \frac 1 {9^k}\lambda(P \cap \mathcal{M}''_k \Delta).$$
 Now by Lemma \ref{lem:meas decay} and 
 Assumptions $(Z_{k-1})$ we have that $\lambda_{d-1}(\mathcal{M}''_k\Delta)\leq C' C^{k-1} t_k\prod_{i=1}^{k-1}\frac{t_i}{s_i}$ and so 
 $$\lambda_{d-1}(\cup_{M \in\mathcal{M}''_k}\mathcal{E}(M))<\rho^{(k+k_0)^2}C' C^{k-1} t_k\prod_{i=1}^{k-1}\frac{t_i}{s_i}.$$
The remainder of the proof is as in Proposition \ref{prop:through restriction}.

 Theorem \ref{thm:big shadow} together with a proof analogous to Proposition \ref{prop:through restriction} implies: 
\begin{prop}There exists $\rho<1$ and $c'>0$, so that except for  a set of planes $P \in \mathcal{P}_{k-1}$ of measure at most $\frac{1}{4}\rho^{(k+k_0)}$,  for a  set of connected components $J\subset P \cap \mathcal{M}_{k}(\Delta)$ of total measure at least $(1-\frac 1 4 \frac 1 {9^k})\lambda_2(P\cap \mathcal{S}_k)$ we have  
$$\lambda_2( J\cap\mathcal{S}_{k+1})\geq c'\lambda_2(J).$$
\end{prop}

\begin{proof}[Proof of Theorem \ref{thm:plane estimates}] 
We first prove \eqref{eq:plane est}.
The last three  propositions establish the inequality except for the condition on diameters and the measure on planes.  By Theorem \ref{thm:diameter} each matrix is cut by a plane whose intersection is at least proportional to $\frac{\omega'}{|C_{\max}(M)|}$ where by Proposition \ref{prop:RHS sing}, $\omega'$ is at least proportional to ${|C_{\min}(M)|}^{-1}$. This is at least proportional to $u_k^{-1}$ in Proposition \ref{prop:sizes}, which is at least  $$
 (2^k 10^{\sum_{i=3}^{2k+1} (i+k_0)^6+2\sum_{i=2}^k(i+k_0)^2+\sum_{i=1}^k(i+k_0)^{2.3}})^{-1}$$
 By Proposition \ref{prop:concave}, the proportion of the measure of a simplex $\Delta_k$ that is cut with diameter smaller that $10^{-(k+k_0)^2}u_k^{-1}$ is at most $10^{-c(k+k_0)^2}\lambda_{d-1}(\Delta_k)$. 
 By the above discussion, the complement of these intersections has diameter at least $\hat{r}_k$. 
 Analogously to the proof of Proposition \ref{prop:through restriction} we obtain the theorem off of a subset of planes in $\mathcal{P}_{k-1}$ of measure at most $\frac{1}{4}\rho^{k+k_0}$. The total loss of measures of planes using this comment and the three propositions is then at most $\rho^{k+k_0}$. Then we take $\mathcal{P}_k$ to be $\mathcal{P}_{k-1}$  where we have thrown out this exponentially small set of planes. The total removed has measure at most  $\sum_{k=1}^\infty \rho^{k+k_0}$.  We  choose $k_0$ large enough, so this sum is at most $1/2$ the total measure of the set of planes, which establishes our condition on the measure of $\mathcal{P}_k$.

 Inequality \eqref{eq:plane est} implies Assumption $(Y_{k})$. Assumption $(Z_k)$ follows from Lemmas \ref{lem:meas decay} and \ref{lem:restrict size}. Indeed, we apply these estimates to each simplex. Summing over the simplices we obtain $(Z_k)$. 
\end{proof}

Note that $\hat{\mathcal{P}}$ in the statement of Theorem \ref{thm:planes} is $\cap_{k=1}^\infty\mathcal{P}_k.$  We call $P\in\mathcal{P}$ {\em good}

\begin{proof}[Proof of (e)]  We need to show that for every $\epsilon>0$ 
\begin{equation}\label{eq:cond 4 suff}
\underset{k \to \infty}{\lim} \bar{r}_k^{\frac 1 2 +\epsilon}\prod_{j=1}^k (\frac{ t_j}{s_j})^{-1}=0.
\end{equation}
By Theorem \ref{thm:plane estimates},  we have that if $P$ is a good plane then 
$\lambda_2(P \cap S_{k+1})\geq c^k\prod_{j=1}^k\frac{t_j}{s_j}$. Now, \eqref{eq:cond 4 suff} implies that for any $\delta>0$, for all large enough $k$ we have that 
$\bar{r}_k\leq\prod_{j=1}^k (\frac{ t_j}{s_j})^{1-\delta}$. So each polygon making up $S_k \cap P$ has area at most $\prod_{j=1}^k (\frac{ t_j}{s_j})^{2-2\delta}$. With the previous Theorem \ref{thm:plane estimates} estimate, this implies that there are at least 
$\prod_{j=1}^k (\frac{ t_j}{s_j})^{1-2\delta}$ (for all $k$ large enough) polygons.

We now prove \eqref{eq:cond 4 suff}. \color{black} There exists $q$ a degree 5 polynomial so that both $t_j$ and $\frac 1 {s_j}$ are at least $10^{-(2j)^6+q(j)}$. 
From this it follows that there exists $p$ a degree $6$ polynomial so that 
$$\prod_{j=1}^k(\frac{t_j}{s_j})^{-1}>10^{\frac 1 7 (2k)^7+p(k)}.$$ 
Now by \eqref{eq:diam upper} we have that there exists a degree at most 6 polynomial $\tilde{p}$ so that 
$\bar{r}_k<10^{-\frac 2 7(2k)^7+\tilde{p}(k)}$. This establishes \eqref{eq:cond 4 suff} and so (e).
\end{proof}

\subsection{Proof of \ref{A:sep}} 
For any $j$ let $F_j=span (e_1,\ldots e_{j-1},e_{j+1},\dots, e_d)$  and  for any matrix $M$ recall  $F_j(M)=M(F_j)$ the $j$ face.  We compute before projectivizing.  For any $i,j$ and matrix $A$ after finishing freedom on LHS  we have 
$$d(C_i(A),F_j)=d(A(e_i),F_j)\geq \frac{1}{\|A\|}.$$  Then
$$d(C_i(MA), F_j(M))=d(MA(e_i),F_j(M))\geq \frac{1}{\|M\|}d(A(e_i),F_j)\geq \frac{1}{\|M\|\cdot\|A\|}.$$
After  projectivizing this says vertices are bounded away from faces 
by an amount which is at least \color{black} $ r_{k+1}.$

\section{Appendix}

 We denote by $A$ matrices of freedom on LHS.
Recall the paths $\gamma(x,m)$ of Rauzy induction defined in Section \ref{sec:repeat} that are $\pi_s$ via $\pi'$ isolated.  We wish to prove \begin{lem}
 \label{lem:avoid hyp}
 There exists $\rho>0$ and $m_0$ so that for any hyperplane $H$ 
$$\lambda_{d-1}(\{x \in \Delta \times \pi_L: \exists \gamma(x,m)\ \pi_s\ \text{via}\ \pi' \text{isolated}\  m\leq m_0\  \text{with}\  A(x,m) \Delta\cap H 
=\emptyset\})>\rho.$$   
\end{lem}

With the same notations as above, via usual balanced estimates we have:
\begin{cor}
For any $\zeta$ large enough, there exists $\rho$ so that for any hyperplane $H$ and $M=M(x,r)$ a $\zeta$-balanced matrix of Rauzy induction, we have 
$$\lambda(\{y \in M\Delta \times \pi_L: \exists m\leq m_0 \text{ with }\ \gamma(R^ry,m)\  \pi_s\ \text{via}\ \pi'\ \text{isolated  and}\ MA(R^ry,m) \cap H =\emptyset\})>\rho\lambda(M\Delta).$$
\end{cor}

  \begin{cor} \label{cor:avoid hyp} Given $\zeta$ there exists $\rho_1$ such that if   $M=M(x,r)$ a matrix of Rauzy induction so that
\begin{itemize}
\item $\frac{|C_i(M)}{|C_j(M|}<\zeta$ for all $1\leq i,j\leq d-2$,
\item $\pi(R^ry)=\pi_L$.
\end{itemize}
and if $H$ is any hyperplane contained in $\spanD(C_1(M),...,C_{d-2}(M))$ then 
\begin{multline*}\lambda_{d-3}(\{y\in \spanD(C_1(M),...,C_{d-2}(M)): \exists m\leq m_0 \text{ with } \gamma(R^ry,m)\  \pi_s\ \text{via}\ \pi'\ \text{isolated and } \\
 MA(R^ry,m)\Delta \cap H =\emptyset\})>\rho_1 \lambda_{d-3}(\spanD(C_1(M),...,C_{d-2}(M))). 
 \end{multline*}
\end{cor}

In order to prove Lemma~\ref{lem:avoid hyp} we need the following simple lemma first.

\begin{lem}
There is constant $\epsilon_0$ such that 
if a hyperplane in $\Delta$ intersects the radius $\frac{1}{2}$ ball about the  $e_1$ vertex it cannot intersect the $\epsilon_0$  ball about every other vertex.
\end{lem}
\begin{proof}
The proof is by contradiction.  If the lemma is false there is a sequence of hyperplanes $H_n$ defined by $a_{1,n}x_1+a_{2,n}x_2+\ldots + a_{d,n}x_n=1$ which intersect the  $\frac{1}{n}$ neighborhoods of $e_j$ for $j>1$ and the $\frac{1}{2}$  neighborhood of $e_1$. This forces the coefficients $a_{j,n}$ to be bounded. Passing to a subsequence and taking a limit we find  that the limiting hyperplane $H$ would contain  $e_j$ for $j>1$, and intersect the  $\frac{1}{2}$ neighborhood of $e_1$. Then $H$ must be of form $$a_1x_1+x_2+\ldots +x_d=1.$$ But this is not a hyperplane subset of $\Delta$. 
\end{proof}

\begin{proof}[Proof of Lemma \ref{lem:avoid hyp}]
Let $\epsilon_0$ from the last lemma.  It  suffices to show that for any $i>1$ there exists  a  bounded length path of Rauzy induction whose corresponding subsimplices  are contained in $B(e_i,\epsilon_0)$, and there exists a path whose corresponding subsimplex of Rauzy induction is contained in $B(e_1,\frac 1 2)$. 

In the second case  consider the path where $1$ wins $d-3$ consecutive times.  It reaches  $\pi'$ and then after $1$ beats $2$ it returns to  $\pi_s$.  So the first interval is longer than the sum of the other intervals, so its length is at least $\frac 1 2 $. 

 In the first case $(i>1)$ starting at $\pi_L$ have $d-2$ beat 1 then $2, \ldots, i-1$. Then have $i$ beat $d-2, d-3,\ldots i+1$ for $\frac{4}{\epsilon_0}$ consecutive times. Then have $d-2$ beat $i$, $i+1\ldots 2$ then have it beat $1, \ldots, d-3$. 
 This implies that  
 $x_{d-2}>\frac{1}{2}(x_1+\ldots x_{i-1}+x_{i+1}\ldots+x_{d-3})$. So, $$x_i>\frac 4 {\epsilon_0}(x_{d-2}-(x_1+\dots+x_{i-1})>\frac 4 {\epsilon_0}\frac 1 2 x_{d-2}>\frac 4 {\epsilon_0}\frac 1 4(x_1+\ldots+x_{i-1}+x_{i+1}+\ldots+x_{d-2}),$$ establishing that the hyperplane intersects $B(e_i, {\epsilon_0})$. 
\end{proof}

\subsection{Measure of parallel planes}

Let $\mathcal{D}$ be a compact, convex set in $\mathbb{R}^k$. Let $P_0$ be a plane and $\mathcal{P}$ be the set of planes parallel to $P_0$ that intersect $\mathcal{D}$. The orthocomplement of $P_0$ is a copy of $\mathbb{R}^{k-1}$ and has Lebesgue measure. We identify $\mathcal{P}$ with a (convex) subset of $E\subset \mathbb{R}^k$, by identifying a point in $\mathcal{P}$ with the point in the orthocomplement it intersects.  This induces a measure $\nu$ on $\mathcal{P}$. Let $$A=\max\{\lambda_2(P \cap \mathcal{D}):P \in \mathcal{P}\}.$$ 
\begin{prop}
\label{prop:concave}
There exists a constant $C$ such that for  all $\epsilon>0$, $$\nu(\{P\in \mathcal{P}:\lambda_2(P \cap \mathcal{D})<\epsilon A\})\leq C\sqrt{\epsilon}\nu(\mathcal{P}).$$  
\end{prop}

We first prove
\begin{lem} 
\label{lem:concave}
Let $f:E\to [0,\infty)$ by $f(e)=\lambda_2(P_e\cap \mathcal{D})$ where $p_e$ is the element of $\mathcal{P}$ that contains $e$.   Then on each line, $f$ can be represented as  $u+v$ where $u$ is  a concave function and $v$ is the square of a  concave function.
\end{lem} 
\begin{proof} Let $e_0,e_1 \in E$, identify $e_0$ with $0$, $e_1$ with $1$ and $(1-t)e_0+te_1$ with $t$.
Consider a direction $w$ in $P_0$ and parametrize the lines parallel to $w$ in each $P \in \mathcal{P}$ by the orthogonal direction to $w$ in $P_0$. 
So on each $P$ we have a family of lines  $\ell_x$  parallel to $v$,   for $x\in \mathbb{R}$. 
  Let $h(t,x)$ be the length of the line $\ell_x$ on $P_{(1-t)e_0+te_1}$ intersected with $\mathcal{D}$. 
We assume our parametrization of orthogonal direction to $w$ is chosen so that $$\inf\{x:h(1,x)>0\}=\inf\{x:h(0,x)>0\}$$ and that these infimum are $0$.   
Let us assume that $$b=\sup\{x:h(1,x)>0\}\geq \sup\{x:h(0,x)>0\}=c.$$
Now since the simplex is convex,  for all $x$ with $h(0,x)>0$ and  for all $0\leq t\leq 1$ $$h(t,x)\geq (1-t)h(0,x)+th(1,x).$$ So we have that $$u(t)=\int_0^ch(t,x)dx$$ is a concave function. 
Now we wish to show that if $$v(t)=\int_c^bh(x,t)dx$$ then
$\sqrt v$ is concave.  To do that it suffices to show that for $0\leq s\leq 1$ 
\begin{equation}\label{eq:concave suff} v(s)\geq s^2 v(1)
\end{equation}
 (because $v(0)=0$). Let $\ell_x(0)$ be the line $\ell_x$ on $P_{e_0}$ and $\ell_x(1)$ be the line $\ell_x$ on $P_{e_1}$. Let $\mathcal{E}$ be the convex hull of 
 $\cup_{x\in[c,b]}\ell_x(1) \cap \mathcal{D}$ and any  point in $\ell_c(0)\cap \mathcal{D}$. Since $\mathcal{D}$ is convex, $\mathcal{E}\subset \mathcal{D}$ and so 
\begin{equation}\label{eq:concave 1} v(s)\geq \lambda_2(P_{(1-s)e_0+se_1}\cap \mathcal{E}).
\end{equation} Now $\mathcal{E}$  is a convex cone so 
\begin{equation}\label{eq:concave 2}\lambda_2(P_{(1-a)e_0+se_1}\cap \mathcal{E})=s^2\lambda_2(\mathcal{E}\cap P_{e_1}).
\end{equation} 
Combining \eqref{eq:concave 1} and \eqref{eq:concave 2} verifies the sufficient condition \eqref{eq:concave suff}, completing the proof.
\end{proof}
\begin{lem}
\label{lem:concave:bound}Let $\Omega\subset \mathbb{R}^k$ be convex and compact. If $g:\Omega\to [0,\infty)$ is concave then for all $\epsilon>0$,  $\lambda_k(\{x:g(x)<\epsilon \max g\})<\epsilon\lambda_k(\Omega)$. 
\end{lem}
\begin{proof} Let $p$ be a point that maximizes $g$.  For any $q\in \Omega$, let $\ell(q,p)$ be the line connecting $q$ and $p$. Let $\phi:[0,r]$ be a unit speed parametrization of $\ell(q,p)$.
  If $g(q)<\epsilon g(p)$ then since $g$ convex,  $q \in [0,\epsilon r)$. The lemma follows.  
\end{proof}
We now prove 
 Proposition~\ref{prop:concave}.
 \begin{proof}

We fix a point $p \in \mathcal{D}$ with $\lambda_2(P_p\cap \mathcal{D})=A$, where $P_p$ is the plane in $\mathcal{P}$ going through $p$. We compute $$\nu(\{P\in \mathcal{P}:\lambda_2(P \cap \mathcal{D})<\epsilon A\})$$ via integrating with respect to polar coordinates. Indeed, $W$ be the $d-2$ dimensional subspace of $\mathbb{R}^{d}$ containing the directions orthogonal to the directions in the planes of $\mathcal{P}$.  Let $S^{d-3}$ denote the unit sphere in $W$, $\ell_{\theta}$ denote the line in direction $\theta$ for each $\theta\in S^{d-3}$. Let $\hat{\phi}$ be Lebesgue measure on $S^{d-3}$. Let $P_q$ denote the plane in $\mathcal{P}$ through $q$. So we consider
$$\int_{S^{d-3}}\lambda_1(\{q:\lambda_2(P_q \cap \mathcal{D})<\epsilon A \text{ and }q\in \ell_\theta\})d\theta.$$
 On $\ell_\theta$,  $\lambda_2(P_q\cap \mathcal{D})=u(q)+v(q)$, the sum of a concave function   of $q$ and the square of a concave function.  Applying Lemma \ref{lem:concave:bound} to each summand we have Proposition \ref{prop:concave} for the $1$ dimensional convex set $\mathcal{D}\cap (\cup_{q\in\ell_{\theta}})$. 
Integrating over the $S^{d-1}$ we obtain the proposition.  
\color{black}
\end{proof}

\end{document}